\documentclass[11pt]{article}
\usepackage{amsmath,amsthm,amssymb}
\usepackage[usenames,dvipsnames]{xcolor}
\usepackage{enumerate}
\usepackage{graphicx}
\usepackage{cite}
\usepackage{comment}
\usepackage{oands}
\usepackage{tikz}
\usepackage{changepage}
\usepackage{bbm}
\usepackage{mathtools}
\usepackage[margin=1in]{geometry}

\usepackage[pagewise,mathlines]{lineno}
\usepackage{appendix}
\usepackage{stmaryrd} 
\usepackage{multicol}
\usepackage{microtype}
\usepackage[colorinlistoftodos,textsize=footnotesize]{todonotes}

\definecolor{dgreen}{RGB}{0,180,0}

\usepackage[pdftitle={An invariance principle for ergodic scale-free random environments},
  pdfauthor={Ewain Gwynne, Jason Miller, and Scott Sheffield},
colorlinks=true,linkcolor=NavyBlue,urlcolor=RoyalBlue,citecolor=PineGreen,bookmarks=true,bookmarksopen=true,bookmarksopenlevel=2,unicode=true,linktocpage]{hyperref}

\theoremstyle{plain}
\newtheorem{thm}{Theorem}[section]
\newtheorem{cor}[thm]{Corollary}
\newtheorem{lem}[thm]{Lemma}
\newtheorem{prop}[thm]{Proposition}

\newtheorem{notation}[thm]{Notation}

\def\@rst #1 #2other{#1}
\newcommand\MR[1]{\relax\ifhmode\unskip\spacefactor3000 \space\fi
  \MRhref{\expandafter\@rst #1 other}{#1}}
\newcommand{\MRhref}[2]{\href{http://www.ams.org/mathscinet-getitem?mr=#1}{MR#2}}

\theoremstyle{definition}
\newtheorem{defn}[thm]{Definition}
\newtheorem{remark}[thm]{Remark}

\newtheorem{example}[thm]{Example}

\numberwithin{equation}{section}

\newcommand{\dsb}{\begin{adjustwidth}{2.5em}{0pt}
\begin{footnotesize}}
\newcommand{\dse}{\end{footnotesize}
\end{adjustwidth}}

\newcommand{\ssb}{\begin{adjustwidth}{2.5em}{0pt}}
\newcommand{\sse}{\end{adjustwidth}}

\newcommand{\aryb}{\begin{eqnarray*}}
\newcommand{\arye}{\end{eqnarray*}}
\def\alb#1\ale{\begin{align*}#1\end{align*}}
\def\allb#1\alle{\begin{align}#1\end{align}}
\newcommand{\eqb}{\begin{equation}}
\newcommand{\eqe}{\end{equation}}
\newcommand{\eqbn}{\begin{equation*}}
\newcommand{\eqen}{\end{equation*}}

\newcommand{\BB}{\mathbbm}
\newcommand{\ol}{\overline}
\newcommand{\ul}{\underline}
\newcommand{\op}{\operatorname}
\newcommand{\la}{\langle}
\newcommand{\ra}{\rangle}

\newcommand{\im}{\operatorname{Im}}
\newcommand{\re}{\operatorname{Re}}
\newcommand{\frk}{\mathfrak}
\newcommand{\eqD}{\overset{d}{=}}
\newcommand{\ep}{\epsilon}
\newcommand{\rta}{\rightarrow}

\newcommand{\wt}{\widetilde}
\newcommand{\wh}{\widehat} 
\newcommand{\mcl}{\mathcal}

\newcommand{\bdy}{\partial}
\newcommand{\rng}{\mathring}

\newcommand{\el}{l}

\newcommand{\indshift}{\theta}

\newcommand{\diam}{\mathrm{diam}}
\newcommand{\area}{\mathrm{Area}}

\let\originalleft\left
\let\originalright\right
\renewcommand{\left}{\mathopen{}\mathclose\bgroup\originalleft}
\renewcommand{\right}{\aftergroup\egroup\originalright}

\title{An invariance principle for ergodic scale-free random environments}

\date{  }
\author{
\begin{tabular}{c} Ewain Gwynne\\[-5pt]\small MIT \end{tabular}
\begin{tabular}{c} Jason Miller\\[-5pt]\small Cambridge \end{tabular}
\begin{tabular}{c} Scott Sheffield\\[-5pt]\small MIT \end{tabular}
}

\begin{document}

\maketitle

\begin{abstract}
There are many classical random walk in random environment results that apply to ergodic random planar environments. We extend some of these results to random environments in which the length scale varies from place to place, so that the law of the environment is in a certain sense only translation invariant {\em modulo scaling}. For our purposes, an ``environment'' consists of an infinite random planar map embedded in $\mathbb C$, each of whose edges comes with a positive real conductance. Our main result is that under modest constraints (translation invariance modulo scaling together with the finiteness of a type of specific energy) a random walk in this kind of environment converges to Brownian motion modulo time parameterization in the quenched sense. 

Environments of the type considered here arise naturally in the study of random planar maps and Liouville quantum gravity. In fact, the results of this paper are used in separate works to prove that certain random planar maps (embedded in the plane via the so-called Tutte embedding) have scaling limits given by SLE-decorated Liouville quantum gravity, and also to provide a more explicit construction of Brownian motion on the Brownian map.  However, the results of this paper are much more general and can be read independently of that program. 

One general consequence of our main result is that if a translation invariant (modulo scaling) random embedded planar map and its dual have {\em finite} energy per area, then they are close on large scales to a {\em minimal energy} embedding (the harmonic embedding). To establish Brownian motion convergence for an {\em infinite} energy embedding, it suffices to show that one can perturb it to make the energy finite.
\end{abstract}

\tableofcontents

\section{Introduction}
\label{sec-intro}

\subsection{Basic definitions}
\label{sec-overview}

The goal of this paper is to prove that simple random walks on certain random planar lattices have Brownian motion as a scaling limit. We will consider random planar lattices whose laws are ``translation invariant modulo scaling'' in a sense we will define, but not necessarily translation invariant in the usual sense. Along the way, we will also explain how certain familiar tools (such as the ergodic theorem) can be adapted to this setting. In this subsection, we define and discuss three terms: embedded lattice (with conductances), translation invariant modulo scaling, and ergodic modulo scaling. The impatient reader can skim this subsection quickly on a first read (referring back later as needed) and proceed to the results overview in Section~\ref{sec-resultoverview}.

\begin{defn} \label{def-embedded-map}
An \emph{embedded lattice} $\mcl M$ is an infinite, locally finite, planar, undirected graph (multiple edges and self-loops allowed) embedded into $\BB C$ in such a way that each edge is a simple (unparametrized) curve with zero Lebesgue measure, the edges intersect only at their endpoints, each compact subset of $\BB C$ intersects at most finitely many edges of $\mcl M$, and each connected component of $\BB C\setminus \mcl M$ is compact. 
For an embedded lattice $\mcl M$, we write $\mcl E\mcl M$ for the set of edges of $\mcl M$ and $\mcl F \mcl M$ for the set of faces of $\mcl M$, i.e., the set of closures of connected components of $\BB C\setminus \mcl M$. An \emph{embedded lattice with conductances} is an embedded lattice $\mcl M$ together with a function $\frk c = \frk c_{\mcl M} : \mcl E\mcl M \rta (0,\infty)$. By a slight abuse of notation we write $\mcl M$ instead of $(\mcl M,\frk c)$. We write $\BB M$ for the space of all embedded lattices with conductances. 

A \emph{finite embedded lattice} is defined in the same manner as above, except that the graph is required to be finite and $\BB C$ is replaced by $\ol U$, where $U$ is a bounded sub-domain of $\BB C$.
\end{defn}

For $\mcl M \in \BB M$ and a set $A\subset\BB C$, we write $\mcl M(A)$ for the embedded lattice consisting of all of the vertices and edges of $\mcl M$ which lie on the boundaries of faces of $\mcl M$ which intersect $A$ (the conductances are the same). 
We endow $\BB M$ with a metric in which $\mcl M'$ and $\mcl M$ are ``close'' when they look very similar on a large finite ball. There are many ways to do this, but we fix the following for concreteness: 
for embedded lattices $\mcl M,\mcl M'$ with conductances $\frk c , \frk c'$, we define their distance by
\eqb \label{eqn-map-metric}
\BB d^{\op{EL}}(\mcl M,\mcl M') := \int_0^\infty e^{-r} \wedge \inf_{f_r} \left\{ \sup_{z\in \BB C} |z - f_r(z)| + \max_{e\in \mcl E\mcl M(B_r(0))} |\frk c(e) - \frk c'(f_r(e)) | \right\}  \,dr
\eqe 
where the infimum is over all homeomorphisms $f_r : \BB C\rta \BB C$ such that $f_r$ takes each vertex (resp.\ edge) of $\mcl M(B_r(0))$ to a vertex (resp.\ edge) of $\mcl M'(B_r(0))$, and $f_r^{-1}$ does the same with $\mcl M$ and $\mcl M'$ reversed. Note that the integrand is equal to $e^{-r}$ for each $r > 0$ such that no such homeomorphism exists.  When we speak of random embedded lattices, we implicitly assume that their laws are defined w.r.t.\ the Borel $\sigma$-algebra associated to the topology defined by \eqref{eqn-map-metric}.  
 
It is a classical problem to determine whether a random walk on a random embedded lattice converges to Brownian motion (or to some other limiting process). 
There is a vast literature on this question, falling under the heading of \emph{random walk in random environment}.  See, e.g.,~\cite{klo-fluctiations,bf-rwre-survey,biskup-rwre-survey,kumagai-rwre-survey} for recent surveys. Within this literature, the term \emph{random conductance model} is used to describe walks which choose an edge to traverse (starting from the current vertex) with probability proportional to the conductance assigned to that edge. These are the types of walks we consider in this paper. 

  Many existing results show the convergence (in the {\em quenched} sense) of random walk to Brownian motion for random embedded lattices that are \emph{translation invariant} (stationary in law with respect to spatial translations), i.e., $\mcl M -z \eqD \mcl M$ for each $z\in \BB C$ or for each $z\in\BB Z^2$, in the discrete case. Examples of such lattices include periodic lattices such as $\BB Z^2$ with conductances assigned in a stationary way; the infinite cluster of supercritical percolation on such a lattice~\cite{berger-biskup-perc-rw}; and various graph structures on stationary point processes in $\BB C$. Much of this work is based on the ``environment seen from the particle" approach initiated in~\cite{kv-rwre,dfgw-rwre}. Brownian motion convergence results for random walks in random environments are often called {\em invariance principles}. 
  
In this paper, we will be interested in random embedded lattices whose laws are in some sense only stationary with respect to translations \emph{modulo spatial scaling}, so that there is no universal length scale between different points. To get some intuition about what this means, suppose the embedded lattice $\mcl M$ is ``fractal" in nature (like the embedded lattice illustrated in Figure~\ref{fig-embedding-sim}, right), so that there are some regions in the plane where the vertices of $\mcl M$ are much more concentrated than others. To quantify this effect, we can consider a point $z$ sampled uniformly from Lebesgue measure on a large ball $B_r(0)$, independently from $\mcl M$, and look at the law of the diameter of the face of $\mcl M$ which contains $z$. We are interested in the case when the law of this diameter does not have a weak limit as $r\rta\infty$, so that the possible face sizes become more and more spread out at larger scales. In other words, there is no ``typical" size for a face.

\begin{figure}[ht!]
 \begin{center} 
\includegraphics[scale=.65]{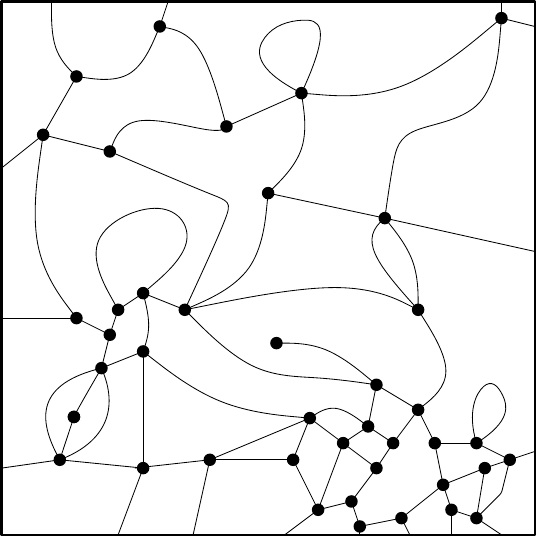} \hspace{15pt}
\includegraphics[scale=.045]{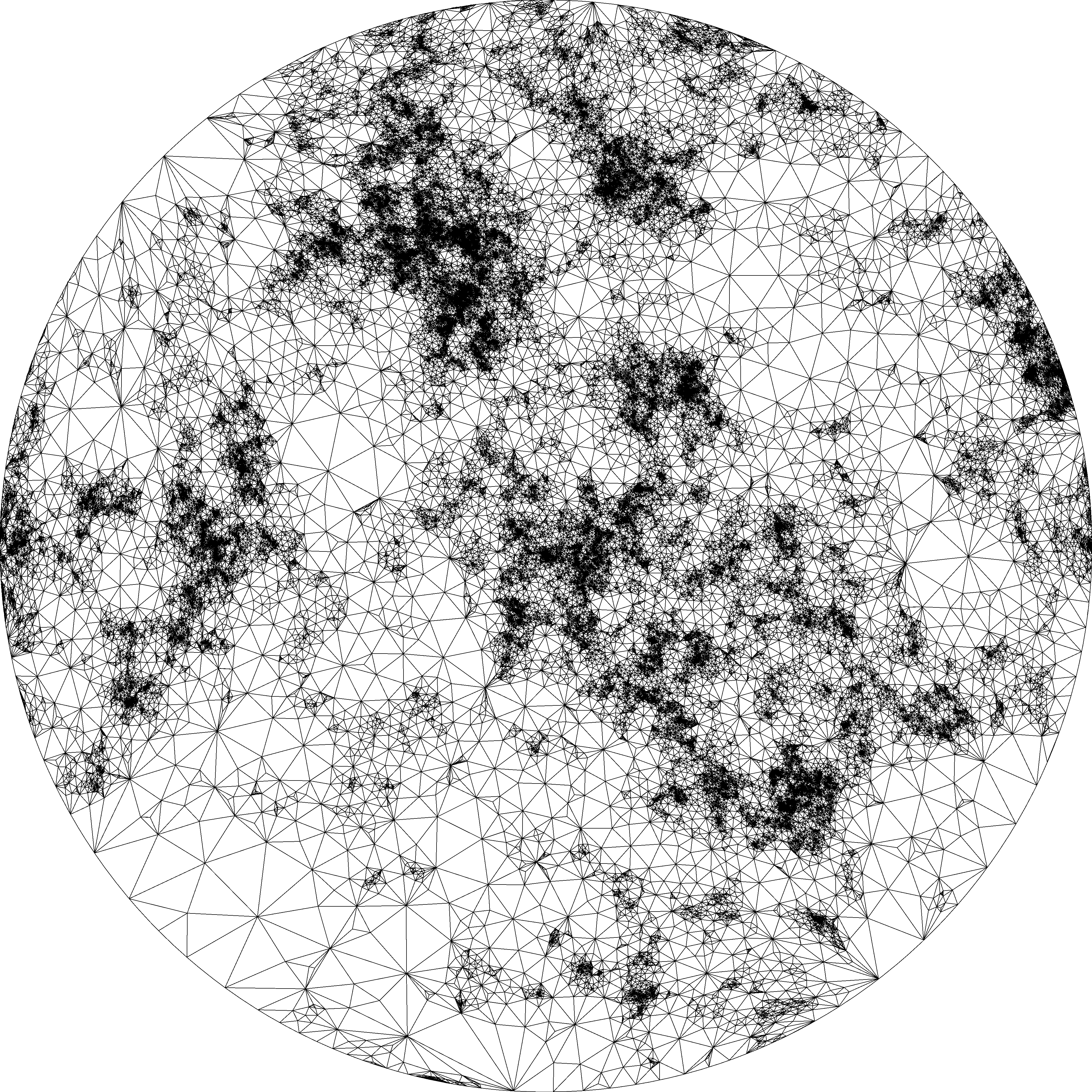}
\vspace{-0.01\textheight}
\caption{\textbf{Left:} The subgraph of an embedded lattice which intersects a square in $\BB C$. \textbf{Right:} A numerical simulation of random planar map (in particular, the $\sqrt2$-mated-CRT map studied in~\cite{gms-tutte}) embedded into $\BB C$ via the Tutte embedding, which means that the two coordinates of the embedding function are discrete harmonic. Note that the faces have very different sizes. 
}\label{fig-embedding-sim}
\end{center}
\vspace{-1em}
\end{figure}

If $\mcl M$ has no typical face size in the above sense, then $\mcl M$ cannot be stationary with respect to spatial translations.
However, it is still possible for such a lattice $\mcl M$ to satisfy a weaker analog of translation invariance, which can be described as ``translation invariance modulo scaling."

Before explaining what this means, let us first define embedded lattices $\mcl M_1$ and $\mcl M_2$ to be {\em equivalent modulo scaling} if $\mcl M_1 = C \mcl M_2$ for some real $C > 0$, where $C \mcl M$ denotes the image of $\mcl M$ under the map $z \to C z$ (the conductances are unchanged, i.e., $\frk c_{C\mcl M}(Ce) = \frk c_{\mcl M}(e)$). Note that if $\mcl M_1$ and $\mcl M_2$ are equivalent modulo scaling, then the constant $C$ is determined by $\mcl M_1$ and $\mcl M_2$: indeed, $C$ is the ratio of the diameters of the origin-containing faces in the two lattices.

If $\mcl M_1$ and $\mcl M_2$ are both random, we say their {\em laws} agree modulo scaling if they assign the same probabilities to any scale invariant Borel measurable subset of $\BB M$ (or, equivalently, if there is a random $C > 0$ such that $C \mcl M_1$ agrees in law with $\mcl M_2$). We will mostly prove statements about $\mcl M$ that would obviously remain true if its law were replaced by a different-but-equivalent-modulo-scaling law. Note that the law of a random embedded lattice $\mcl M$ is equivalent modulo scaling to a unique law w.r.t.\ which the origin-containing face a.s.\ has area one; thus, for most of this paper, there is no lost generality if one assumes that the law of $\mcl M$ is such that the origin-containing face a.s.\ has area one.\footnote{There are two (subtly different) ways to formulate ``modulo scaling'' analogs of probability measures on the space $\BB M$ of embedded lattices. The first approach is to define the probability measure only for the $\sigma$-algebra of Borel measurable (w.r.t.\ \eqref{eqn-map-metric}) events that are {\em scale invariant} (i.e., that are unions of modulo scaling equivalence classes). Such a $\sigma$-algebra contains no information about scale, so a ``sample'' from such a measure could be interpreted as a random modulo-scaling equivalence class. The second approach is to require that the probability measure be defined for all Borel measurable subsets of $\BB M$, but to prove statements that {\em only depend} on the restriction of the measure to the scale-invariant $\sigma$-algebra. We take the latter approach (which has the minor cosmetic advantage of allowing us to speak of ``random embedded lattices'' rather than ``random equivalence classes of embedded lattices'') but the former would work as well.}

Definition~\ref{def-translation-invariance} below describes what it means for a law to be translation invariant modulo scaling. On a first read, it is enough to internalize {\em any one} of the equivalent conditions in Definition~\ref{def-translation-invariance}. Most of the proofs in this paper will make use of a fifth condition, defined in terms of a so-called \emph{dyadic system} (see Section~\ref{sec-dyadic-system}) which is implied by any of the four conditions listed below. Conceptual motivation for the conditions in Definition~\ref{def-translation-invariance} is given {\em just after} the full statement of Definition~\ref{def-translation-invariance}.

\begin{defn} \label{def-translation-invariance}
A random embedded lattice $\mcl M$, endowed with conductance function $\frk c$, is {\bf translation invariant modulo scaling} if a.s.\ all of its faces are compact and it satisfies one of the following equivalent conditions.
\begin{enumerate}
\item {\bf Attainability as Lebesgue-centered infinite volume limit (modulo scaling):} There is a sequence $\{\mcl M_n\}_{n\in\BB N}$ of random finite embedded lattices such that the following is true.
\begin{enumerate}  
\item The union of the closures of the bounded faces of each $\mcl M_n$ is either a closed disk or a closed square (as in Figure~\ref{fig-embedding-sim}). 
\item For each $n$, conditional on $\mcl M_n$ let $z_n$ be sampled uniformly from Lebesgue measure on the union of faces in $\mcl M_n$. There is a random $C_n >0$ (depending on $z_n$ and $\mcl M_n$) such that $C_n(\mcl M_n - z_n)$ converges in law to $\mcl M$ as $n \to \infty$ with respect to the metric~\eqref{eqn-map-metric}.
\end{enumerate}
\label{item-inf-volume}
\item {\bf Attainability as spatial average of sample instance (modulo scaling):} There is a way of associating with $\mcl M$ a (possibly random) sequence $U_n$ of origin-containing sets, each of which is a square or a disk, such that $\bigcup_{n=1}^\infty U_n = \BB C$ a.s.,  and the following is true: suppose that, given $\mcl M$ and $U_n$, a point $z_n$ is sampled from Lebesgue measure on $U_n$.  Then there is a random $C_n >0$ (depending on $z_n$, $U_n$, and $\mcl M$) such that the random embedded lattices $C_n (\mcl M  -z_n)$ converge in law to $\mcl M $ as $n \to \infty$.
\label{item-averaging}
\item {\bf Invariance under repositioning origin within an origin-containing ``block'':} Suppose we associate with each embedded lattice $\mcl M$ a random partition $\mcl B(\mcl M)$ of the plane into a countable collection of {\em blocks}, where each block is a closed set with finite diameter and positive area, and each $z \in \BB C$ a.s.\ belongs to exactly one block\footnote{One can define a topology, hence a $\sigma$-algebra, on block decompositions in a similar manner to~\eqref{eqn-map-metric}, but with the $f_r$'s required to preserve blocks instead of vertices and edges. See also~\eqref{eqn-cell-metric}. Also note that a block decomposition differs from a cell configuration (Definition~\ref{def-cell-config}) in that blocks are not required to be connected.} 
(we allow $\mcl B(\mcl M)$ to depend on $\mcl M$ and possibly on additional randomness; see examples after this definition below).
Suppose further that the procedure for constructing $\mcl B$ from $\mcl M$ commutes with translations and dilations in the sense that 
\[
(C(\mcl M-z) , C(\mcl B(\mcl M) -z) )\eqD (C(\mcl M-z) , \mcl B(C(\mcl M-z) )) 
\]
for each $C>0$ and $z\in\BB C$. For each possible choice of $\mcl B(\mcl M)$, the following is true. Conditional on $\mcl M$ and $\mcl B(\mcl M)$, let $z$ be sampled uniformly from Lebesgue measure on the origin-containing block. Then for some random $C>0$ (depending on $\mcl M$ and $z$), the embedded lattice $C(\mcl M - z)$ agrees in law with $\mcl M$. 
\label{item-block}
\item {\bf Mass transport:} Consider the set of non-negative measurable functions $F : \BB M \times \BB C^2 \rta [0,\infty)$ on the space of embedded lattices with two marked points with the property that $F$ is covariant with respect to dilations and translations of the plane in the sense that 
\eqb \label{eqn-mass-transport-commutation}
F\left(C (\mcl M-z) , C(w_0-z) , C(w_1-z) \right) = C^{-2} F\left(  \mcl M, w_0 , w_1 \right) ,\quad \forall C> 0, \quad \forall z\in\BB C .
\eqe
For every such function $F$,
\eqb \label{eqn-mass-transport} 
\BB E\left[ \int_{\BB C} F(\mcl M, 0, w)\, dw\right] = \BB E\left[ \int_{\BB C} F(\mcl M, w, 0)\, dw \right].
\eqe  
\label{item-mass-transport}
\end{enumerate}
\end{defn}

It is not hard to show that the four conditions listed in Definition~\ref{def-translation-invariance} are equivalent for embedded lattices with compact faces. We will prove that this is the case in Appendix~\ref{sec-equivalence} (see also Lemma~\ref{lem-dyadic-resample}).

The infinite volume limit condition~\ref{item-inf-volume} of Definition~\ref{def-translation-invariance} is natural from the point of view of random planar maps: if $\{\mcl M_n\}_{n\in\BB N}$ is any sequence of random planar maps embedded in the unit disk in some way (e.g., via a circle packing procedure), then we may choose $z$ uniformly on the unit disk and consider the random embedded lattices $\mcl M_n - z$.  If $\mcl M_n - z$ (or some subsequence) converges in law modulo scaling to an unbounded embedded lattice $\mcl M$ then $\mcl M$ satisfies translation invariance modulo scaling.

The spatial averaging condition~\ref{item-averaging} says that although sizes of the faces of $\mcl M$ near two different points of $\BB C$ can differ dramatically, the shapes of these faces and their adjacency structure are such that the local behavior near zero (modulo scaling) looks like the average behavior over a sufficiently large set.  Note that when we speak of the law of $C_n(\mcl M - z_n)$, we mean the overall law not the conditional law given $\mcl M$.  (In other settings, analogous statements involving conditional laws are sometimes used to define ergodicity; but for the moment we are only defining translation invariance modulo scaling, not ergodicity modulo scaling.)
 
For a simple example of a block decomposition satisfying the condition~\ref{item-block} given in Definition~\ref{def-translation-invariance}, one may let the blocks of $\mcl M$ be the faces of $\mcl M$. Another option is to ``mark'' each vertex of $\mcl M$ independently with some small probability $p$, and let the blocks be the cells of the Voronoi tessellation corresponding to the marked vertices. Section~\ref{sec-dyadic-system} will introduce a particularly convenient block decomposition defined in terms of a so-called \emph{dyadic system} of squares, which we will use throughout most of our proofs 

The mass transport condition~\ref{item-mass-transport} is a continuum analog of the discrete mass transport principle used to define the concept of {\em unimodularity} for a random rooted graph~\cite{aldous-lyons-unimodular}, which says that the root vertex is {\em some} sense ``typical'' w.r.t.\ counting measure on vertices. The analogous statement in our story is that the origin-containing face is ``typical'' w.r.t.\ the measure that assigns each face its Lebesgue measure.  We interpret $F$ itself as a ``mass transport'' rule in which $\int_A \int_B F(\mcl M, w, z) dzdw$ is the amount of mass ``transported'' from the set $A$ to the set $B$. The $C^{-2}$ in~\eqref{eqn-mass-transport-commutation} ensures that neither integral in~\eqref{eqn-mass-transport} changes if we replace $\mcl M$ by $C\mcl M$. It also implies that for $C>0$ we have 
\eqbn
\int_{CA} \int_{CB} F(C \mcl M, w, z) dzdw = C^2\int_A \int_B F(\mcl M, w, z) dzdw,
\eqen
 so that in a sense the total mass transported scales like area. For a concrete example, suppose that every face $H$ of $\mcl M$ is assigned a {\em target} face $H'$, and let $F(\mcl M, w, z) = 1_{z \in H_w'}/\area(H_w')$ where $H_w$ is the face containing $w$. Then $F$ describes a transport rule that sends an $\area(H)$ amount of mass from $H$ to $H'$. Informally, the transport rule takes all of the Lebesgue measure contained in $H$ and spreads it uniformly over $H'$. And \eqref{eqn-mass-transport} states that the expected amount of mass flowing into (an infinitesimal neighborhood of) the origin is equal to the expected amount flowing out.
 
The following is one of two main conditions we need to impose to show that the simple random walk on $\mcl M$ converges to Brownian motion with a deterministic covariance matrix. 

\begin{defn} \label{def-ergodic}
We say the law of $\mcl M$ is {\em ergodic modulo scaling} if it is translation invariant modulo scaling and it assigns probability $0$ or $1$ to every Borel measurable event which is invariant under translations/dilations of the form $\mcl M \mapsto C (\mcl M - z)$ for $C>0$ and $z\in\BB C$.
\end{defn}

We remark that if $\mcl I$ is the $\sigma$-algebra of events invariant under translations/dilations of the form $\mcl M \mapsto C (\mcl M - z)$, and the law of $\mcl M$ is translation invariant modulo scaling, then the regular conditional law of $\mcl M$ given $\mcl I$ (which exists because $\mathbb M$ is a separable metric space) is a.s.\ ergodic modulo scaling. In other words, the law of $\mcl M$ is a weighted average of laws that are ergodic modulo scaling.

The other main condition involved in our theorem statements is a bound on some notion of ``expected Dirichlet energy per unit area". 
Let us now briefly explain what this means.

\begin{defn} \label{def-discrete-dirichlet}
For a graph $G$ with a conductance function $\frk c : \mcl E(G) \rta (0,\infty)$ and a function $f : \mcl V(G) \rta \BB C$, we define its \emph{Dirichlet energy} to be the sum over unoriented edges
\eqbn
\op{Energy}(f; G) := \sum_{\{x,y\} \in \mcl E(G)} \frk c(x,y) |f(x) - f(y)|^2  . 
\eqen 
\end{defn}

If the law of a random embedded lattice $\mcl M$ is translation invariant modulo scaling, then the {\em specific Dirichlet energy} is defined to be (roughly speaking) the expected amount of Dirichlet energy per area for the function which takes each vertex of $\mcl M$ to its position in $\BB C$.  Informally, each embedded edge has an energy (equal to its conductance times the square of its diameter) and the specific Dirichlet energy is the expected energy per area of $\mcl M$ in a region containing the origin.

In fact, there are various ways to make this definition precise (depending how one chooses the origin-containing region and how each edge's energy is ``localized'').  Appendix~\ref{sec-specific} formalizes these definitions and uses mass transport to show that they are all equivalent. One of these definitions (the first precise definition of specific Dirichlet energy we present) appears as \eqref{eqn-dual-hyp-moment} within the statement of Theorem~\ref{thm-dual-clt} (where the main theorem hypothesis is the finiteness of the specific Dirichlet energy of both $\mcl M$ and its dual). Several other theorems in this paper require the finiteness of the expectation of different but closely related quantites (which provide upper bounds on the ``Dirichlet energy per unit area"; see, e.g., condition~\ref{item-hyp-moment} in Section~\ref{sec-main-results}).  These alternative {\em finite expectation} conditions will be explained within the theorem statements themselves.

\subsection{Overview of results} 
\label{sec-resultoverview}
    
Sections~\ref{sec-main-results} and~\ref{sec-variants} present our main theorem statements, all of which are variants of the following: if a random embedded lattice $\mcl M$ is ergodic modulo scaling and satisfies an appropriate ``finite specific Dirichlet energy'' condition, then simple random walk on $\mcl M$ scales to Brownian motion modulo time parameterization in the \emph{quenched} sense.  That is, given $\mcl M$, it is a.s.\ the case that simple random walk on $\epsilon \mcl M$ (appropriately time changed) converges as $\epsilon \to 0$ to Brownian motion with some deterministic diffusion matrix.

We will prove variants of this statement for random walk on the faces (Theorem~\ref{thm-general-clt0}) and vertices (Theorems~\ref{thm-graph-clt}, \ref{thm-dual-clt}, and~\ref{thm-general-clt}) of $\mcl M$ with slightly different energy conditions. In Theorem~\ref{thm-dual-clt} the condition is that both $\mcl M$ and a simultaneously embedded dual lattice $\mcl M^*$ satisfy a ``finite specific Dirichlet energy" condition. In this case, roughly speaking, the constraint on $\mcl M$ gives an upper bound on the asymptotically homogenized {\em conductance}, while the constraint on $\mcl M^*$ gives an upper bound on the asymptotically homogenized {\em resistance} (i.e., the inverse of the conductance). Taken together, these bounds imply that the average effective conductance or resistance in any direction is strictly between zero and infinity. 
Theorem~\ref{thm-general-clt} is a generalization of Theorem~\ref{thm-general-clt0} that allows for some degree of non-planarity. This is the version cited in the companion paper~\cite{gms-tutte}, as well as in \cite{gms-poisson-voronoi}. It is not the most general non-planar theorem that we expect to be true, but it is at least reasonably easy to state.

Theorems~\ref{thm-general-clt0} and~\ref{thm-graph-clt} can be derived\footnote{Theorems~\ref{thm-general-clt0} and~\ref{thm-graph-clt} are derived from Theorem~\ref{thm-general-clt} in Section~\ref{sec-variant-proof}; they can alternatively be derived from Theorem~\ref{thm-dual-clt} by imagining one chooses the ``worst possible'' choice for $\mcl M^*$ given $\mcl M$.  }
as consequences of either Theorem~\ref{thm-dual-clt} or Theorem~\ref{thm-general-clt}, but it is useful to state them separately since they can be stated without introducing any extra objects besides the embedded lattice itself, so the statements are slightly simpler.

The main novelty of this paper is that we are able to treat environments that are translation invariant modulo scaling, not translation invariant in the usual sense. On the other hand, our main theorems have non-trivial content even for translation invariant random conductance-weighted subgraphs of $\mathbb Z^2$. In particular, our results include as special cases many existing theorems for random walk in {\em stationary} planar random environments (see examples in 
Sections~\ref{sec-main-results} and~\ref{sec-variants} below) which is not too surprising given that we borrow some of our techniques from these earlier papers.

Environments that are only translation invariant modulo scaling (not stationary in the conventional sense) arise in the study of random planar maps. In particular, the results of this paper are a key input in the companion paper~\cite{gms-tutte} which gives the first rigorous proof that certain embedded random planar maps converge to Liouville quantum gravity, as well as the more recent paper~\cite{gms-poisson-voronoi} which shows that, modulo time parameterization, Brownian motion on the Brownian map is a scaling limit of simple random walks on discretizations of the Brownian map. See Section~\ref{sec-lqg-application} for further discussion of these points. 

Finally, we remark that there are many other works which study random walks in ``fractal" environments: see~\cite{kumagai-rwre-survey,barlow-rwre-survey} for surveys. Recent examples of work include papers by Murugan~\cite{murugan-heat-kernel}, and work by Biskup. Ding, and Goswami~\cite{bdg-lqg-rw} on random walk on $\BB Z^2$ with conductances determined by the exponential of the discrete Gaussian free field.  The setting of the latter work differs from that of the present paper in that the \emph{conductances} (rather than the lattice itself) are what fails to be translation invariant. We do not expect that the random walk in the setting of~\cite{bdg-lqg-rw} converges to Brownian motion modulo time parameterization.

\bigskip

\noindent\textbf{Acknowledgements.} 
We thank an anonymous referee for helpful comments on an earlier version of this article.
We thank the Mathematical Research Institute of Oberwolfach for its hospitality during a workshop where part of this work was completed.  E.G.\ was partially funded by NSF grant DMS-1209044.  S.S.\ was partially supported by NSF grants DMS-1712862 and DMS-1209044 and a Simons Fellowship with award number 306120. We thank Marek Biskup, Jean-Dominique Deuschel, Jian Ding, and Tom Hutchcroft for helpful conversations.

\subsection{Random walk on the faces of an embedded lattice}
\label{sec-main-results}

We will first state and discuss a version of our main result for random walk on the faces of an embedded lattice $\mcl M$, or equivalently, random walk on the dual planar map $\mcl M^*$ of $\mcl M$. One convenient aspect of this approach is that Lebesgue-a.e.\ point of $\BB C$ corresponds to a unique face (namely, the one containing it) and thus $\mcl M$ comes with a distinguished origin-containing face (which is a natural place to start the walk).  Theorems~\ref{thm-graph-clt} and~\ref{thm-dual-clt} below are variants involving random walk on vertices. 

For $z\in \BB C$, we write $H_z$ for the face of $\mcl M$ containing $z$, chosen in some arbitrary manner if $z$ lies on one of the edges of $\mcl M$ (the union of these edges has zero Lebesgue measure). 
For two faces $H,H' \in \mcl F\mcl M$, we write $H\sim H'$ if $H$ and $H'$ share an edge, in which case we write $\frk c(H,H') $ for the conductance of this edge (or the sum of the conductances of the edges, if there is more than one). 

For $H\in\mcl F\mcl M$, define the primal and dual stationary measures by
\eqb \label{eqn-stationary-measure}
\pi(H) := \sum_{\substack{H'\in\mcl F\mcl M : H'\sim H}} \frk c(H,H') \quad \op{and} \quad
\pi^*(H) := \sum_{\substack{H'\in\mcl F\mcl M : H'\sim H}} \frac{1}{ \frk c(H,H') }. 
\eqe
Note that in the case of unit conductances, $\pi(H)= \pi^*(H)$ is the degree of $H$. 

We will be interested in the simple random walk on the faces of a random embedded lattice $\mcl M$ with conductances $\frk c$ satisfying the following hypotheses. 
\begin{enumerate} 
\item \textbf{Ergodicicity modulo scaling.} $\mcl M$ is ergodic modulo scaling, i.e., $\mcl M$ satisfies one (equivalently, all) of the conditions of Definition~\ref{def-translation-invariance} as well as the additional condition of Definition~\ref{def-ergodic}. \label{item-hyp-resampling} 
\item \textbf{Finite expectation.} With $H_0$ the face containing 0 and $\pi,\pi^*$ as in~\eqref{eqn-stationary-measure}, we have \label{item-hyp-moment}
\eqb\label{eqn-hyp-moment}
\BB E\left[ \frac{\op{diam}(H_0)^2}{\op{Area}(H_0)}   \pi(H_0)   \right] < \infty 
\quad \op{and} \quad
\BB E\left[ \frac{\op{diam}(H_0)^2}{\op{Area}(H_0)}   \pi^*(H_0)   \right] < \infty .
\eqe  
\end{enumerate}

\begin{thm} \label{thm-general-clt0}
Let $\mcl M$ be a random embedded lattice satisfying the above two hypotheses. 
Let $X : [0,\infty) \rta   \mcl F\mcl M$ be the continuous-time simple random walk on the dual of $\mcl M$ (viewed as an edge-weighted graph with conductances $\frk c(H,H')$) which spends exactly $\op{Area}(H)/\pi(H)$ units of time at each face before jumping to the next. Let $\wh X = \phi\circ X : [0,\infty)\rta \BB C$, where $\phi$ is a function that sends each face $H$ to an (arbitrary) point of $H$. There is a deterministic covariance matrix $\Sigma$ with $\det\Sigma\not=0$, depending on the law of $\mcl M$, such that a.s.\ as $\ep\rta 0$ the conditional law of $t\mapsto \ep \wh X_{t/\ep^2}$ given $\mcl M$ converges to the law of Brownian motion on $\BB C$ started from 0 with covariance matrix $\Sigma$ with respect to the local uniform topology. 
\end{thm}
 
In many examples we are interested in, the law of the embedded lattice $\mcl M$ will satisfy some sort of rotational symmetry, which will allow us to conclude that $\Sigma$ is a positive scalar multiple of the identity matrix. 

The choice of time parameterization in Theorem~\ref{thm-general-clt0} is somewhat arbitrary. One can obtain convergence under other parameterizations --- e.g., the one where the walk spends $\op{diam}(H)^2$ units of time at each face --- using Lemma~\ref{lem-walk-ergodic} below. 
In general, however, we do \emph{not} expect that the random walk on $\mcl M$ parameterized according to counting measure on its steps converges to standard Brownian motion. Indeed, if $\mcl M$ looks like the embedded lattice in the right panel of Figure~\ref{fig-embedding-sim}, then the random walk on $\mcl M$ will travel much faster through some regions than others. If $\mcl M$ is a discretization of Liouville quantum gravity (such as an embedded random planar map of an appropriate type) then we expect that the random walk on $\mcl M$ parameterized by counting measure on steps converges to a re-parameterized variant of Brownian motion called \emph{Liouville Brownian motion}~\cite{berestycki-lbm,grv-lbm}, but this has not been proven.  
 
We will actually prove a generalization of Theorem~\ref{thm-general-clt0} (see Theorem~\ref{thm-general-clt} below) in which $\mcl F\mcl M$ is replaced by a so-called \emph{cell configuration}. A cell configuration is a generalization of the face set of an embedded lattice where the faces are replaced by ``cells" which can be arbitrary compact connected sets, so the adjacency graph of cells is not necessarily planar.
We will also prove (Theorem~\ref{thm-general-clt-uniform}) that the rate of convergence of random walk to Brownian motion in Theorem~\ref{thm-general-clt0} is uniform on compact subsets of $\BB C$. 

The fact that we are able to replace exact translation invariance with ergodicity modulo scaling is the main novelty of our result.  
The finite expectation hypothesis~\ref{item-hyp-moment}  provides an upper bound for the Dirichlet energy per unit area of the given embedding of $\mcl M$.
Indeed, integrating $\frac{\op{diam}(H_z)^2}{\op{Area}(H_z)} \pi(H_z) $ over all $z$ in some domain $D\subset\BB C$ gives an upper bound on the discrete Dirichlet energy of the function on $\mcl M^*$ which sends each face (dual vertex) which intersects $D$ to a point in the corresponding face (see also Lemmas~\ref{lem-cell-sum} and~\ref{lem-square-energy}). As such hypothesis~\ref{item-hyp-moment} tells us that the expected Dirichlet energy of this function is locally finite. We need to include the finite expectation condition with $ \pi^*$ in place of $\pi$ in order to transfer from Dirichlet energy bounds to pointwise bounds (this is done in Lemma~\ref{lem-path-sum}). Since $\pi(H_0) + \pi^*(H_0) $ is at least the degree (number of neighbors) of $H_0$,~\eqref{eqn-hyp-moment} implies that this degree has finite expectation. 

There is one natural hypothesis which is conspicuously missing from the statement of Theorem~\ref{thm-general-clt0}: namely, that $\mcl M$ does not have any macroscopic faces in the sense that the maximal diameter of the faces of $\mcl M$ which intersect $B_r(0)$ grows sublinearly in $r$. It turns out that this is implied by the given hypotheses; see Lemma~\ref{lem-max-cell-diam}.

A function $f : \mcl F\mcl M \rta \BB C$ is called \emph{$\frk c$-discrete harmonic} if for any $H\in \mcl F\mcl M$, 
\eqbn
f(H) = \frac{1}{\pi(H)} \sum_{\substack{H'\in\mcl F\mcl M : H'\sim H}} \frk c(H,H') f(H')
\eqen 
As is common in the random walk in random environment literature (see, e.g.,~\cite{biskup-rwre-survey}), Theorem~\ref{thm-general-clt0} will be proven by constructing a discrete harmonic function $\phi_\infty$ on $\mcl M^*$ which is close to the a priori embedding (which sends each face to a point in or near the face) at large scales, and then using the martingale central limit theorem to show that the image of random walk on $\mcl F\mcl M$ under $\phi_\infty$ converges to Brownian motion.

\begin{thm} \label{thm-general-tutte}
Recall the restricted embedded lattice $\mcl M(B_r(0))$ for $r>0$ defined just above~\eqref{eqn-map-metric}. 
Almost surely, there exists a $\frk c$-discrete harmonic function $\phi_\infty : \mcl F\mcl M \rta \BB C$ such that 
\eqb \label{eqn-general-tutte}
\lim_{r\rta 0} \frac{1}{r} \sup_{H \in \mcl F\mcl M(B_r(0))} |\phi_\infty(H) - \phi_0(H)|= 0,
\eqe 
where $\phi_0(H) := \op{Area}(H)^{-1} \int_H z \,dz$ is the Lebesgue center of mass of $H$. 
\end{thm}

In the RWRE literature, $\phi_\infty-\phi_0$ is often referred to as the \emph{corrector} for random walk on $\mcl F\mcl M$. Theorem~\ref{thm-general-tutte} says that the corrector grows sublinearly with respect to the Euclidean metric.

Another important intermediate step in the proof of Theorem~\ref{thm-general-clt0} is the following statement, which is proven in Section~\ref{sec-recurrence}.  

\begin{thm} \label{thm-recurrent}
Under the same hypotheses as Theorem~\ref{thm-general-clt0}, it is a.s.\ the case that the random walk on the dual map $ \mcl M^*$ with conductances $\frk c$ is recurrent.
\end{thm}

Gurel-Gurevich and Nachmias~\cite{gn-recurrence} showed that the simple random walk is a.s.\ recurrent on any random planar map which is the local limit of finite maps based at a uniform random vertex; and which is such that the law of the degree of the root vertex has an exponential tail. This result was later extended to a wider class of graphs by Lee~\cite[Theorem 1.6]{lee-conformal-growth}. 

The criteria of~\cite{gn-recurrence,lee-conformal-growth} do \emph{not} apply to the graph $\mcl M$ above or to its dual. Indeed, $\mcl M^*$ is not a local limit of finite maps based at a uniform random root vertex, nor is it unimodular (except in trivial cases): rather, by translation invariance modulo scaling, the root face $H_0$ is a typical face from the perspective of Lebesuge measure, not a typical face from the perspective of the counting measure on faces of $\mcl M$. Furthermore, our bound on the degree of $H_0$ from hypothesis~\ref{item-hyp-moment} is weaker than the exponential tail bound required in~\cite{gn-recurrence} as well as the degree bound required in~\cite{lee-conformal-growth}.

To illustrate the utility of Theorem~\ref{thm-general-clt0}, we now give three simple examples of its application to particular random environments. The examples do not use the full force of Theorem~\ref{thm-general-clt0} since environments in these examples are exactly stationary with respect to spatial translations. Applications of Theorem~\ref{thm-general-clt0} to non-stationary random environments are discussed in Section~\ref{sec-lqg-application} and worked out in detail in~\cite{gms-tutte,gms-poisson-voronoi}. 

\begin{example}[Random conductance model on $\BB Z^2$] \label{ex-random-conductance}
Consider a random conductance function $\frk c $ on the nearest-neighbor edges of $\BB Z^2$ which is stationary and ergodic with respect to spatial translations of $\BB Z^2$, 
and such that the conductances and their reciprocals have finite expectation. Theorem~\ref{thm-general-clt0} implies that the random walk on $\BB Z^2$ with these conductances converges in law to Brownian motion in the quenched sense. This convergence was originally proven by Biskup~\cite{biskup-rwre-survey} (see~\cite{ss-random-conductance} for an earlier proof in the case that the conductances are bounded above and below). To deduce the convergence from Theorem~\ref{thm-general-clt0}, let $\mcl M$ be the dual lattice $\BB Z^2 + (1/2,1/2)$. If we sample $w$ uniformly from Lebesgue measure on $[0,1]\times [0,1]$, then the randomly shifted lattice $\mcl M - w$ is ergodic (hence ergodic modulo scaling) and the finite expectation hypothesis is immediate from our assumption on the conductances.
\end{example}

\begin{example}[Dual of a random subgraph of $\BB Z^2$] \label{ex-perc-face}
Let $\mcl M$ be a random subgraph of $\BB Z^2$ which is stationary and ergodic with respect to spatial translations, with unit conductances (e.g., $\mcl M$ could be the infinite cluster of a supercritical bond percolation or Ising model on $\BB Z^2$). If the second moment of the number of edges on the boundary of the face of $\mcl M$ containing 0 is finite, then we can apply Theorem~\ref{thm-general-clt0}, with unit conductances, to $\mcl M$ (shifted by a uniform element of $[0,1]\times [0,1]$ as in Example~\ref{ex-random-conductance}) to find that the simple random walk on the faces of $\mcl M$ converges in law to Brownian motion in the quenched sense. If we assume a slightly stronger moment hypothesis, then Theorem~\ref{thm-graph-clt} below shows that random walk on the vertices of $\mcl M$ likewise converges to Brownian motion (Example~\ref{ex-perc-vertex}). 
\end{example}

\begin{example}[Voronoi tessellation of a homogeneous Poisson point process] \label{ex-ppp}
Let $\mcl P$ be a homogeneous Poisson point process on $\BB C$ (so the law of $\mcl P$ is translation invariant). Let $\mcl M$ be the Voronoi tessellation associated with $\mcl P$, viewed as an embedded lattice with unit conductances whose faces are the Voronoi cells, whose edges are the linear segments on the boundaries of the cells, and whose vertices are the points where these linear segments meet. 
It is straightforward to verify that $\mcl M$ satisfies the hypotheses of Theorem~\ref{thm-general-tutte}, so the simple random walk on the faces of $\mcl M$ converges in law to Brownian motion in the quenched sense.  
\end{example}

\subsection{Variants of Theorem~\ref{thm-general-clt0}}
\label{sec-variants}

Here we state several variants and extensions of Theorem~\ref{thm-general-clt0} with different setups.  
The reader who only wants to see the proof of Theorem~\ref{thm-general-clt0} can skip this subsection. 

To avoid having to specify time parameterizations of the walks, for some of the statements in this subsection we will consider the topology on curves viewed modulo time parameterization, which we now recall. If $\beta_1 : [0,T_{\beta_1}] \rta \BB C$ and $\beta_2 : [0,T_{\beta_2}] \rta \BB C$ are continuous curves defined on possibly different time intervals, we set 
\eqb \label{eqn-cmp-metric}
\BB d^{\op{CMP}} \left( \beta_1,\beta_2 \right) := \inf_{\phi } \sup_{t\in [0,T_{\beta_1} ]} \left| \beta_1(t) - \beta_2(\phi(t)) \right| 
\eqe 
where the infimum is over all increasing homeomorphisms $\phi : [0,T_{\beta_1}]  \rta [0,T_{\beta_2}]$ (the CMP stands for ``curves modulo parameterization"). It is shown in~\cite[Lemma~2.1]{ab-random-curves} that $\BB d^{\op{CMP}}$ induces a complete metric on the set of curves viewed modulo time parameterization. 

In the case of curves defined for infinite time, it is convenient to have a local variant of the metric $\BB d^{\op{CMP}}$. Suppose $\beta_1 : [0,\infty) \rta \BB C$ and $\beta_2 : [0,\infty) \rta \BB C$ are two such curves. For $r > 0$, let $T_{1,r}$ (resp.\ $T_{2,r}$) be the first exit time of $\beta_1$ (resp.\ $\beta_2$) from the ball $B_r(0)$ (or 0 if the curve starts outside $B_r(0)$). 
We define 
\eqb \label{eqn-cmp-metric-loc}
\BB d^{\op{CMP}}_{\op{loc}} \left( \beta_1,\beta_2 \right) := \int_1^\infty e^{-r} \left( 1 \wedge \BB d^{\op{CMP}}\left(\beta_1|_{[0,T_{1,r}]} , \beta_2|_{[0,T_{2,r}]} \right) \right) \, dr ,
\eqe 
so that $\BB d^{\op{CMP}}_{\op{loc}} (\beta^n , \beta) \rta 0$ if and only if for Lebesgue a.e.\ $r > 0$, $\beta^n$ stopped at its first exit time from $B_r(0)$ converges to $\beta$ stopped at its first exit time from $B_r(0)$ with respect to the metric~\eqref{eqn-cmp-metric}. 
We note that the definition~\eqref{eqn-cmp-metric} of $\BB d^{\op{CMP}}\left(\beta_1|_{[0,T_{1,r}]} , \beta_2|_{[0,T_{2,r}]} \right)$ makes sense even if one or both of $T_{1,r}$ or $T_{2,r}$ is infinite, provided we allow $\BB d^{\op{CMP}}\left(\beta_1|_{[0,T_{1,r}]} , \beta_2|_{[0,T_{2,r}]} \right) = \infty$ (this doesn't pose a problem due to the $1\wedge$ in~\eqref{eqn-cmp-metric-loc}). 
 
\subsubsection{Random walk on vertices} 
\label{sec-vertices}

If we assume a slightly stronger variant of hypothesis~\ref{item-hyp-moment} in Theorem~\ref{thm-general-clt0}, we also get quenched convergence of the random walk on the vertices of $\mcl M$ to Brownian motion. 
To formulate this statement, for an embedded lattice $\mcl M$ and $x\in\mcl V\mcl M$, define the \emph{outradius} of $x$ by
\eqb \label{eqn-inrad-outrad} 
\op{Outrad}(x) := \op{diam}\left(  \bigcup_{H \in \mcl F\mcl M : x\in\bdy H} H \right)  ,
\eqe   
i.e., the diameter of the union of the faces with $x$ on their boundaries. 
Analogously to~\eqref{eqn-stationary-measure}, if $\mcl M$ is equipped with a conductance function $\frk c$ we also define 
\eqb \label{eqn-graph-stationary}
\pi(x) := \sum_{\substack{y\in \mcl V\mcl M \\ y\sim x}} \frk c(x,y) \quad\op{and} \quad
\pi^*(x) := \sum_{\substack{y\in \mcl V\mcl M \\ y\sim x}} \frac{1}{\frk c(x,y)} ,
\eqe 
where $x\sim y$ means that the vertices $x$ and $y$ are joined by an edge. 
 
\begin{thm} \label{thm-graph-clt}
Suppose $\mcl M$ is a random embedded lattice which is ergodic modulo scaling and satisfies the following stronger version of hypothesis~\ref{item-hyp-moment} of Theorem~\ref{thm-general-clt0}:
\begin{enumerate}

\item[$2'$.] With $H_0$ the face of $\mcl M$ containing 0, 
\eqbn
\BB E\left[ \sum_{x \in \mcl V\mcl M \cap \bdy H_0}  \frac{  \op{Outrad}(x)^2 \pi(x) }{ \op{Area}(H_0) }   \right] < \infty 
\quad \op{and} \quad
\BB E\left[ \sum_{x \in \mcl V\mcl M\cap \bdy H_0}  \frac{  \op{Outrad}(x)^2 \pi^*(x) }{ \op{Area}(H_0) }   \right] < \infty . \label{item-graph-hyp-moment} 
\eqen 
\end{enumerate}
Almost surely, as $\ep\rta 0$ the conditional law given $\mcl M$ of the simple random walk on $\ep \mcl M$ started from any vertex on the boundary of the origin-containing face and extended from $\BB N_0$ to $[0,\infty)$ by linear interpolation converges in law modulo time parameterization to planar Brownian motion started from 0 with some deterministic, non-degenerate covariance matrix. Furthermore, the simple random walk on $\mcl M$ is a.s.\ recurrent and there is a discrete harmonic function $\phi_\infty : \mcl V\mcl M \rta \BB C$ such that a.s.\ 
\eqbn
\lim_{r\rta\infty} \frac{1}{r} \max_{x\in\mcl V\mcl M\cap B_r(0)} |\phi_\infty(x) - x| =  0 .
\eqen
\end{thm}

We will deduce Theorem~\ref{thm-graph-clt} from Theorem~\ref{thm-general-clt0} by drawing in the dual map $\mcl M^*$ of $\mcl M$ in such a way that the hypotheses of Theorem~\ref{thm-general-clt0} are satisfied for $\mcl M^*$. 
We state the convergence in Theorem~\ref{thm-graph-clt} modulo time parameterization since (unlike for $\mcl M^*$) there is not a canonical way of associating an area with each vertex in order to decide how much time the walk spends there. 

As in the case of Theorem~\ref{thm-general-clt0}, Theorem~\ref{thm-graph-clt} has applications to random walk in random environments on $\BB Z^2$. 

\begin{example}[Random subgraphs of $\BB Z^2$] \label{ex-perc-vertex}
Let $\mcl M$ be a random subgraph of $\BB Z^2$ which is stationary and ergodic with respect to spatial translations, as in Example~\ref{ex-perc-face}. Since each vertex of $\mcl M$ has degree at most 4, Theorem~\ref{thm-graph-clt} (applied to $\mcl M -w$ for $w$ a uniform sample from $[0,1]\times [0,1]$) implies that the random walk on $\mcl M$ converges to Brownian motion modulo time parameterization in the quenched sense provided
\eqb \label{eqn-perc-hypothesis}
\BB E\left[ \left(\sum_{x\in \BB Z^2 \cap \bdy H_0} \op{Outrad}(x) \right)^2 \right]  < \infty ,
\eqe  
where $H_0$ is the face containing $[0,1]\times [0,1]$. 
The bound~\eqref{eqn-perc-hypothesis} is satisfied, e.g., for supercritical percolation on $\BB Z^2$ since the size of the origin-containing cluster in the dual subcritical percolation has an exponential tail~\cite[Theorem 5.4]{grimmett-book}. Other random subgraphs of $\BB Z^2$ satisfying this bound arise in subcritical Ising models, FK models, etc.  Random walk on the infinite cluster of supercritical bond percolation on $\BB Z^d$ was shown to converge to Brownian motion uniformly (not just modulo time parameterization) in~\cite{berger-biskup-perc-rw}; see also~\cite{prs-perc-rw} for a generalization of this result. 
\end{example}

\subsubsection{Embedded primal/dual lattice pair}
\label{sec-dual-clt}

In this subsection we formulate an analog of Theorem~\ref{thm-general-clt0} when we have embeddings of \emph{both} $\mcl M$ and its planar dual $\mcl M^*$. 
The hypotheses of this version of the theorem are in some ways more elegant since the conditions are symmetric in $\mcl M$ and $\mcl M^*$. 
 
If $(\mcl M,\mcl M^*)$ is a pair of embedded lattices, we say that $\mcl M^*$ is \emph{dual} to $\mcl M$ if each face of $\mcl M$ (resp.\ $\mcl M^*$) contains exactly one vertex of $\mcl M^*$ (resp.\ $\mcl M$) and each edge of $\mcl M$ (resp.\ $\mcl M^*$) crosses exactly one edge of $\mcl M^*$ (resp.\ $\mcl M$). 
If $\mcl M$ and $\mcl M^*$ are equipped with edge conductances $\frk c$ and $\frk c^*$, we also require that $\frk c(e) = 1/\frk c^*(e^*)$, for $e^*$ the edge of $\mcl M^*$ which crosses $e$. See Figure~\ref{fig-weird-cell}, left, for an illustration. 
There is an obvious generalization of the notions of translation invariance modulo scaling to the case of a pair of dual embedded lattices: indeed, we just replace all of the statements for re-scaled translated versions of $\mcl M$ in Definition~\ref{def-translation-invariance} by joint statements for re-scaled translated versions of $(\mcl M , \mcl M^*)$ (we require that all translation and scaling factors are the same for both $\mcl M$ and $\mcl M^*$).  The definition of ergodic modulo scaling (Definition~\ref{def-ergodic}) is similarly well defined for the pair $(\mcl M , \mcl M^*)$.

 \begin{thm} \label{thm-dual-clt}
Suppose the pair $(\mcl M ,\mcl M^*)$ (where $\mcl M$ is a.s.\ dual to $\mcl M^*$) is ergodic modulo scaling and that the following finiteness condition holds. If $H_0$ (resp.\ $H_0^*$) is the face of $\mcl M$ (resp.\ $\mcl M^*$) containing the origin and $E^*$ (resp.\ $E $) denotes the set of edges of $\mcl M^*$ (resp.\ $\mcl M$) which cross edges on the boundary of $H_0$ (resp.\ $H_0^*$), then
\eqb \label{eqn-dual-hyp-moment}
\BB E\left[ \sum_{e\in E} \frac{\op{diam}(e)^2 \frk c(e)}{ \op{Area}(H_0^*)}   \right] < \infty  
\quad \op{and} \quad
\BB E\left[  \sum_{e \in E^*} \frac{\op{diam}(e^*)^2 \frk c (e^*) }{\op{Area}(H_0)}  \right] < \infty .
\eqe 
Almost surely, as $\ep\rta 0$ the conditional law given $(\mcl M,\mcl M^*)$ of the simple random walk on $\ep \mcl M$ started from any vertex on the boundary of the origin-containing face and extended from $\BB N_0$ to $[0,\infty)$ by linear interpolation converges in law modulo time parameterization to planar Brownian motion started from 0 with some deterministic, non-degenerate covariance matrix. Furthermore, the simple random walk on $\mcl M$ is a.s.\ recurrent and there is a discrete harmonic function $\phi_\infty : \mcl V\mcl M \rta \BB C$ such that a.s.\ 
\eqbn
\lim_{r\rta\infty} \frac{1}{r} \max_{x\in\mcl V\mcl M\cap B_r(0)} |\phi_\infty(x) - x| =  0 .
\eqen
The same holds with $\mcl M^*$ in place of $\mcl M$. 
\end{thm}

We can take the two quantities in \eqref{eqn-dual-hyp-moment} --- both required to be finite --- as the {\em definition} of the specific Dirichlet energy of $\mcl M$ and $\mcl M^*$.  See Appendix~\ref{sec-specific} for further justification of this definition.

As we will see, Theorem~\ref{thm-dual-clt} follows from essentially the same proof as Theorem~\ref{thm-general-clt0} (or Theorem~\ref{thm-general-clt} below) modulo the following technicality: since the finite expectation hypothesis~\eqref{eqn-dual-hyp-moment} only gives a bound for the expectation of squared edge diameter times conductance, it is much harder to show that there are no macroscopic edges. This is carried out in Appendix~\ref{sec-dual-max-diam} to avoid distracting from the main arguments. We now give an example to illustrate the usage of Theorem~\ref{thm-dual-clt}.

\begin{example}[Random conductance model on $\BB Z^2$ without finite expectation] 
We expect that Theorem~\ref{thm-dual-clt} can be used to prove quenched invariance principles (modulo time parameterization) for some kinds of stationary, ergodic random conductance models on $\BB Z^2$ whose conductances and their reciprocals do not have finite expectation. This can be done by ``shrinking" edges with large conductance and the dual edges (i.e., edges of $\BB Z^2 + (1/2,1/2)$) corresponding to edges with small conductance to get a new embedded lattice/dual lattice pair satisfying the hypotheses of Theorem~\ref{thm-dual-clt}. We do not work out any examples in detail here, but we note that similar ideas are used to deal with i.i.d.\ random conductance models without a finite expectaton hypothesis, in general dimension, in~\cite{abdh-random-conductance,barlow-deuschel-unbounded}. 
\end{example}

\begin{figure}[t!]
 \begin{center}
\includegraphics[scale=.55]{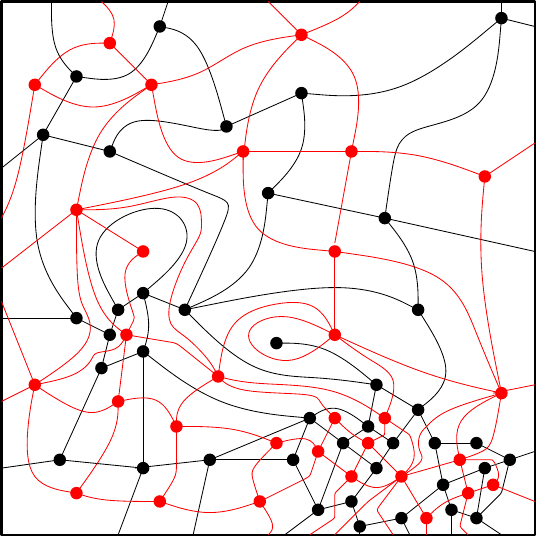} \hspace{15pt}
\includegraphics[scale=.85]{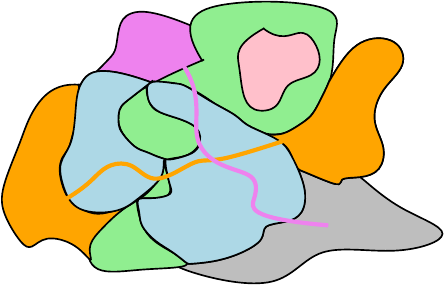}
\vspace{-0.01\textheight}
\caption{\textbf{Left:} An embedded lattice/dual lattice pair. 
\textbf{Right:} A possible arrangement of six cells of a cell configuration. The orange cell consists of two simply connected regions joined by a curve of zero Lebesgue measure. The blue cell is cut in two by the green cell. The pink cell does not share a boundary arc with any cell other than the green one. The purple cell has a ``tentacle" joining it to the grey cell. Note that these sorts of situations show that $\mcl H$, as a graph, need not be planar. 
}\label{fig-weird-cell}
\end{center}
\vspace{-1em}
\end{figure} 

\subsubsection{Random walk on cells}
\label{sec-clt-thm} 

We now state a generalization of Theorem~\ref{thm-general-clt0} (Theorem~\ref{thm-general-clt}) where faces are replaced by general compact connected sets. This allows for graph structures which are not planar. This is the variant which will be used in~\cite{gms-tutte} and the proof is identical to that of Theorem~\ref{thm-general-clt0} (which is a special case). Consequently, the version of the theorem stated in this subsection is the first one which we will prove. 

\begin{defn} \label{def-cell-config}
A \emph{cell configuration} on a domain $D\subset \BB C$ consists of the following objects.
\begin{enumerate}
\item A locally finite collection $\mcl H$ of compact connected subsets of $\ol D$ (``cells") with non-empty interiors whose union is $\ol D$ and such that the intersection of any two elements of $\mcl H$ has zero Lebesgue measure.
\item A symmetric relation $\sim$ on $\mcl H\times \mcl H$ (``adjacency") such that if $H\sim H'$, then $H\cap H'\not=\emptyset$ and $H\not=H'$. 
\item A function $\frk c = \frk c_{\mcl H}$ (``conductance") from the set of pairs $(H,H')$ with $H\sim H'$ to $(0,\infty)$ such that $\frk c(H,H') = \frk c(H',H)$. 
\end{enumerate}
We typically slightly abuse notation by making the relation $\sim$ and the function $\frk c$ implicit, so we write $\mcl H$ instead of $(\mcl H,\sim , \frk c)$. 
\end{defn}

We will almost always consider only the case when $D=\BB C$. When we refer to ``cell configuration" without a specified domain, we mean that $D=\BB C$. The case when $D\not=\BB C$ is only used to formulate the cell configuration analog of Condition~\ref{item-inf-volume} of Definition~\ref{def-translation-invariance}.
We view a cell configuration $\mcl H$ (with arbitrary domain) as a weighted graph whose vertices are the cells of $\mcl H$ and whose edge set is
\eqb \label{eqn-cell-edges}
\mcl E\mcl H := \left\{ \{H,H'\} \in \mcl H\times \mcl H : H\sim H' \right\} ,
\eqe
and such that any edge $\{H,H'\}$ is assigned the weight $\frk c(H,H')$. 
Note that any two cells which are joined by an edge intersect but two intersecting cells need not be joined by an edge. 

A simple example of a cell configuration is the case that $\mcl H $ is the set of faces of an embedded lattice, with two such faces adjacent if and only if they share an edge. 
In this case the resulting cell configuration is isomorphic to $\mcl M^*$. Consequently, one can think of an embedded lattice as a special case of a cell configuration (except that we do not have distinguished vertices and edges for a cell configuration, so we lose some information which does not affect the random walk on the faces). 
However, the cells in a cell configuration need not be simply connected and they are allowed to overlap (so long as the overlap has zero Lebesgue measure), so a cell configuration need not correspond to a planar graph. An example of a cell configuration is illustrated in Figure~\ref{fig-weird-cell}, right. 

Generalizing the above notation for embedded lattices, for a cell configuration $\mcl H$ we define
\eqb \label{eqn-cell-restrict}
 \mcl H(A) := \left\{H\in\mcl F\mcl M : H\cap A \not=\emptyset \right\}   ,\quad \forall A\subset\BB C
\eqe 
and we view $\mcl H(A)$ as an edge-weighted graph with edge set consisting of all of the edges in $\mcl E\mcl H$ joining elements of $\mcl H(A)$ and conductances given by the restriction of $\frk c$. We define a metric on the space of cell configurations via the obvious extension of~\eqref{eqn-map-metric}:
\allb \label{eqn-cell-metric}
\BB d^{\op{CC}}(\mcl H,\mcl H') &:= \int_0^\infty e^{-r} \wedge \inf_{f_r} \big\{ \sup_{z\in \BB C} |z - f_r(z)|   \notag \\
&\qquad \qquad  + \max_{\{H_1,H_2\} \in \mcl E\mcl H(B_r(0))} |\frk c(H_1,H_2) - \frk c'(f_r(H_1) , f_r(H_2) ) | \big\}  \,dr
\alle 
where each of the infima is over all homeomorphisms $f_r : \BB C\rta \BB C$ such that $f_r$ takes each cell in $\mcl H(B_r(0))$ to a cell in $\mcl H'(B_r(0))$ and preserves the adjacency relation, and $f_r^{-1}$ does the same with $\mcl H$ and $\mcl H'$ reversed. We also say that a cell configuration is \emph{translation invariant modulo scaling} if it satisfies any of the equivalent conditions of Definition~\ref{def-translation-invariance}, with $\mcl H$ in place of $\mcl M$ (and cells in place of faces), and \emph{ergodic modulo scaling} if it satisfies the conditions of Definition~\ref{def-ergodic}, with $\mcl H$ in place of $\mcl M$.  

For $z\in \BB C$, we write $H_z$ for one of the cells in $\mcl H$ containing $z$, chosen in some arbitrary manner if there is more than one such cell (the cell is unique for Lebesgue-a.e.\ $z$). Exactly as in~\eqref{eqn-stationary-measure}, we also define
\eqb \label{eqn-stationary-measure-cell}
\pi(H) := \sum_{\substack{H'\in\mcl H : H'\sim H}} \frk c(H,H') \quad \op{and} \quad
\pi^*(H) := \sum_{\substack{H'\in\mcl H : H'\sim H}} \frac{1}{ \frk c(H,H') }. 
\eqe

The following theorem generalizes Theorems~\ref{thm-general-clt0},~\ref{thm-general-tutte}, and~\ref{thm-recurrent}. We need an extra hypothesis to ensure that our cell configuration is sufficiently ``locally connected" (since we do not assume that cells which intersect are joined by an edge). 

\begin{thm} \label{thm-general-clt}
Let $\mcl H$ be a random cell configuration which is ergodic modulo scaling, satisfies the finite expectation hypothesis~\eqref{eqn-hyp-moment} with $\mcl H$ in place of $\mcl F\mcl M$, and satisfies the following additional condition:
\begin{enumerate}
\setcounter{enumi}{2}
\item \textbf{Connectedness along lines.} Almost surely, for each horizontal or vertical  line segment $L \subset \BB C$, the subgraph of $\mcl H$ consisting of the set of cells which intersect $L$ and the edges joining these cells is connected. \label{item-hyp-adjacency}
\end{enumerate} 
Let $X : [0,\infty) \rta \mcl H$ be the continuous-time simple random walk on $\mcl H$ (viewed as an edge-weighted graph as above) which spends $\op{Area}(H)/\pi(H)$ units of time at each cell before jumping to the next. Let $\wh X : [0,\infty)\rta \BB C$ be the process obtained by composing $X$ with a function that sends each cell $H$ to an (arbitrary) point of $H$. There is a deterministic covariance matrix $\Sigma$ with $\det\Sigma\not=0$, depending on the law of $\mcl H$, such that a.s.\ as $\ep\rta 0$ the conditional law of $t\mapsto \ep \wh X_{t/\ep^2}$ given $\mcl H$ converges to the law of planar Brownian motion started from 0 with covariance matrix $\Sigma$ with respect to the local uniform topology. Furthermore, the simple random walk on $\mcl H$ is a.s.\ recurrent and there is a discrete harmonic function $\phi_\infty : \mcl V\mcl H \rta \BB C$ such that a.s.\ 
\eqb \label{eqn-cell-tutte-conv}
\lim_{r\rta\infty} \frac{1}{r} \max_{H \in\mcl H( B_r(0) )} |\phi_\infty(H) - \phi_0(H)| =  0 ,
\eqe 
where $\phi_0(H) := \int_H z\,dz$ is the Euclidean center of $H$. 
\end{thm}
  
Theorem~\ref{thm-general-clt} can be applied to show that random walk on an embedded lattice which is allowed to have finite-range jumps converges to Brownian motion. 
To illustrate this, we give an example of a random environment on $\BB Z^2$ to which Theorem~\ref{thm-general-clt}, but not any of our previous theorems, applies.

\begin{example}[Finite-range random conductance model on $\BB Z^2$] \label{ex-long-range}
Let $\frk c : \BB Z^2\times \BB Z^2 \rta (0,\infty)$ be a conductance function which has finite range in the sense that there exists $N\in\BB N$ such that $\frk c(x,y) = 0$ whenever $|x-y| > N$. Assume that $\frk c$ is stationary and ergodic with respect to spatial translations and satisfies $\BB E[ \frk c(x,y) ] <\infty$ and $\BB E[\frk c(x,y)^{-1}] < \infty$ whenever $|x-y| \leq N$.  
Then Theorem~\ref{thm-general-clt} implies that a.s.\ the random walk on $\BB Z^2$ with conductances $\frk c$ converges in law to Brownian motion in the quenched sense, under diffusive scaling. To see this, for $x\in\BB Z^2$ let $H_x$ be the cell which is the union of the square of side length $1/2$ centered at $x$ and all of the straight line segments from $x$ to points $y\in\BB Z^2$ with $|x-y|\leq N$. It is easily seen that the hypotheses of Theorem~\ref{thm-general-clt} are satisfied for the cell configuration with $\mcl H = \{H_x : x\in\BB Z^2-w\}$, adjacency defined by $H_x\sim H_y$ whenever $|x-y|\leq N$, and conductances defined by $\frk c(H_x,H_y) = \frk c(x ,y )$ (provided we shift $\BB Z^2$ by a uniform random element of $[0,1]\times [0,1]$, as in previous examples). 
\end{example}

\subsection{Why $d=2$ is special}
\label{sec-d=2}

Our arguments are very specific to two dimensions. To give some intuition about why $d=2$ is special, we remark that it is a simple exercise to show that random walk started at the center of Figure~\ref{fig-split-grid} approximates an ordinary boundary reflected Brownian motion (up to time change), while the analogous statement in other dimensions is false. For example, the analogous figure in dimension one is a vertical interval, whose upper half is subdivided into twice as many pieces as its lower half; a simple random walk started at the midpoint of the vertical interval would have a $2/3$ chance to hit the bottom before the top.

Similarly, if one sums the squared edge lengths in the upper half of Figure~\ref{fig-split-grid}, one gets the same value as if one sums the squared edge lengths in the lower half (as both quantities are proportional to area).  In other words, the Dirichlet energy per unit area of the embedded lattice is the same in both halves, a statement that would also be false in any dimension $d\not=2$.  Because Dirichlet energy and area both scale as length squared when $d=2$, we will be able to make sense of the expected Dirichlet energy per area of $\mcl M$ as a quantity that does not depend on scaling, and to show that on large regions, the Dirichlet energy per area averages out to be close to a deterministic constant with high probability.

If one wishes to extend the results of this paper to other dimensions, one might try the following: instead of working modulo scalings of the form $(\mcl M, \frk c) \to (C \mcl M, \frk c)$, work modulo scalings of the form $(\mcl M, \frk c) \to (C \mcl M, C^{d-2}\frk c)$ (where $\frk c$ applied to an edge of $C \mcl M$ is just defined equal $\frk c$ applied to the corresponding edge of $\mcl M$). This way Dirichlet energy per volume would still be a scale invariant notion.  It is an interesting open question to determine the extent to which the results of this paper can be extended to higher dimensions in this context.  To start with a specific example, one could try to prove an invariance principle for some discretization of higher dimensional Gaussian multiplicative chaos. We will not further discuss this problem here.

\begin{figure}[t!]
 \begin{center}
\includegraphics[scale=.55]{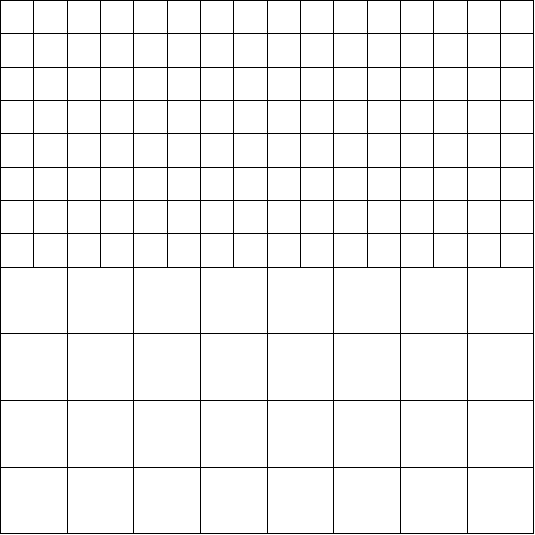} \vspace{-0.01\textheight}
\caption{ A unit square with lower half divided into grid of $2^{-k} \times 2^{-k}$ squares and upper half divided into a grid of $2^{-k-1}$ by $2^{-k-1}$ squares, where $k = 3$. In the $k \to \infty$ limit, a simple random walk on this figure converges to a time change of (boundary-reflected) Brownian motion, but the analogous statement would be false in any dimension other than $2$.}\label{fig-split-grid}
\end{center}
\vspace{-1em}
\end{figure}

\subsection{Outline}
\label{sec-outline}

\begin{figure}[t!]
 \begin{center}
\includegraphics[scale=.85]{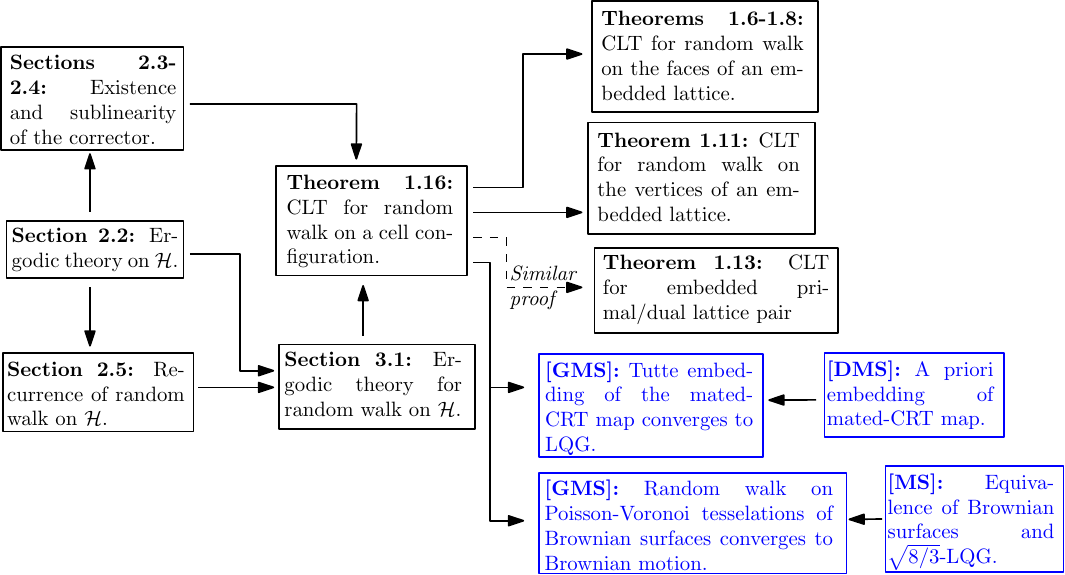}
\vspace{-0.01\textheight}
\caption{Schematic illustration of how results which are proven in this paper (black) and related results proven in other papers (blue) fit together. 
}\label{fig-outline}
\end{center}
\vspace{-1em}
\end{figure}

In this section we provide a moderately detailed overview of the content of the rest of the paper. See Figure~\ref{fig-outline} for a schematic illustration. 

Most of the paper is devoted to the proof of Theorem~\ref{thm-general-clt}. The proof is essentially identical to that of Theorem~\ref{thm-general-clt0}, so the reader who cares only about embedded lattices can think only of the setting of Theorem~\ref{thm-general-clt0}, which is a special case of Theorem~\ref{thm-general-clt}.  
As one might expect from the hypotheses, the theorem will be proven using ergodic theory. Some of our arguments are inspired by those in~\cite{berger-biskup-perc-rw,bp-bounded-conductance,biskup-rwre-survey}, but in many places very different techniques are needed due to the lack of exact stationarity with respect to spatial translations.   
\medskip

\noindent \textbf{Section~\ref{sec-corrector}} is devoted to the proof of two parts of Theorem~\ref{thm-general-clt} (which correspond to Theorems~\ref{thm-general-tutte} and~\ref{thm-recurrent}): the existence of the discrete harmonic function $\phi_\infty$ satisfying~\eqref{eqn-cell-tutte-conv} and the recurrence of the random walk on $\mcl H$. 
We start in Section~\ref{sec-dyadic-system} by defining a particularly convenient collection of squares, called a \emph{uniform dyadic system}, which allows us to formulate another condition which is equivalent to the conditions of Definition~\ref{def-translation-invariance} and which is the only condition which we will use in our proofs. Roughly speaking, a uniform dyadic system $\mcl D$ contains a unique bi-infinite sequence of squares containing any point in $\BB C$ (except for a Lebesgue measure-zero set of points which lie on boundaries of squares) whose side lengths belong to $\{2^{s+k}\}_{k\in\BB Z^2}$, for $s$ a uniform $[0,1]$ random variable. Throughout the rest of the Sections~\ref{sec-corrector} and~\ref{sec-general-clt} we will fix a uniform dyadic system $\mcl D$ which is independent from $\mcl H$.

In Section~\ref{sec-ergodic-avg} we use the backward martingale convergence theorem and ergodicity modulo scaling to prove a.s.\ limit theorems for the averages of various quantities over the origin-containing squares in $\mcl D$ (Lemma~\ref{lem-ergodic-avg}). 
Combined with our finite expectation hypothesis, this leads to a proof that the maximum diameter of the cells which intersect a large origin-containing square in $\mcl D$ is typically of smaller order than its side length (Lemma~\ref{lem-max-cell-diam}) and a bound for the sum of the squared diameter times degree of the cells which intersect such a square (Lemma~\ref{lem-cell-sum}). 

In Section~\ref{sec-corrector-existence} we construct the discrete harmonic function $\phi_\infty$ of Theorem~\ref{thm-general-clt} as a limit of functions $\phi_m$ for $m\in\BB N$ which agree with the a priori embedding function $\phi_0$ on the boundaries of certain large squares in $\mcl D$ and are discrete harmonic on the interiors of these squares. The key input in the proof is a Dirichlet energy bound for these functions (Lemma~\ref{lem-square-energy}), which comes from the diameter$^2$ times degree bound of the preceding subsection. This allows us to show that the discrete Dirichlet energy of $\phi_m - \phi_{m'}$ is small when $m$ and $m'$ are large (Lemma~\ref{lem-tail-energy}), which in turn implies that the $\phi_m$'s are Cauchy in probability and hence the limit $\phi_\infty$ exists.  

In Section~\ref{sec-corrector-sublinearity} we prove that $\phi_\infty$ approximates $\phi_0$ in the sense of~\eqref{eqn-cell-tutte-conv}, as follows. We observe that for fixed $m \in \BB N$, the maximum of $|\phi_m - \phi_0|$ of a large square in $\mcl D$ is typically of smaller order than the side length of the square since $\phi_m$ and $\phi_0$ agree on the boundaries of many small subsquares of this big square (Lemma~\ref{lem-m-0-diff}). We then use the Cauchy-Schwarz inequality and the maximum principle to bound the maximum of $|\phi_\infty-\phi_m|$ over a large square in terms of its Dirichlet energy (see in particular Lemmas~\ref{lem-path-sum}), which we know is small when $m$ is large by the results of Section~\ref{sec-corrector-existence}.

In Section~\ref{sec-recurrence} we prove the recurrence of the simple random walk on $\mcl H$ using the results of Section~\ref{sec-ergodic-avg} and the criterion for recurrence in terms of effective resistance (equivalently, Dirichlet energy). 
\medskip

\noindent
In \textbf{Section~\ref{sec-general-clt}} we will prove our main results, starting with Theorem~\ref{thm-general-clt}. The basic idea of the proof is to apply the martingale functional central limit theorem to show convergence of the image of the random walk on $\mcl H$ under the discrete harmonic function $\phi_\infty$, then apply Theorem~\ref{thm-general-clt-uniform}. In order to apply the multi-dimensional martingale central limit theorem, we need to know that the quadratic variations and covariation of the two coordinates of the image of the walk differ by (approximately) a constant factor. For this purpose we will need analogs of some of the ergodic theory results of Section~\ref{sec-ergodic-avg} when we translate according to the time parameterization of a simple random walk on $\mcl H$, rather than uniformly from Lebesgue measure on a square. Such results are proven in Section~\ref{sec-walk-ergodic} and rely on the recurrence of the walk to obtain tail triviality. 

In Section~\ref{sec-martingale-clt}, we apply these ergodic theory results and the martingale functional central limit theorem to complete the proof of Theorem~\ref{thm-general-clt}. In Section~\ref{sec-clt-uniform} we prove an extension of Theorem~\ref{thm-general-clt} (Theorem~\ref{thm-general-clt-uniform}) wherein the convergence of the random walk on $\mcl H$ to Brownian motion holds uniformly over all choices of starting point for the walk in a compact subset of $\BB C$. In Section~\ref{sec-variant-proof}, we deduce Theorems~\ref{thm-general-clt0} and~\ref{thm-graph-clt} from Theorem~\ref{thm-general-clt} and explain the adaptions needed to prove Theorem~\ref{thm-dual-clt}. 
\medskip

\noindent
\textbf{Section~\ref{sec-lqg-application}} discusses applications of our work to random planar maps and Liouville quantum gravity (which are explored further in the companion paper~\cite{gms-tutte}).
\medskip

\noindent
\textbf{Appendix~\ref{sec-equivalence}} contains a proof of the equivalence of the definitions of translation invariance modulo scaling stated in Definition~\ref{def-translation-invariance}. \textbf{Appendix~\ref{sec-dual-max-diam}} contains a proof that the hypotheses of Theorem~\ref{thm-dual-clt} imply that a.s.\ the maximal size of the edges or dual edges which intersect $B_r(0)$ is $o_r(r)$. This result is not needed for the proofs of any of the other results in the paper or for any of the applications of our results in~\cite{gms-tutte}.

\subsection*{Basic notation}

We write $\BB N = \{1,2,3,\dots\}$ and $\BB N_0 = \BB N \cup \{0\}$.   
For $a < b$, we define the discrete interval $[a,b]_{\BB Z}:= [a,b]\cap\BB Z$.  
If $f  :(0,\infty) \rta \BB R$ and $g : (0,\infty) \rta (0,\infty)$, we say that $f(\ep) = O_\ep(g(\ep))$ (resp.\ $f(\ep) = o_\ep(g(\ep))$) as $\ep\rta 0$ if $f(\ep)/g(\ep)$ remains bounded (resp.\ tends to zero) as $\ep\rta 0$. We similarly define $O(\cdot)$ and $o(\cdot)$ errors as a parameter goes to infinity. 
If $f,g : (0,\infty) \rta [0,\infty)$, we say that $f(\ep) \preceq g(\ep)$ if there is a constant $C>0$ (independent from $\ep$ and possibly from other parameters of interest) such that $f(\ep) \leq  C g(\ep)$. We write $f(\ep) \asymp g(\ep)$ if $f(\ep) \preceq g(\ep)$ and $g(\ep) \preceq f(\ep)$.

\section{Existence and sublinearity of the corrector}
\label{sec-corrector}

\subsection{Dyadic systems}
\label{sec-dyadic-system}

A key technical tool in our proofs is the concept of a \emph{dyadic system}, which leads to a convenient scale/translation invariant way of decomposing space into ``blocks" and which we will define in this subsection.  
Let $S \subset \BB C$ be a square (not necessarily dyadic). We write $|S|$ for the side length of $S$. A \emph{dyadic child} of $S$ is one of the four squares $S' \subset S$ (with $|S'| = |S|/2$) whose corners include one corner of $S$ and the center of $S$. A \emph{dyadic parent} of $S$ is one of the four squares with $S$ as a dyadic child. Note that each square has 4 dyadic parents and 4 dyadic children. A \emph{dyadic descendant} (resp.\ \emph{dyadic ancestor}) of $S$ is a square which can be obtained from $S$ by iteratively choosing dyadic children (resp.\ parents) finitely many times. 

A \emph{dyadic system} is a collection $\mcl D$ of closed squares (not necessarily dyadic) with the following properties.
\begin{enumerate}
\item If $S\in\mcl D$, then each of the four dyadic children of $S$ is in $\mcl D$.
\item If $S\in\mcl D$, then exactly one of the dyadic parents of $S$ is in $\mcl D$. 
\item Any two squares in $\mcl D$ have a common dyadic ancestor. 
\end{enumerate} 
See Figure~\ref{fig-dyadic-system} for an illustration of the origin-containing squares of a dyadic system. 
The set of all side lengths of the squares in a dyadic system is precisely $\{2^{s+k}\}_{k\in\BB Z}$ for some $s\in [0,1]$ (determined by the system). 
If $\mcl D$ is a dyadic system, then for any $z\in \BB C$ there is a bi-infinite sequence of squares $\{S_k^z \}_{k\in\BB Z} \subset \mcl D$ which contain $z$, numbered so that $S_k^z$ is a dyadic parent of $S_{k-1}^z$ for each $k$. If $z$ does not lie on the boundary of a square in $\mcl D$, then this sequence is unique up to translation of the indices. If we are given any single square $S\in\mcl D$ and all of its dyadic ancestors, then $\mcl D$ is uniquely determined: $\mcl D$ is the set of all dyadic descendants of $S$ and its dyadic ancestors. This in particular allows us to define a topology on the space of dyadic systems, e.g., by looking at the local Hausdorff distance on the union of the origin-containing squares.

In what follows, there will generally be a dyadic system which is understood from the context.  We emphasize that when we refer to a dyadic ancestor of a square $S$, we mean any of the possible dyadic ancestors of $S$ and not just the one which is contained in the dyadic system.

A \emph{uniform dyadic system} is the random dyadic system defined as follows. Let $s$ be sampled uniformly from $[0,1]$ and, conditional on $s$, let $w$ be sampled uniformly from $[0,2^s] \times [0,2^s]$. Set $S_0 := [0,2^s]\times [0,2^s] - w$. For $k \in \BB N$, inductively let $S_k$ be sampled uniformly from the four dyadic parents of $S_{k-1}$. Then let $\mcl D$ be the set of all dyadic descendants of $S_k$ for each $k\in\BB N$. 
We view $\mcl D$ as a dyadic system, so in particular $\mcl D$ only determines the sequence of origin-containing squares up to an index shift.
The main reason for considering a uniform dyadic system is the following lemma.
  
\begin{figure}[t!]
 \begin{center}
\includegraphics[scale=.65]{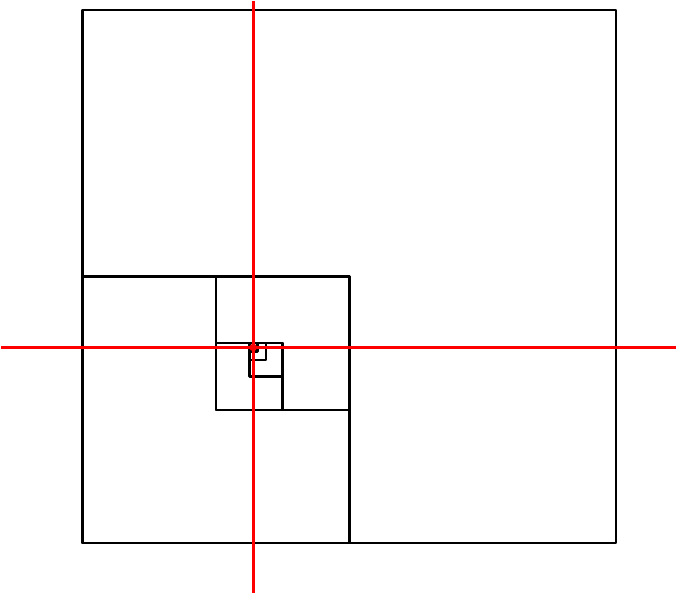}
\vspace{-0.01\textheight}
\caption{The squares $S_k$ which contain the origin (red) in a dyadic system, for $k \leq 0$. Note that the dyadic system also includes the origin-containing squares $S_k$ for $k  > 0$, which have side lengths tending to $\infty$ as $k\rta \infty$, and all of the dyadic descendants of the $S_k$'s. 
}\label{fig-dyadic-system}
\end{center}
\vspace{-1em}
\end{figure} 

\begin{lem} \label{lem-uniform-dyadic}
Let $\mcl D$ be a uniform dyadic system and let $C>0$ and $z\in\BB C$ be deterministic. Then the scaled/translated dyadic system $C ( \mcl D-z )$ obtained by translating and scaling the squares agrees in law with $\mcl D$. 
\end{lem} 
\begin{proof} 
We first check translation invariance. Fix $z\in\BB C$ and set $\wt{\mcl D} := \mcl D-z$. We seek to show that $\wt{\mcl D} \eqD \mcl D$. 
For $k\in\BB Z$, let $S_k$ (resp.\ $\wt S_k$) be the square in $\mcl D$ (resp.\ $\wt{\mcl D}$) containing 0 with side length $2^{k+s}$.  
Since a dyadic system is a.s.\ determined by any fixed square and all of its dyadic ancestors, it suffices to show that the total variation distance between the laws of $\{S_k\}_{k\geq K}$ and $\{\wt S_k\}_{k\geq K}$ tends to 0 as $K \rta\infty$. 

We first argue that $\wt S_0 \eqD S_0$. Indeed, $\wt S_0 $ is obtained by translating the square of $\mcl D$ which has side length $2^s$ and contains $z$ by $-z$, so $\wt S_0 = [0 , 2^s] \times [0,2^s] - \wt w$, where $\wt w \in [0,2^s ]\times [0,2^s]$ is chosen so that $(z+w) - \wt w \in 2^s\BB Z^2$. Since $w$ is uniform on $[0,2^s] \times [0,2^s]$ conditional on $s$, the same is true of $\wt w$. Thus $\wt S_0 \eqD S_0$. 

We next observe that as $k\rta\infty$, the fraction of the $4^k$ dyadic ancestors of $S_0$ with side length $2^{s+k}$ which are not also dyadic ancestors of the square $\wt S_0 + z\in\mcl D$ decays exponentially in $k$. Since each $S_k$ is sampled uniformly from the set of dyadic ancestors of $S_0$ with side length $2^{s+k}$, it follows that the total variation distance between the law of $\{S_k - z\}_{k \geq K}$ and the law of a sequence of random squares defined in the same manner as $\{S_k  \}_{k \geq K}$ as above but with $\wt S_0$ used in place of $S_0$ tends to 0 as $K \rta\infty$.  
On the other hand, a.s.\ $z\in S_k$ and hence $\wt S_k =  S_k - z$ for large enough $k\in\BB N$.
Combining this with the preceding paragraph shows our claim about the total variation distance, and hence that $\mcl D \eqD \wt{\mcl D}$. 

One checks that $C\mcl D\eqD \mcl D$ for $C>0$ using a similar but simpler argument to the one above, based on the fact that the fractional part of $s  +\log_2 C$ is uniform on $[0,1]$. 
\end{proof}

Throughout the rest of this section, we fix a random cell configuration $\mcl H$ and a uniform dyadic system $\mcl D$ independent from $\mcl H$ (in subsequent subsections we will also require that $\mcl H$ satisfies the hypotheses of Theorem~\ref{thm-general-clt}). We let $\{S_k\}_{k\in\BB Z}$ be the sequence of squares containing the origin in $\mcl D$, labeled so that the side length of $S_0$ is $2^s$, as above.

It will be convenient to have a natural, scale-invariant way to index the squares in $\mcl D$. To this end, for $z\in\BB C$ and $m>0$, we define 
\eqb \label{eqn-mass-time-def}
\wh S_m^z  := \left\{ \text{largest square $S\in\mcl D$ with}\: z\in S \: \text{and} \:  \sum_{H\in \mcl H(S )} \frac{\op{Area}(H\cap S )}{\op{Area}(H)} \leq m \right\} . 
\eqe 
Then $\wh S_m^z$ is well-defined on the full-probability event that $z$ does not lie on the boundary of any square in $\mcl D$. 
We abbreviate $\wh S_m := \wh S_m^0$. 

Roughly speaking, $\wh S_m^z$ is the largest square containing $z$ which contains at most $m$ cells, but fractional parts of cells which intersect the boundary of the square are also counted. We emphasize that we do \emph{not} restrict to integer values of $m$. 

\begin{lem} \label{lem-dyadic-square-contain}
Almost surely, for each square $S$ in $\mcl D$ which contains $z$ in its interior there is a non-trivial interval $I\subset (0,\infty)$ such that $S = \wh S_m^z$ for all $m\in I$. 
Furthermore, $\wh S_m^z = \wh S_m^w$ whenever $w \in \wh S_m^z\setminus \bdy \wh S_m^z$. 
\end{lem}
\begin{proof}
If $S'$ is one of the four dyadic parents of $S$, then 
\eqbn
 \sum_{H\in \mcl H(S')} \frac{\op{Area}(H\cap S' )}{\op{Area}(H)} >  \sum_{H\in \mcl H(S )} \frac{\op{Area}(H\cap S )}{\op{Area}(H)} .
\eqen
Note that this would not be true if we did not include fractional parts of cells in the sum. The lemma statement is immediate from this observation and the definition~\eqref{eqn-mass-time-def} of $\wh S_m^z$. 
\end{proof}

Dyadic systems allow us to formulate yet another equivalent definition of translation invariance modulo scaling (which is the only version which we will use in our proofs) and also provide a convenient useful tool for proving the equivalence of the conditions in Definition~\ref{def-translation-invariance}. 
 
\begin{lem} \label{lem-dyadic-resample}
Let $\mcl H$ be a random cell configuration. Each of the conditions of Definition~\ref{def-translation-invariance} (with $\mcl H$ in place of $\mcl M$) are equivalent to each other and to the following additional condition. 
\begin{enumerate}
\setcounter{enumi}{4}
\item \textbf{Dyadic system re-sampling property.} Let $\mcl D $ be a uniform dyadic system independent from $\mcl H$ and define $\wh S_m =\wh S_m^0$ for $m > 0$ as in~\eqref{eqn-mass-time-def}. For each $m > 0$, the following is true. Conditional on $\mcl H$ and $\mcl D$, let $z$ be sampled uniformly from Lebesgue measure on $\wh S_m$. Then the translated cell configuration/dyadic system pair $(\mcl H - z , \mcl D - z)$ agrees in law with $(\mcl H ,\mcl D)$ modulo scaling, i.e., there is a random $C>0$, possibly depending on $z$, such that $(C(\mcl H -z) , C(\mcl D-z) ) \eqD (\mcl H ,\mcl D)$.   \label{item-dyadic}
\end{enumerate}   
\end{lem}  

The proof of Lemma~\ref{lem-dyadic-resample} proceeds by showing that each of the conditions of Definition~\ref{def-translation-invariance} is equivalent to the dyadic system resampling condition in the lemma statement. The proof is routine in nature, so is deferred to Appendix~\ref{sec-equivalence} to avoiding interrupting the main argument.
We emphasize that the condition given in Lemma~\ref{lem-dyadic-resample} is the only form of translation invariance modulo scaling which is used in the proofs of our main results.

\begin{remark} \label{remark-1d-dyadic-system}
The definition of the uniform dyadic system has an obvious $d$-dimensional
generalization for any integer $d\geq 1$, where squares are replaced by $d$-dimensional cubes, and each
cube has $2d$ dyadic children. We will make use of one-dimensional uniform dyadic systems in Section~\ref{sec-walk-ergodic} of  this paper
when we discuss continuous time random walks indexed by $\BB R$. 
\end{remark}

\subsection{Ergodic averages using dyadic systems}
\label{sec-ergodic-avg}
 
Throughout the rest of this section and the next, we assume that the cell configuration $\mcl H$ satisfies the hypotheses of Theorem~\ref{thm-general-clt}. 
Recall the independent dyadic system $\mcl D$ and the sequence $\{S_k\}_{k\in\BB Z}$ of origin-containing squares. In this subsection we will establish some basic properties of the pair $(\mcl H, \mcl D)$ which we will use frequently in what follows. In particular, we will record an a.s.\ convergence statement for averages over the squares $\wh S_m^0$ of~\eqref{eqn-mass-time-def} (Lemma~\ref{lem-ergodic-avg}) which follows from the backward martingale convergence theorem and our ergodicity hypothesis (Definition~\ref{def-ergodic}). We will then use this to show that a.s.\ the maximal size of the cells of $\mcl H$ which lie at distance at most $r$ from the origin grows sublinearly in $r$ (Lemma~\ref{lem-max-cell-diam}). We will also show that the sum of $\op{diam}(H)^2 \left(  \pi(H)+ \pi^*(H) \right)$ over all cells $H$ which intersect one of the origin-containing squares $S_k$ is typically of order $|S_k|^2$ (Lemma~\ref{lem-cell-sum}). 

\begin{defn} \label{def-dyadic-sigma-algebra}
Recall the squares $\wh S_m = \wh S_m^0$ for $m  >0$ from~\eqref{eqn-mass-time-def}. 
For $m  > 0$, let $\mcl F_m$ be the $\sigma$-algebra generated by the measurable functions $F = F(\mcl H,\mcl D)$ of $(\mcl H,\mcl D)$ which satisfy 
\eqb \label{eqn-dyadic-sigma-algebra}
F\left(C(\mcl H-z) , C(\mcl D-z) \right) = F(  \mcl H,   \mcl D) ,\quad \forall z\in \wh S_m ,\quad \forall C >0 .
\eqe
\end{defn}

The $\sigma$-algebra $\mcl F_m$ encodes $(\mcl H,\mcl D)$, viewed modulo translation within $\wh S_m$ and scaling, but not the location of the origin within the square $\wh S_m$.  

\begin{lem} \label{lem-tail-trivial}
Every event in $\bigcap_{m > 0} \mcl F_m$ has probability zero or one.
\end{lem}
\begin{proof}
If $F =F(\mcl H,\mcl D)$ is a $\bigcap_{m > 0} \mcl F_m$-measurable function, then since the union of the squares $\wh S_m$ is a.s.\ equal to all of $\BB C$, it follows that $F(C(\mcl H-z) , C(\mcl D-z)) = F(\mcl H,\mcl D)$ for every $C>0$ and $z\in\BB C$. 
Let $\frk m$ be the law of $\mcl D$ and let
\eqbn
\wt F = \wt F(\mcl H) :=  \int F(\mcl H , \mcl D) \, d\frk m(\mcl D) 
\eqen
so that $\wt F$ depends only on $\mcl H$. 

By Lemma~\ref{lem-uniform-dyadic}, if $C>0$ and $z\in\BB C$ are random and independent from $\mcl D$ (but allowed to depend on $\mcl H$), then $C(\mcl D -z)$ is a uniform dyadic system independent from $\mcl H$, i.e., $(C(\mcl H-z) , C(\mcl D-z)) \eqD (C(\mcl H-z) , \mcl D)$. From this  and the scale/translation invariance of $F$, we infer that for any such choice of $C$ and $z$, we have $\wt F(C(\mcl H-z)) = \wt F(\mcl H)$. 
Our ergodicity modulo scaling hypothesis for $\mcl H$ (Definition~\ref{def-ergodic}) therefore shows that $\wt F$ is a.s.\ equal to a deterministic constant, i.e., $F$ is a.s.\ determined by $\mcl D$. In fact, by the first paragraph $F$ is a.s.\ determined by $\mcl D$ viewed modulo translation and scaling. 
 
If we are given any arbitrary \emph{deterministic} dyadic system $\mcl D'$ and $R>0$, then we can find $C>0$ and $z\in\BB C$ (depending on $\mcl D$ and $\mcl D'$) such that the squares surrounding the origin in $C(\mcl D-z)$ and $\mcl D'$ with side lengths in $[1/R ,R]$ are identical. Since $R>0$ is arbitrary, a dyadic system is determined by its origin-containing squares, and the successive origin-containing squares in $\mcl D$ are chosen uniformly at random, it follows that $F$ is equal to a deterministic constant a.s. 
\end{proof}

From Lemma~\ref{lem-dyadic-resample} and~\ref{lem-tail-trivial}, we obtain the following ergodicity statement, which will be a key tool in what follows.

\begin{lem} \label{lem-ergodic-avg}
Let $F = F(\mcl H,\mcl D)$ be a measurable function on the space of cell configuration/dyadic system pairs which is scale invariant (i.e., $F(C\mcl H , C\mcl D) = F(\mcl H ,\mcl D)$ for each $C>0$). 
If either $\BB E[|F|] < \infty$ or $F\geq 0$ a.s., then a.s.\ 
\eqb \label{eqn-ergodic-avg}
\lim_{k \rta\infty} \frac{1}{|S_{k}|^2} \int_{S_{k}} F(\mcl H-z , \mcl D-z) \, dz  = \BB E[F]  .
\eqe  
\end{lem}
\begin{proof}
Suppose first that $\BB E[|F|] < \infty$. 
By Lemma~\ref{lem-dyadic-resample}, with $\mcl F_m$ as above, it holds for each $m > 0$ that
\eqbn
\frac{1}{|\wh S_m |^2} \int_{\wh S_m} F(\mcl H-z, G-z) \, dz  = \BB E\left[ F \,|\, \mcl F_m \right] .
\eqen
Since $\mcl F_m$ is a decreasing sequence of $\sigma$-algebras and each $S_k$ is one of the $\wh S_m$'s, the backward martingale convergence theorem~\cite[Theorem 5.6.1]{durrett} implies that the limit on the left side of~\eqref{eqn-ergodic-avg} exists (as a random variable) a.s.\ and in $L^1$. The limit is $\bigcap_{m > 0} \mcl F_m$-measurable, so is a.s.\ equal to $\BB E[F]$ by Lemma~\ref{lem-tail-trivial}. 
If $F\geq 0$ a.s.\ and $\BB E[F] = \infty$, then applying the above result with $F\wedge C$ for $C>0$ in place of $F$ gives
\eqbn
\liminf_{k \rta\infty} \frac{1}{|S_{k}|^2} \int_{S_{k}} F(\mcl H-z , \mcl D-z) \, dz
\geq \lim_{k \rta\infty} \frac{1}{|S_{k}|^2} \int_{S_{k}} F(\mcl H-z , \mcl D-z) \wedge C \, dz
= \BB E[F\wedge C] ,
\eqen
which tends to $\infty$ as $C\rta\infty$. 
\end{proof}  

Lemma~\ref{lem-ergodic-avg} will allow us to prove a number of properties of the origin-containing squares $\{S_k\}_{k\in \BB Z}$, but we will also have occasion to consider other squares in $\mcl D$. The following lemma will enable us to do so. For the statement, we emphasize that in our terminology each square $S$ has $4^\el$ dyadic ancestors of side length $2^\el |S|$ for each $\el\in\BB N$.

\begin{lem} \label{lem-more-squares}
Let $E(S) = E(S,\mcl H)$ be an event depending on a square $S \subset \BB C$ and our cell configuration $\mcl H$. 
Suppose that a.s.\ $E(S_k)$ occurs for each large enough $k\in\BB N$. Then for each $\el\in\BB N$, it is a.s.\ the case that for large enough $k\in\BB N$, $E(S)$ occurs for each dyadic ancestor $S$ of $S_k$ of side length at most $2^\el |S_k|$ (even those which do not belong to $\mcl D$) and each dyadic descendant $S$ of $S_k$ of side length at least $2^{-\el} |S_k|$. 
\end{lem} 
\begin{proof}
We first consider dyadic ancestors. 
For $k \in \BB N$, let $\mcl K_k$ be the smallest $k'\geq k$ for which $E(S)^c$ occurs for at least one of the dyadic ancestor $S$ of $S_{k'}$ of side length at most $2^\el |S_{k'}|$, or $\mcl K_k =\infty$ if no such $k'$ exists. We want to show that $\BB P[\mcl K_k=\infty] \rta 1$ as $k\rta\infty$. 

We observe that $\mcl K_k$ is a stopping time for the filtration generated by $\mcl H$ and $\{S_j\}_{j\leq k'}$, so if $\mcl K_k < \infty$ and we condition on $\mcl H$ and $\{S_j\}_{j\leq k'}$ then for each $r\in\BB N$ the square $S_{\mcl K_k + r}$ is sampled uniformly from the $4^r$ possibilities. On the event $\{\mcl K_k < \infty\}$, let $S_* $ be a dyadic ancestor of $S_{\mcl K_k}$ with side length at most $2^\el |S_{\mcl K_k}|$ for which $E(S_*)^c$ occurs, chosen in a manner depending only on $\mcl H$ and $\{S_j\}_{j\leq \mcl K_k}$ and let $\el_* \in [1,\el]_{\BB Z}$ be chosen so that $|S_*| = 2^{\el_*} |S_{\mcl K_k}|$. Then
\eqbn
  \BB P\left[ \mcl K_k < \infty \right] \leq \frac{  \BB P\left[ E(S_{\mcl K_k + \el_* })^c \right]  }{ \BB P\left[  S_{k+\el_*} = S_* \,|\,  \mcl K_k < \infty     \right] } \leq 4^\el \BB P\left[  E(S_{\mcl K_k + \el_* })^c \right] , 
\eqen
which tends to $0$ as $k\rta\infty$ since a.s.\ $E(S_k)$ occurs for each large enough $k$. 

Next we consider dyadic descendants. We will use the re-sampling property of Lemma~\ref{lem-dyadic-resample}.  
For $m\in\BB N$, let $M_m$ be the largest $m' \leq m$ for which $E(S)^c$ occurs for at least one dyadic descendant $S$ of $\wh S_{m'}$ of side length at least $2^{-\el} |\wt S_{m'}|$, or $M_m=-\infty$ if there is no such $m'$. We want to show that a.s.\ $\limsup_{m\rta\infty} M_m < \infty$, equivalently a.s.\ $\limsup_{m\rta\infty} K_{M_m} < \infty$. 

For $m > 0$, let $K_m \in\BB Z$ be the largest $k\in\BB Z$ for which $\wh S_m  = S_k$. It suffices to show that $\BB P[K_{M_m} > k] \rta 0$ as $k\rta\infty$, uniformly over all $m > 0$. 
Each $M_m$ is a reverse stopping time for the filtration of Definition~\ref{def-dyadic-sigma-algebra}. By Lemma~\ref{lem-dyadic-resample}, if we condition on $\mcl F_{M_m}$ then for each $r \in\BB N$, the square $S_{K_{M_m} - r}$ is sampled uniformly from the $4^r$ dyadic descendants of $S_{K_{M_m}}$ with side length $2^{-r} |S_{K_{M_m}}|$. As above, on the event $\{M_m \not=-\infty\}$ we let $S_*$ be a dyadic descendant of $S_{K_{M_m}} = \wh S_{M_m}$ with side length at least $2^{-\el} |\wh S_{M_m}|$ for which $E(S_*)^c$ occurs, chosen in a $\mcl F_m$-measurable manner and we let $\el_* \in [1,\el]_{\BB Z}$ be chosen so that $|S_*| = 2^{-\el_*} |\wh S_{M_m}|$. Then 
\eqbn
\BB P\left[  K_{M_m} > k\right] \leq \frac{ \BB P\left[ E(S_{K_{M_m} -  \el_* })^c ,\, K_{M_m} > k \right] }{  \BB P\left[ S_{K_{M_m} - \el_*} = S_* \,|\, K_{M_m} > k   \right]  } \leq 4^\el \BB P\left[ E(S_{K_{M_m} -  \el_* })^c ,\, K_{M_m} > k \right] ,
\eqen
which tends to 0 as $k\rta\infty$, uniformly in $m$, since a.s.\ $E(S_k)$ occurs for large enough $k$. Since $K_{M_m}$ is non-decreasing in $m$, this implies that a.s.\ $\limsup_{m\rta\infty} K_{M_m}  <\infty$. 
\end{proof}

In what follows, we will apply Lemma~\ref{lem-ergodic-avg} with
\eqb \label{eqn-area-diam-deg}
F(\mcl H,\mcl D) = \frac{\op{diam}(H_0)^2 }{\op{Area}(H_0) } \left(  \pi(H_0) + \pi^*(H_0)  \right) .
\eqe
The random variable $F$ has finite expectation by hypothesis. We first establish that the maximum cell size of $\mcl H$ grows sublinearly in the distance to the origin.

\begin{lem} \label{lem-max-cell-diam}
Almost surely, for each $\ep \in (0,1)$ it holds for large enough $k\in\BB N$ that
\eqb \label{eqn-max-cell-diam}
\op{diam}(H) \leq \ep |S_k| ,\quad \forall H \in \mcl H(S_k) .
\eqe 
\end{lem}
\begin{proof}
For $\ep  \in (0,1)$, let
\eqbn
M_\ep = M_\ep(\mcl H ,\mcl D) := \min\left\{  m > 0 :  \op{diam}(H_0) \leq \ep |\wh S_m| \right\}  .
\eqen
Since $F$, defined as in~\eqref{eqn-area-diam-deg}, has finite expectation there is a deterministic $m_\ep \in \BB N$ such that $\BB E[F \BB 1_{( M_\ep > m_\ep )} ] \leq \ep^3$. Note that $F\BB 1_{(M_\ep  >m_\ep )}$ is a scale invariant random variable in the sense of Lemma~\ref{lem-ergodic-avg}. By Lemma~\ref{lem-ergodic-avg} applied with $F \BB 1_{(M_\ep > m_\ep )}$ in place of $F$, we get that a.s.\ 
\eqb  \label{eqn-max-cell-truncate}
\limsup_{m \rta\infty} \frac{1}{|\wh S_m|^2} \int_{\wh S_m}  \frac{\op{diam}(H_z)^2}{\op{Area}(H_z)   } \left( \pi(H_z) + \pi^*(H_z) \right)    \BB 1_{\left( \op{diam}(H_z) > \ep |\wh S_m| \right)}  \, dz \leq \ep^3 
\eqe 
where here we use that $ \op{diam}(H_z) > \ep |\wh S_m |  $ implies $ M_\ep(\mcl H-z,\mcl D-z)  > m_\ep $ whenever $m\geq m_\ep$ and $z\in \wh S_m$. 
By breaking up the integral over $\wh S_m$ into the sum of the integrals over $H \cap \wh S_m$ for cells $H\in\mcl H(\wh S_m)$, then dropping the boundary terms, and using that $\pi(H_z) + \pi^*(H_z) \geq 1$, we see that~\eqref{eqn-max-cell-truncate} implies that
\eqb \label{eqn-max-sum-interior} 
\limsup_{m \rta\infty }  \sum_{  H\in \mcl H(\wh S_m) \setminus \mcl H(\bdy \wh S_m)}    \frac{\op{diam}(H)^2}{|\wh S_m|^2}    \BB 1_{\left(\op{diam}(H) > \ep |\wh S_m|\right)}              \leq  \ep^3  .
\eqe 
Each non-zero summand on the left side of~\eqref{eqn-max-sum-interior} is at least $\ep^2$, so it follows from~\eqref{eqn-max-sum-interior} (and the fact that each $S_k$ is one of the $\wh S_m$'s) that a.s.\ for large enough $k \in \BB N$, 
\eqb \label{eqn-max-cell-interior} 
\op{diam}(H) \leq \ep |S_{k}|  ,\quad \forall H\in \mcl H(S_{k}) \setminus \mcl H(\bdy S_{k})  .
\eqe 

It remains to deal with cells which intersect $\bdy S_k$. For this purpose, consider a fixed $k\in\BB N$ and let $k'$ be the smallest integer $\geq k$ for which each cell in $\mcl H(S_{k'})$ has diameter smaller than $\ep |S_{k}| $ (such a $k'$ exists since each cell is compact). 
We claim that a.s.\ $k' = k$ for large enough $k$, which will show~\eqref{eqn-max-cell-diam}. 

Indeed, since $\ep < 1$ each $H\in\mcl H(S_{k})$ is contained in one of the 16 dyadic ancestors of $S_{k'}$ with side length $4|S_{k'}|$. 
By~\eqref{eqn-max-cell-interior} (applied with $\ep /16$ in place of $\ep$) together with Lemma~\ref{lem-more-squares}, it is a.s.\ the case that for large enough $k \in \BB N$, each cell $H$ which is contained in any one of these 16 dyadic ancestors has diameter at most $(\ep/4) |S_{k'}|$. 
By the definition of $k'$, if $k'\not=k$ then there is some cell in $\mcl H(S_{k}) \subset \mcl H(S_{k'})$ with diameter at least $(\ep/2) |S_{k'}|$. Therefore a.s.\ $k' = k$ for large enough $k$.  
\end{proof}

\begin{lem} \label{lem-cell-sum}
There is a deterministic constant $C >0$ such that a.s.\ 
\eqb \label{eqn-cell-sum}
\limsup_{k \rta\infty} \frac{1}{|S_{k}|^2} \sum_{H\in \mcl H(S_{k})} \op{diam}(H)^2 \left(  \pi(H)+ \pi^*(H) \right)  \leq C .
\eqe 
\end{lem}
\begin{proof}
By Lemma~\ref{lem-ergodic-avg} applied with $F$ as in~\eqref{eqn-area-diam-deg}, we find that there is a deterministic, finite constant $C_0 >0$ such that a.s.\
\eqbn
\lim_{k \rta\infty} \frac{1}{|S_{k}|^2} \int_{S_{k}} \frac{\op{diam}(H_z)^2}{\op{Area}(H_0)  }  \left(  \pi(H_z)+ \pi^*(H_z) \right)    \, dz   = C_0 .
\eqen  
By breaking up the integral over $S_k$ into the sum of the integrals over $H \cap S_k$ for cells $H\in\mcl H(S_k)$, then dropping the boundary terms, we get that a.s.
\eqb \label{eqn-cell-sum-interior} 
\limsup_{k\rta\infty } \frac{1}{|S_k|^2} \sum_{  H\in \mcl H(S_k) \setminus \mcl H(\bdy S_k)}    \op{diam}(H )^2   \left(  \pi(H )+ \pi^*(H ) \right)   \leq C_0 .
\eqe 
To deal with the cells which intersect $\bdy S_k$, we observe that Lemma~\ref{lem-max-cell-diam} (with $\ep = 1/2$, say) implies that a.s.\ for large enough $k\in\BB N$ each cell in $\mcl H(S_k)$ is contained in the interior of an appropriate dyadic ancestor of $S_k$ of side length $4|S_k|$. By~\eqref{eqn-cell-sum-interior} and Lemma~\ref{lem-more-squares} (applied with $\el=2$) we find that~\eqref{eqn-ergodic-avg} holds a.s.\ with $C = 16 C_0$. 
\end{proof}

\subsection{Existence of the corrector} 
\label{sec-corrector-existence}

In this subsection we will construct the discrete harmonic function $\phi_\infty$ appearing in Theorem~\ref{thm-general-tutte} as a limit of functions $\phi_m : \mcl H\rta \BB C $ for $m\in\BB N_0$ which interpolate between the function $\phi_0$ of Theorem~\ref{thm-general-tutte} (which sends each cell to its Euclidean center) at $m=0$ and the discrete harmonic function $\phi_\infty$ as $m\rta\infty$. 
Roughly speaking, $\phi_m$ will agree with the a priori embedding on the cells which intersect boundary of each square of $\mcl D$ which contains approximately $m$ cells of $\mcl H$; and will be discrete harmonic on the cells contained in the interior of each such square. 

Throughout this section we define the square $\wh S_m^z$ for $z\in\BB C$ and $m > 0$ as in~\eqref{eqn-mass-time-def} and we define the collection of squares
\eqb \label{eqn-dyadic-tiling}
\mcl S_m := \left\{ \wh S_m^z : z\in\BB C ,\: \text{$z$ is not on the boundary of a square in $\mcl D$} \right\} \subset\mcl D. 
\eqe 
If $z' $ is in the interior of $ \wh S_m^z$, then $\wh S_m^{z'} = \wh S_m^z$, so the squares in $\mcl S_m$ intersect only along their boundaries. Clearly, these squares cover all of $\BB C$. 
See Figure~\ref{fig-tiling} for an illustration of these squares.

\begin{figure}[t!]
 \begin{center}
\includegraphics[scale=.65]{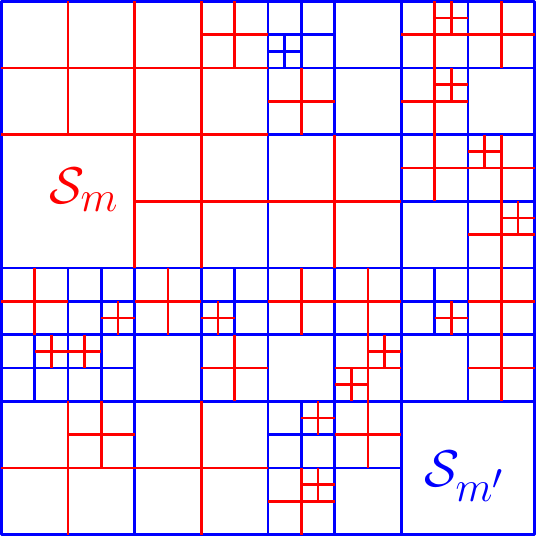}
\vspace{-0.01\textheight}
\caption{The collections of squares $\mcl S_m$ and $\mcl S_{m'}$ of $\mcl D$ for $m < m'$. The squares of $\mcl S_m$ (resp.\ $\mcl S_{m'}$) are the minimal squares with red and/or blue (resp.\ blue) boundary and each contain exactly $m$ (resp.\ $m'$) cells of $\mcl H$ if we count fractional parts of cells.  
}\label{fig-tiling}
\end{center}
\vspace{-1em}
\end{figure} 

We now use the set of squares $\mcl S_m$ to define the above-mentioned functions $\phi_m : \mcl H\rta \BB C $ for $m\in\BB N_0$.  
As in Theorem~\ref{thm-general-tutte}, let
$\phi_0(H) := \frac{1}{\op{Area}(H)} \int_H z \,dz $.
For $m\in\BB N$, let $\phi_m :   \mcl H   \rta \BB C$ be the function which agrees with $\phi_0$ on $\mcl H(\bdy \wh S)$ for each $\wh S \in \mcl S_m$ and is $\frk c$-discrete harmonic on $\mcl H(\wh S)\setminus \mcl H(\bdy \wh S)$ for each such $\wh S$. 

We observe that if $C>0$ and $z\in\BB C$ and we replace $(\mcl H,\mcl D)$ by the scaled/translated pair $(C (\mcl H-z) , C(\mcl D -z))$, then the corresponding functions $\phi_m$ are replaced by $C ( \phi_m(\cdot / C + z) -z ) $. This will be important when we apply Lemma~\ref{lem-ergodic-avg}. 

The main goal of this subsection is to establish the following proposition. 

\begin{prop} \label{prop-corrector}
There exists a $\frk c$-discrete harmonic function $\phi_\infty : \mcl H\rta\BB C$ such that 
\eqb \label{eqn-corrector-conv}
\phi_m - \phi_m(H_0) \rta \phi_\infty \quad\text{in probability as $m\rta \infty $},
\eqe 
with respect to the topology of uniform convergence on finite subsets of $\mcl H$, i.e., for each finite set $\mcl A\subset \mcl H$ (chosen in some measurable manner) and each $\ep > 0$,
\eqbn
\lim_{m\rta\infty} \BB P\left[  \max_{H\in \mcl A} |\phi_m(H) - \phi_m(H_0) - \phi_\infty(H)| > \ep \right]  = 0.
\eqen
Furthermore, $\phi_\infty$ is covariant with respect to scaling and translation in the sense that if $C>0$ and $z\in\BB C$ and we replace $(\mcl H,\mcl D)$ by the scaled/translated pair $(C (\mcl H-z) , C(\mcl D-z))$, then $\phi_\infty$ is replaced by $C ( \phi_\infty(\cdot / C + z) - \phi_\infty(H_z) ) $.
\end{prop}

The idea of the proof of Proposition~\ref{prop-corrector} is as follows. We first use Lemma~\ref{lem-cell-sum} to argue that $\phi_0$ (and hence each $\phi_m$) has uniformly bounded Dirichlet energy over $S_k$ when we rescale by $|S_k|^{-2}$ (Lemma~\ref{lem-square-energy}). Once this is established, we will use the orthogonality of the increments $\phi_m - \phi_{m-1}$ with respect to the Dirichlet inner product to show that for large enough $m\in\BB N$, the expected discrete ``Dirichlet energy per unit area" of $\phi_m - \phi_{m'}$, which is made precise in~\eqref{eqn-se-def}, is small uniformly over all $m'\geq m$ (Lemma~\ref{lem-tail-energy}). This will allow us to show that the functions $\phi_m - \phi_m(H_0)$ are Cauchy in probability, and thereby establish Proposition~\ref{prop-corrector}.

\begin{lem} \label{lem-square-energy}
For each $n,m \in\BB N_0$ with $n\geq m$, we have (in the notation of Definition~\ref{def-discrete-dirichlet})
\eqb \label{eqn-energy-mono} 
\op{Energy}\left( \phi_m ; \mcl H(\wh S_n) \right) \leq \op{Energy}\left( \phi_0 ; \mcl H(\wh S_n) \right) .
\eqe 
Furthermore, there is a deterministic constant $C>0$ such that a.s.\  
\eqb \label{eqn-square-energy}  
  \limsup_{k \rta\infty} \frac{1}{| S_k|^2} \op{Energy}\left( \phi_0 ; \mcl H( S_k ) \right)  \leq C .
\eqe 
\end{lem}
\begin{proof}
The relation~\eqref{eqn-energy-mono} is immediate from the fact that discrete harmonic functions minimize Dirichlet energy subject to specified boundary data, so we only need to prove~\eqref{eqn-square-energy}.
For an edge $  \{ H,H' \} \in \mcl E\mcl H(S_k)$, we have $H \cap H' \not= \emptyset$ and so
\eqbn
|\phi_0(H ) - \phi_0(H')|^2 \leq  \left( \op{diam}(H  ) + \op{diam}(H') \right)^2 \leq 2 \op{diam}(H )^2 + 2\op{diam}(H')^2 .
\eqen
Breaking up the sum over edges into a sum over vertices, we get
\alb
\op{Energy}\left( \phi_0 ; \mcl H(S_k) \right) 
\leq 2 \sum_{\{H,H'\} \in \mcl E\mcl H(S_k) } \frk c(H,H') \left( \op{diam}(H)^2 + \op{diam}(H')^2 \right) 
\leq 4 \sum_{H \in \mcl H(S_k) }  \op{diam}(H )^2 \pi(H) . 
\ale
We now conclude by means of Lemma~\ref{lem-cell-sum}.
\end{proof}

For $m_1,m_2 \in \BB N_0$ with $m_1 < m_2$, let
\eqbn
\phi_{m_1,m_2} := \phi_{m_2} - \phi_{m_1} .
\eqen
We define
\eqb \label{eqn-se-def}
\op{SE}_{m_1,m_2} :=   \sum_{\substack{H\in\mcl H(\wh S_{m_2} ) \\ H\sim H_0}} \frac{ \frk c(H_0,H) | \phi_{m_1,m_2}(H)  - \phi_{m_1,m_2}(H_0)   |^2}{2\op{Area}(H_0 \cap \wh S_{m_2} )}  .
\eqe 
We can think of $\op{SE}_{m_1,m_2}$ as representing the ``Dirichlet energy per unit area" of $\phi_{m_1,m_2}$, so that $\BB E[\op{SE}_{m_1,m_2}]$ is the ``specific Dirichlet energy" of $\phi_{m_1,m_2}$. The following lemma makes this notion precise.

\begin{lem} \label{lem-se-energy}
For each fixed $m_1,m_2 \in \BB N_0$ with $m_1<m_2$, almost surely
\eqb \label{eqn-se-energy}
\lim_{n \rta\infty} \frac{1}{|\wh S_n|^2} \op{Energy}\left( \phi_{m_1,m_2} ; \mcl H( \wh S_n )  \right) = \BB E\left[ \op{SE}_{m_1,m_2} \right]  ,
\eqe 
and this expectation is bounded above by a finite constant which does not depend on $m_1,m_2$. 
\end{lem}
\begin{proof}
For $z\in\BB C$, let $\op{SE}_{m_1,m_2}^z$ be defined as in~\eqref{eqn-se-def} with the translated cell configuration/dyadic system $(\mcl H-z,\mcl D-z)$ in place of $(\mcl H,\mcl D)$. Equivalently,
\eqb \label{eqn-se-def-z}
\op{SE}_{m_1,m_2}^z :=   \sum_{\substack{H\in\mcl H( \wh S_{m_2}^z ) \\ H\sim H_z}} \frac{ \frk c(H_z,H) | \phi_{m_1,m_2}(H)  - \phi_{m_1,m_2}(H_z)    |^2}{2\op{Area}(H_z \cap \wh S_{m_2}^z)}  .
\eqe 
Lemma~\ref{lem-ergodic-avg} implies that a.s.\
\eqb \label{eqn-se-lim}
\lim_{n\rta\infty} \frac{1}{|\wh S_n|^2} \int_{\wh S_n} \op{SE}_{m_1,m_2}^z   \,dz = \BB E\left[ \op{SE}_{m_1,m_2} \right] .
\eqe 
We will now argue that this integral is exactly equal to $\op{Energy}\left( \phi_{m_1,m_2} ; \mcl H(\wh S_n)  \right)$ for $n\geq m_2$. 

We first consider the case $n=m_2$. By breaking up the integral over $\wh S_{m_2}$ into integrals over $H\cap \wh S_{m_2}$ for $H\in\mcl H(\wh S_{m_2})$, we find that 
\allb \label{eqn-se-sum0}
\int_{\wh S_{m_2}} \op{SE}_{m_1,m_2}^z  \,dz
&= \frac12 \sum_{H\in \mcl H(\wh S_{m_2} )  }   \sum_{\substack{H'\in \mcl H(\wh S_{m_2} ) \\ H'\sim H}} \frk c(H,H') |\phi_{m_1,m_2} (H') -  \phi_{m_1,m_2}(H)|^2  \notag \\ 
&=   \op{Energy}\left( \phi_{m_1,m_2} ; \mcl H(\wh S_{m_2} ) \right) .
\alle
For $n\geq m_2$, we can break up the integral over $\wh S_n$ as a sum of integrals over the squares of the form $\wh S_{m_2}^z$ for $z\in \wh S_n$ to get
\allb \label{eqn-se-sum}
\int_{\wh S_n} \op{SE}_{m_1,m_2}^z  \,dz 
 =   \sum_{  \wh S \in \mcl S_{m_2} : \wh S\subset \wh S_n }   \op{Energy}\left( \phi_{m_1,m_2} ; \mcl H(\wh S  ) \right) .
\alle
The squares $\wh S \in \mcl S_{m_2}$ with $\wh S\subset \wh S_n$ intersect only along their boundaries and their union is all of $\wh S_n$. Therefore, the sum on the right side of~\eqref{eqn-se-sum} is equal to the sum of $\frk c(H,H') |\phi_{m_1,m_2}(H') - \phi_{m_1,m_2}(H)|^2$ over all edges $\{H,H'\} \in \mcl E\mcl H(\wh S_n)$ except that edges are counted with multiplicity equal to the number of such squares $\wh S$ for which $\{H,H'\} \in \mcl E\mcl H(\wh S)$. Since $H\sim H'$ implies that $H\cap H'\not=\emptyset$, any edge $\{H,H'\}$ which is counted more than once must have $H,H' \in \mcl E\mcl H(\bdy \wh S)$ for some $\wh S \in \mcl S_{m_2}$. Both $\phi_{m_1}$ and $\phi_{m_2}$ coincide with $\phi_0$ on the boundary of each such square $\wh S$, so $|\phi_{m_1,m_2}(H') - \phi_{m_1,m_2}(H)|^2 = 0$ for each such edge $\{H,H'\}$. Therefore,~\eqref{eqn-se-sum} equals $\op{Energy}\left( \phi_{m_1,m_2} ; \mcl H(\wh S_n  ) \right) $. Combining this with~\eqref{eqn-se-lim} shows that~\eqref{eqn-se-energy} holds. The limiting value is uniformly bounded above by Lemma~\ref{lem-square-energy} and the triangle inequality for the Dirichlet norm.  
\end{proof}

We now argue that the specific Dirichlet energy $\BB E[\op{SE}_{m,m'}]$ has to be small when $m$ and $m'$ are large. Essentially, the reason for this is that the Dirichlet energy of $\phi_0$ gets ``used up" in the increments corresponding to indices smaller than $m$ and $m'$.

\begin{lem} \label{lem-tail-energy}
For each $\ep > 0$, there exists $m_\ep \in \BB N$ such that 
\eqb \label{eqn-tail-energy} 
\BB E\left[ \op{SE}_{m_\ep , m}  \right] \leq \ep ,\quad \forall m \geq m_\ep.
\eqe
\end{lem}
\begin{proof}
For each $k \in \BB N$, the function $\phi_{m-1,m}$ vanishes on $\mcl H\left(\bigcup_{\wh S \in \mcl S_m} \bdy \wh S  \right)$ (since both $\phi_m$ and $\phi_{m-1}$ agree with $\phi_0$ there) and is discrete harmonic on $\mcl H \setminus \mcl H\left(\bigcup_{\wh S \in \mcl S_m} \bdy \wh S  \right)$. Therefore, the discrete analog of Green's formula (summation by parts) implies that $\phi_{m-1,m}$ and $\phi_{m'-1,m'}$ are orthogonal with respect to the Dirichlet inner product for $m\not=m'$. Hence for $m\in\BB N$ and $n\geq m$, 
\eqb \label{eqn-energy-decomp}
\op{Energy}\left( \phi_{0,m} ; \mcl H(\wh S_n )  \right) = \sum_{j=1}^m \op{Energy}\left( \phi_{j-1,j} ; \mcl H(\wh S_n)  \right).
\eqe  
Dividing both sides of~\eqref{eqn-energy-decomp} by $|\wh S_n|^2$, sending $n\rta\infty$, and applying Lemma~\ref{lem-se-energy} shows that
\eqb \label{eqn-se-decomp}
\BB E\left[ \op{SE}_{0,m} \right] = \sum_{j=1}^m \BB E\left[ \op{SE}_{j-1,j} \right] .
\eqe 
By the last part of Lemma~\ref{lem-se-energy}, this last sum is uniformly bounded above, independently of $m$. Therefore, $\sum_{j=1}^\infty \BB E\left[ \op{SE}_{j-1,j} \right]
 < \infty$ so for any $\ep > 0$ we can find $m_\ep \in \BB N$ such that 
\eqb \label{eqn-tail-energy0} 
\BB E\left[ \op{SE}_{m_\ep , m}  \right]  \leq \sum_{j=m_\ep+1}^\infty \BB E\left[ \op{SE}_{j-1,j} \right] \leq \ep ,\quad \forall m \geq m_\ep .
\eqe
\end{proof}

\begin{proof}[Proof of Proposition~\ref{prop-corrector}] 
We will show that $\phi_m-\phi_m(H_0)$ is Cauchy in probability with respect to the local uniform topology on $\mcl H$ using Lemma~\ref{lem-tail-energy}. 
To this end, fix $\el \in \BB N$; we will work on $\wh S_\el$. 
Given $\ep  >0$, let $m_\ep \in \BB N$ be as in Lemma~\ref{lem-tail-energy} and define $\op{SE}_{m_\ep,m}^z$ as in~\eqref{eqn-se-def-z}. Then if $m\geq \el \vee m_\ep$, 
\allb \label{eqn-corrector-sum}
\int_{\wh S_\el} \op{SE}_{m_\ep,m}^z \,dz 
&= \frac12 \sum_{H\in \mcl H(\wh S_\el )  }  \frac{\op{Area}(H\cap \wh S_\el)}{\op{Area}(H\cap \wh S_m)}  \sum_{\substack{H'\in \mcl H(\wh S_m ) \\ H'\sim H}} \frk c(H,H') |\phi_{m_\ep, m} (H') -  \phi_{m_\ep, m}(H)|^2  \notag  \\
&\geq \frac12 \sum_{H\in \mcl H(\wh S_\el ) \setminus \mcl H(\bdy \wh S_\el) }    \sum_{\substack{H'\in \mcl H(\wh S_m)  \\ H'\sim H}} \frk c(H,H') |\phi_{m_\ep,m} (H') -  \phi_{m_\ep,  m}(H)|^2  \notag  \\
&\geq \op{Energy}\left( \phi_{m_\ep, m}  ;    \mcl H(\wh S_\el ) \setminus \mcl H(\bdy \wh S_\el)    \right)   .
\alle 

Dividing by $|\wh S_\el|^2$ in~\eqref{eqn-corrector-sum}, taking expectations, and applying Lemma~\ref{lem-dyadic-resample} and our choice of $m_\ep$ to bound the expectation of the integral over $\wh S_\el$, we get that  
\eqb \label{eqn-diff-energy}
\sup_{m\geq m_\ep} \BB E\left[ \frac{1}{|\wh S_\el|^2 } \op{Energy}\left( \phi_{m_\ep, m}  ;    \mcl H(\wh S_\el ) \setminus \mcl H(\bdy \wh S_\el)    \right) \right] 
\leq  \ep   . 
\eqe
The variation of $\phi_{m_\ep , m}$ over any single edge $\{H,H'\}$ of $\mcl H(\wh S_\el ) \setminus \mcl H(\bdy \wh S_\el)$ is bounded above by $\frk c(H,H')^{-1/2}$ times the square root of the energy of $\phi_{m_\ep,m}$ over $\mcl H(\wh S_\el ) \setminus \mcl H(\bdy \wh S_\el)$. Summing over all such edges and applying~\eqref{eqn-diff-energy} shows that for each fixed $\el\in\BB N$, 
\eqb  \label{eqn-max-component}
  \sup_{m \geq m_\ep } \BB E\left[  \frac{ \max_{H\in \rng{\mcl H}(\wh S_\el)} \left| \phi_{m_\ep , m}(H) - \phi_{m_\ep , m}(H_0) \right| }{ |\wh S_\el|  \sum_{\{H,H'\} \in \mcl E\rng{\mcl H}(\wh S_\el)  } \frk c(H,H')^{-1/2}}  \right] \leq \ep^{1/2} 
\eqe 
where $\rng{\mcl H}(\wh S_\el)$ denotes the connected component of $\mcl H(\wh S_\el ) \setminus \mcl H(\bdy \wh S_\el)  $ which contains $H_0$. 
Since $\ep > 0$ is arbitrary and $\el$ is fixed, the sequence $\{\phi_m-\phi_m(H_0)\}_{m\in\BB N}$ is Cauchy in probability with respect to the topology of uniform convergence on $\rng{\mcl H}(\wh S_\el)$.  
By the local finiteness of $\mcl H$, a.s.\ each finite subgraph of $\mcl H$ is contained in $\rng{\mcl H}(\wh S_\el)$ for large enough $\el\in\BB N$. 
Therefore, there is a function $\phi_\infty: \mcl H \rta \BB C$ such that~\eqref{eqn-corrector-conv} holds. 

The function $\phi_\infty$ is a.s.\ discrete harmonic since each $\phi_m$ is discrete harmonic on $\mcl H(\wh S_m) \setminus \mcl H(\bdy \wh S_m)$ and the pointwise limit of discrete harmonic functions is discrete harmonic. The scaling/translation property of $\phi_\infty$ follows from the analogous property for the $\phi_m$'s. 
\end{proof}

\subsection{Sublinearity of the corrector} 
\label{sec-corrector-sublinearity}

In this subsection we prove~\eqref{eqn-cell-tutte-conv} of Theorem~\ref{thm-general-tutte} for the discrete harmonic function $\phi_\infty$ of Proposition~\ref{prop-corrector} (this corresponds to Theorem~\ref{thm-general-tutte} in the special case when $\mcl H$ is the planar dual of an embedded lattice). We continue to use the notation of Section~\ref{sec-corrector-existence}. 
We will first argue that the Dirichlet energy per unit area of $\phi_\infty-\phi_m$ is small when $m$ is large (Lemma~\ref{lem-energy-lim-infty}), which is a consequence of Lemmas~\ref{lem-ergodic-avg} and~\ref{lem-tail-energy}. We will then show (essentially using H\"older's inequality) that this implies that the total variation of $\phi_\infty - \phi_m$ over most horizontal and vertical line segments of the square $S_k$ is of smaller order than $|S_k|$ when $m$ is large (but fixed) and $k$ is very large (Lemma~\ref{lem-path-sum}). It is easily seen from the definition of $\phi_m$ that for any fixed $m\in\BB N$, the maximum value of $|\phi_m - \phi_0|$ over $S_k$ is $o_k(|S_k|)$ as $k\rta\infty$ (Lemma~\ref{lem-m-0-diff}). We will then conclude by combining our bounds for $\phi_\infty-\phi_m$ and $\phi_m - \phi_0$ and using the maximum principle for discrete harmonic functions to transfer from a bound for ``most" horizontal and vertical line segments to a bound which holds for all points simultaneously (Lemma~\ref{lem-var-lim-infty}).

Extending the definition~\eqref{eqn-se-def}, for $m\in\BB N$ we set $\phi_{m,\infty} := \phi_\infty-\phi_m$ and define 
\eqb \label{eqn-se-def-infty}
\op{SE}_{m,\infty} := \sum_{\substack{H\in\mcl H  \\ H\sim H_0}} \frac{ \frk c(H_0,H) |\phi_{m ,\infty }(H)  - \phi_{m, \infty }(H_0)   |^2}{2\op{Area}(H_0   )}  .
\eqe  

\begin{lem} \label{lem-se-infty}
We have $\BB E[S_{m,\infty}] <\infty$ for each $m\in\BB N_0$ and
\eqb \label{eqn-se-infty} 
\lim_{m\rta\infty} \BB E\left[ S_{m,\infty} \right] = 0 .
\eqe 
\end{lem}
\begin{proof}
It is clear from Proposition~\ref{prop-corrector} that for each fixed $m\in\BB N$, we have $  \op{SE}_{m,m'} \rta \op{SE}_{m,\infty}$ in probability as $m'\rta\infty$. By Fatou's lemma,  
\eqbn
\BB E\left[ \op{SE}_{m,\infty} \right] \leq \liminf_{m'\rta\infty} \BB E\left[\op{SE}_{m,m'}  \right] ,
\eqen
which is finite by Lemma~\ref{lem-se-energy}. 
By Lemma~\ref{lem-tail-energy}, $\lim_{m\rta\infty}\liminf_{m'\rta\infty} \BB E\left[\op{SE}_{m,m'}  \right] = 0$. 
\end{proof}

From Lemma~\ref{lem-se-infty} we can deduce the following. 

\begin{lem} \label{lem-energy-lim-infty}
Almost surely, 
\eqb \label{eqn-energy-lim-infty}
\lim_{m\rta\infty} \limsup_{n\rta\infty}   \frac{1}{|\wh S_n|^2} \op{Energy}\left( \phi_{m,\infty} ; \mcl H(\wh S_n) \right) = 0 .
\eqe 
\end{lem}
\begin{proof}
For $z\in\BB C$, let $\op{SE}_{m,\infty}^z $ be defined as in~\eqref{eqn-se-def-infty} with $H_z$ in place of $H_0$. Equivalently, $\op{SE}_{m,\infty}^z$ is defined in the same manner as $\op{SE}_{m,\infty}$ but with $(\mcl H-z , \mcl D-z )$ in place of $(\mcl H,\mcl D)$. 
By Lemma~\ref{lem-ergodic-avg}, for each fixed $m\in\BB N$, a.s.\ 
\eqb \label{eqn-se-int-infty}
\lim_{n \rta\infty} \frac{1}{|\wh S_n|^2} \int_{\wh S_n} \op{SE}_{m,\infty}^z \, dz  = \BB E\left[ \op{SE}_{m,\infty} \right] .
\eqe 
As in~\eqref{eqn-corrector-sum}, for $m,n\in\BB N$, 
\allb
\int_{\wh S_n} \op{SE}_{m,\infty}^z \, dz 
&= \frac12 \sum_{H\in \mcl H(\wh S_n)} \frac{\op{Area}(H\cap \wh S_n)}{\op{Area}(H)} \sum_{\substack{H'\in\mcl H \\ H'\sim H}} \frk c(H,H') |\phi_{m,\infty}(H') - \phi_{m,\infty}(H)|^2 \notag\\ 
&\geq \op{Energy}\left( \phi_{m,\infty} ; \mcl H(\wh S_n) \setminus \mcl H(\bdy \wh S_n) \right) .
\alle
Combining this with~\eqref{eqn-se-int-infty} shows that a.s.\ 
\eqb \label{eqn-se-interior-infty}
\limsup_{n\rta\infty} \frac{1}{|\wh S_n|^2} \op{Energy}\left( \phi_{m,\infty} ; \mcl H(\wh S_n) \setminus \mcl H(\bdy \wh S_n) \right) \leq  \BB E\left[ \op{SE}_{m,\infty} \right]  .
\eqe 
Lemma~\ref{lem-max-cell-diam} implies that a.s.\ for large enough $n\in\BB N$, each cell in $\mcl H(\wh S_n)$ has diameter at most $(1/2) |\wh S_n|$, and hence is contained in the union of the four dyadic parents of $\wh S_n$. 
Furthermore,~\eqref{eqn-se-interior-infty} combined with Lemma~\ref{lem-more-squares} implies that a.s.\ for large enough $n\in\BB N$, the relation~\eqref{eqn-se-interior-infty} holds with $\wh S_n$ replaced by any of the four dyadic parents of $\wh S_n$. Consequently, a.s.\ 
\eqb \label{eqn-se-interior-infty'}
\limsup_{n\rta\infty}\frac{1}{|\wh S_n|^2}  \op{Energy}\left( \phi_{m,\infty} ; \mcl H(\wh S_n)  \right) \leq 4\BB E\left[ \op{SE}_{m,\infty} \right]   
\eqe 
which together with Lemma~\ref{lem-se-infty} implies~\eqref{eqn-energy-lim-infty}.
\end{proof}

Lemma~\ref{lem-energy-lim-infty} gives us an upper bound for the discrete Dirichlet energy of $\phi_\infty-\phi_m$ when $m$ is large. The following lemma allows us to transfer this to a pointwise bound.

\begin{lem} \label{lem-path-sum}
There is a deterministic constant $C   >0$ such that the following is true a.s. For $k\in\BB N$ and $r \in [0,1]$, let $P_k(r)$ be the horizontal line segment connecting the left and right boundaries of the square $S_k$ whose distance from the lower boundary of $S_k$ is $r |S_k|$. 
Then for each large enough $k\in\BB N$ and each function $f : \mcl H(S_k ) \rta \BB C$, 
\eqb \label{eqn-path-sum}
 \int_0^1 \sum_{\{H,H'\} \in \mcl E\mcl H( P_k(r) )} |f(H ) - f(H')| \, dr  \leq C  \op{Energy}\left( f ;\mcl H(S_k) \right)^{1/2} .
\eqe 
The same is true if we replace horizontal line segments by vertical line segments. 
\end{lem}
\begin{proof} 
By Lemma~\ref{lem-cell-sum}, there is a deterministic constant $C_0 > 0$ such that a.s.\ for sufficiently large $k\in\BB N$, 
\eqb \label{eqn-use-cell-sum}
 \frac{1}{|S_k |^2} \sum_{H\in \mcl H(S_k  )} \op{diam}(H)^2 \pi^*(H) \leq C_0  ,
\eqe  
where here we recall the definition of $\pi^*$ from~\eqref{eqn-stationary-measure}. 
Assume now that~\eqref{eqn-use-cell-sum} holds. For $k\in\BB N$ and an edge $\{H,H'\} \in \mcl E\mcl H(  S_k    )$, let $L_k(H,H')$ be the Lebesgue measure of the set of $r\in [0,1]$ for which $H$ and $H'$ both belong to $\mcl H(  P_k(r)  )$. Then, by interchanging integration and summation, we get
\eqb \label{eqn-path-int0}
\int_{0}^{1} \sum_{\{H,H'\} \in \mcl E\mcl H( P_k(r) )} |f(H ) - f(H')| \, dr 
= \sum_{\{H,H'\} \in \mcl E\mcl H(S_k)  } L_k(H,H') |f(H) - f(H')| .
\eqe 
The quantity $L_k(H,H')$ is bounded above by the Lebesgue measure of the set of $r$ for which $H \in \mcl H(P_k(r))$, so
\eqb \label{eqn-path-diam}
L_k(H , H') \leq \frac{2}{|S_k|} \op{diam}\left( H  \right)  .
\eqe 
By~\eqref{eqn-path-int0},~\eqref{eqn-path-diam}, and the Cauchy-Schwarz inequality, we see that if~\eqref{eqn-use-cell-sum} holds then
\allb
&\int_0^1 \sum_{\{H,H'\} \in \mcl E\mcl H(P_k(r) )} |f(H ) - f(H')| \, dr 
 \leq \frac{2}{|S_k|} \sum_{\{H,H'\} \in \mcl H(S_k) } \op{diam}(H)   |f(H) - f(H')| \notag \\
&\qquad \qquad \leq 2 \left( \frac{1}{|S_k|^2} \sum_{\{H,H'\} \in \mcl H(S_k) } \op{diam}(H)^2 \frk c(H,H')^{-1} \right)^{1/2} 
\left(  \sum_{\{H,H'\} \in \mcl H(S_k) }  \frk c(H,H') |f(H)-f(H')|^2 \right)^{1/2} \notag \\
&\qquad \qquad \leq 2 \left( \frac{1}{|S_k|^2} \sum_{H\in \mcl H(S_k  ) } \op{diam}(H)^2 \pi^*(H) \right)^{1/2}   \op{Energy}\left( f ; \mcl H(S_k) \right)^{1/2} \notag \\
&\qquad \qquad \leq 2 C_0^{1/2} \op{Energy}\left( f ; \mcl H(S_k)  \right)^{1/2}  .
\alle
This concludes the proof with $C =  2 C_0^{1/2}$. Obviously, the case of vertical line segments follows from the same argument. 
\end{proof}

The preceding lemmas allow us to approximate $\phi_\infty$ by $\phi_m$ for a large value of $m\in\BB N$. 
To complement this, we have the following bound for $\phi_m-\phi_0$ for a fixed value of $m$.

\begin{lem} \label{lem-m-0-diff}
For each fixed $m\in\BB N$, a.s.\
\eqb \label{eqn-m-0-diff}
\limsup_{k \rta\infty}  \frac{1}{|S_k|}  \max_{H \in \mcl H(S_k)} \left| \phi_m(H) - \phi_0(H)  \right|      = 0 .
\eqe 
\end{lem}
\begin{proof}
By Lemma~\ref{lem-max-cell-diam}, a.s.\ 
\eqb \label{eqn-max-cell-diam-diff}
\limsup_{k\rta\infty} \frac{1}{|S_k|} \max_{H\in \mcl H(S_k)} \op{diam}(H) = 0 .
\eqe 
By~\eqref{eqn-mass-time-def}, the interior of each $\wh S\in \mcl S_m$ contains at most $m$ cells of $\mcl H$. Consequently, the center point of $\wh S$ can be joined to $\bdy \wh S$ by a Euclidean path which intersects at most $m+1$ cells of $\mcl H(\wh S)$. Since $\wh S$ is a square, it follows that the side length of $\wh S$ is at most $m+1$ times the maximal diameter of the cells which intersect $\wh S$, so we can multiply the estimate~\eqref{eqn-max-cell-diam-diff} by $m+1$ to get that a.s.\ 
\eqb \label{eqn-max-square-diam}
\limsup_{k\rta\infty} \frac{1}{|S_k|} \max_{ \wh S \in \mcl S_m : \wh S \subset S_k } \op{diam}(\wh S) = 0 .
\eqe 
By definition, $\phi_0$ and $\phi_m$ agree on $\mcl H(\bdy \wh S)$ for each $\wh S \in \mcl S_m$. Moreover, $\phi_m$ is discrete harmonic on $\mcl H(\wh S) \setminus \mcl H(\bdy \wh S)$ for each such square $\wh S$. By combining~\eqref{eqn-max-cell-diam-diff} and~\eqref{eqn-max-square-diam} with the maximum principle for $\phi_m|_{\mcl H(\wh S)}$ and the fact that $|\phi_0(H) - z| \leq \op{diam}(H)$ for each $z\in H\in \mcl H$, we obtain~\eqref{eqn-m-0-diff}.
\end{proof}

\begin{lem}  \label{lem-var-lim-infty}
Almost surely, 
\eqb \label{eqn-var-lim-infty}
 \limsup_{k \rta\infty}  \frac{1}{|S_k|} \max_{H \in \mcl H(S_k)} \left|  \phi_\infty(H)  - \phi_0(H)   \right|      = 0 .
\eqe 
\end{lem}
\begin{proof}
For $k\in\BB N$, define the horizontal line segments $P_k(r) \subset S_k$ for $r\in[0,1]$ as in Lemma~\ref{lem-path-sum}. 
By Lemmas~\ref{lem-energy-lim-infty} and~\ref{lem-path-sum}, a.s.\ 
\eqb \label{eqn-var-lim-line}
\lim_{m\rta\infty} \limsup_{k\rta\infty} \frac{1}{|S_k|} \int_0^1 \sum_{\{H,H'\} \in \mcl E\mcl H( P_k(r) )} |\phi_{m,\infty}(H ) - \phi_{m,\infty}(H')| \, dr  = 0.
\eqe 
Given $\ep > 0$, choose a (random) $m_\ep \in \BB N$ sufficiently large that for $m\geq m_\ep$, the limsup in~\eqref{eqn-var-lim-line} is at most $\ep^3/100$. 
By Markov's inequality, a.s.\ for $m\geq m_\ep$ it holds for large enough $k\in\BB N$, the Lebesgue measure of the set of $r \in [0,1]$ for which the corresponding sum in~\eqref{eqn-var-lim-line} is at most $|S_k| \ep^2/100$ is at least $1-\ep$. Therefore, for large enough $k\in\BB N$ we can find a finite collection $\mcl P_k$ of at most $2/\ep$ horizontal line segments joining the left and right boundaries of $S_k$ such that each point of $S_k$ lies within Euclidean distance $\ep |S_k|$ of one of these line segments and
\eqb \label{eqn-segment-var}
\frac{1}{|S_k|} \sum_{P \in \mcl P_k} \sum_{\{H,H'\} \in \mcl E\mcl H( P )} |\phi_{m_\ep ,\infty}(H ) - \phi_{m_\ep ,\infty}(H')| \leq \frac{\ep}{2} .
\eqe 
Similarly, after possibly increasing $m_\ep$ it is a.s.\ the case that for large enough $k\in\BB N$, we can find an $\ep|S_k|$-dense collection $\mcl P_k'$ of at most $2/\ep$ vertical line segments joining the top and bottom boundaries of $S_k$ such that~\eqref{eqn-segment-var} holds with $\mcl P_k'$ in place of $\mcl P_k$. 

By Lemma~\ref{lem-m-0-diff}, it is a.s.\ the case that for large enough $k\in\BB N$, we have $|S_k|^{-1} \max_{H\in \mcl H(S_k)} |\phi_{m_\ep}(S) - \phi_0(S)| \leq \ep/4$. Combining this with~\eqref{eqn-segment-var} and its analog for $\mcl P_k'$, we see that a.s.\ for large enough $k\in\BB N$, 
\eqb \label{eqn-segment-var0}
\frac{1}{|S_k|} \sum_{P \in \mcl P_k} \sum_{\{H,H'\} \in \mcl E\mcl H( P )} |\phi_{0 ,\infty}(H ) - \phi_{0 ,\infty}(H')| \leq \ep
\eqe
and the same holds with $\mcl P_k'$ in place of $\mcl P_k$. Each $P \in \mcl P_k \cup \mcl P_k'$ is a horizontal or vertical line segment, so assumption~\ref{item-hyp-adjacency} implies that $\mcl H(P)$ (viewed as a graph) is connected for each $P\in\mcl P_k\cup \mcl P_k'$, i.e., any two cells in $\mcl H(P)$ can be joined by a path of edges in $\mcl E\mcl H(P)$. Hence $\sum_{\{H,H'\} \in \mcl E\mcl H( P )} |\phi_{0 ,\infty}(H ) - \phi_{0 ,\infty}(H')|$ is an upper bound for the maximum of $|\phi_{0,\infty}(H_1) - \phi_{0,\infty}(H_2)|$ over all $H_1,H_2 \in \mcl H(P)$. By this and~\eqref{eqn-segment-var0}, we get that a.s.\ for each large enough $k\in\BB N$, 
\eqb\label{eqn-segment-var1} 
\frac{1}{|S_k|} \max\left\{ |\phi_{0,\infty}(H_1) - \phi_{0,\infty}(H_2)| : H_1, H_2 \in \mcl H\left( \bigcup_{P\in\mcl P_k\cup \mcl P_k'} P \right) \right\} \leq 2\ep.
\eqe   

We now want to extend from~\eqref{eqn-segment-var1} to a bound which holds for all $H_1,H_2 \in \mcl H(S_k)$. Recall that $\phi_{0,\infty} = \phi_\infty - \phi_0$. 
Since the collections $\mcl P_k$ and $\mcl P_k'$ are $\ep|S_k|$-dense, a.s.\ for large enough $k\in\BB N$ each $H\in\mcl H(S_k)$ which is not contained in the Euclidean $\ep |S_k|$-neighborhood of $\bdy S_k$ intersects a rectangle of diameter at most $2\ep |S_k|$ whose boundary is formed by segments of four elements of $\mcl P_k \cup \mcl P_k'$. Since $\phi_\infty$ is discrete harmonic, it attains its maximum and minimum values on this rectangle on the boundary of the rectangle. Furthermore, by Lemma~\ref{lem-max-cell-diam} and the definition of $\phi_0$, the maximum and minimum values of $\phi_0$ on each such rectangle differ from the maximum and minimum values of $\phi_0$ on the boundary of the rectangle by at most a quantity which is of order $o_k(|S_k|)$ a.s.\ as $k\rta\infty$, uniformly over all of the rectangles. Combining these facts with~\eqref{eqn-segment-var1} shows that a.s.\ for each large enough $k\in\BB N$, $|\phi_{0,\infty}(H_1) - \phi_{0,\infty}(H_2)| \leq 4\ep |S_k|$ for each $H_1,H_2\in\mcl H(S_k)$ which are not contained in the Euclidean $\ep |S_k|$-neighborhood of $\bdy S_k$.

We will now deal with the cells $H\in\mcl H(S_k)$ which are close to $\bdy S_k$. 
Let $S_k'$ be a dyadic ancestor of $S_k$ of side length $4|S_k|$ which contains $S_k$ in its interior. 
By the preceding paragraph and Lemma~\ref{lem-more-squares}, it is a.s.\ the case for each large enough $k\in\BB N$, $|\phi_{0,\infty}(H_1) - \phi_{0,\infty}(H_2)| \leq 16 \ep |S_k|$ for each $H_1,H_2 \in\mcl H(S_k')$ which are not contained in the Euclidean $\ep |S_k'|$-neighborhood of $\bdy S_k'$. 
Since $S_k\subset S_k'$ and $S_k$ lies at Euclidean distance at least $\ep|S_k'|$ from $\bdy S_k'$, we get that a.s.\ for each each large enough $k\in\BB N$, $|\phi_{0,\infty}(H_1) - \phi_{0,\infty}(H_2)| \leq 16 \ep |S_k|$ for each $H\in\mcl H(S_k)$. 

We have 
\eqbn
| \phi_{0,\infty}(H_0)| =  |\phi_0(H_0)|  = \frac{1}{\op{Area}(H_0)} \left| \int_{H_0} z \,dz \right| \leq \op{diam}(H_0). 
\eqen 
From this and the preceding paragraph (applied with one of $H_1$ or $H_2$ equal to $H_0$), we get that a.s.\ for each each large enough $k\in\BB N$, $|\phi_\infty(H) - \phi_0(H)| \leq 17\ep |S_k|$ for each $H_1,H_2\in\mcl H(S_k)$ which are not contained in the Euclidean $\ep |S_k|$-neighborhood of $\bdy S_k$.
Since $\ep > 0$ is arbitrary, this implies~\eqref{eqn-var-lim-infty}.
\end{proof}

\begin{proof}[Proof of~\eqref{eqn-cell-tutte-conv} of Theorem~\ref{thm-general-clt}]
Lemmas~\ref{lem-more-squares} and~\ref{lem-var-lim-infty} together imply that for each $\ep> 0$, it a.s.\ holds for large enough $k\in\BB N$ that $\left|  \phi_\infty(H)  - \phi_0(H)   \right|  \leq \ep |S_k|$ for each cell $H$ which intersects one of the four dyadic parents of $S_k$. Since $0\in S_k$, the union of these four dyadic parents contains $B_{|S_k|}(0)$. Since $|S_k| \rta \infty$ as $k\rta\infty$, this implies~\eqref{eqn-general-tutte}.  
\end{proof}

Although it is not needed in the proof of~\eqref{eqn-cell-tutte-conv}, we record here the following lemma which will be used in Section~\ref{sec-martingale-clt}.

\begin{lem} \label{lem-infty-energy}
The function $\phi_\infty$ has finite specific Dirichlet energy, i.e., 
\eqbn
\BB E\left[ \sum_{\substack{H\in\mcl H  \\ H\sim H_0}} \frac{ \frk c( H_0 , H) |\phi_{ \infty }(H)     |^2}{ \op{Area}(H_0   )} \right] < \infty .
\eqen
\end{lem}
\begin{proof}
By Lemma~\ref{lem-se-infty} we have $\BB E[\op{SE}_{0,\infty}] < \infty$ so we only need to show that  
\eqb \label{eqn-se0}
\BB E[\op{SE}_0] < \infty \quad\text{where} \quad \op{SE}_0 := \sum_{\substack{H\in\mcl H  \\ H\sim H_0}} \frac{  \frk c( H_0 , H) |\phi_{0}(H)  - \phi_{ 0 }(H_0)   |^2}{ \op{Area}(H_0   )} .
\eqe 
To prove~\eqref{eqn-se0}, we use Lemma~\ref{lem-ergodic-avg} to get that a.s.\
\eqb \label{eqn-se0-int}
\BB E[\op{SE}_0] = \lim_{k\rta\infty} \frac{1}{|S_k|^2} \int_{S_k} \sum_{\substack{H\in\mcl H  \\ H\sim H_z}} \frac{  \frk c( H_z,  H) |\phi_{0}(H)  - \phi_{ 0 }(H_z)   |^2}{ \op{Area}(H_z   )} \,dz .
\eqe 
By breaking up the integral into a sum over the regions $H \cap S_k$ and applying Lemma~\ref{lem-max-cell-diam}, we find that the integral on the right side of~\eqref{eqn-se0-int} is a.s.\ bounded above by twice the Dirichlet energy of $\phi_0$ over the union of the four dyadic parents of $S_k$ for large enough $k\in\BB N$. The bound~\eqref{eqn-se0} therefore follows from~\eqref{eqn-se0-int} combined with Lemmas~\ref{lem-more-squares} and~\ref{lem-square-energy}. 
\end{proof}

\subsection{Random walk on $\mcl H$ is recurrent}
\label{sec-recurrence}

In this subsection we will establish the following, which is part of Theorem~\ref{thm-general-clt} and is a generalization of Theorem~\ref{thm-recurrent}. 

\begin{prop} \label{prop-cell-recurrent}
In the setting of Theorem~\ref{thm-general-clt}, the simple random walk on $\mcl H$ is a.s.\ recurrent. 
\end{prop}

 The proof of Proposition~\ref{prop-cell-recurrent} uses only the results of Section~\ref{sec-ergodic-avg}. 
By a standard criterion, it suffices to produce a sequence of finite vertex sets $\{V_r\}_{r\in\BB N} \subset \mcl H$ containing $H_0$, whose union is all of $\mcl H$, with the property that the effective resistance from 0 to $\mcl H\setminus V_r$ tends to $\infty$ as $r\rta\infty$. By Dirichlet's principle (see, e.g.,~\cite[Exercise 2.13]{lyons-peres}), to prove this it suffices to find functions $\mcl H \rta [0,1]$ which are identically equal to 1 on $\mcl H\setminus V_r$ and vanish at $H_0$ whose discrete Dirichlet energy tends to 0 as $r\rta\infty$. 

We will construct the sets $V_r$ and the corresponding functions using the dyadic system $\mcl D$. 
For $k  \in \BB Z$ and $r\in\BB N_0$, let $B_k^{(r)}$ be the Euclidean ball of side length $2^r | S_k|$ with the same center at $S_k$.
For $r \geq 2$, let $\frk g_k^{(r)} : B_k^{(r)} \setminus B_k^{(1)} \rta [0,1]$ be the continuous function which equals 0 on $\bdy B_k^{(1)}$, equals 1 on $\bdy B_k^{(r)}$, and is discrete harmonic on the rest of $B_k^{(r)} \setminus B_k^{(1)}$. That is,
\eqb \label{eqn-cont-log-function}
\frk g_k^{(r)}(z) = \frac{ \log |z - v_k| - \log (2 |S_k|)}{ \log( 2^{r-1} )}  ,\quad\forall z \in B_k^{(r)} \setminus B_k^{(1)},
\eqe 
where $v_k$ is the center of $S_k$. 
Also let $\frk f_k^{(r)} : \mcl H \rta [0,1]$ be the function which vanishes on $\mcl H(B_k^{(1)})$, equals 1 on $\mcl H\setminus \mcl H( B_k^{(1)})$, and for other cells $H\in\mcl H(B_k^{(r)} \setminus B_k^{(1)}) \setminus \mcl H(\bdy ( B_k^{(r)} \setminus B_k^{(1)}   ) ) $ satisfies $\frk f_k^{(r)}(H) = \frk g_k^{(r)}(z)$ for some (arbitrary) choice of $z\in H \cap (B_k^{(r)} \setminus B_k^{(1)} )$. By the discussion just above (applied with $V_r = \mcl H(B_{k_r}^{(r)})$ for a large random $k_r$), to prove Proposition~\ref{prop-cell-recurrent}, it suffices to establish the following.

\begin{lem} \label{lem-log-energy} 
There is a deterministic constant $C>0$ such that for each $r\in \BB N$, it a.s.\ holds for large enough $k\in\BB N$ that
\eqb \label{eqn-log-energy}
\op{Energy}\left( \frk f_k^{(r)} ; \mcl H   \right) \leq \frac{C}{r} .
\eqe 
\end{lem}
\begin{proof}
The proof is an extension of that of Lemma~\ref{lem-square-energy} with a few extra complications.
Each of the balls $B_k^{(r)}$ above is contained in the union of four dyadic ancestors of $S_k$ of side length $2^{r+1} |S_k|$. 
If therefore follows from Lemma~\ref{lem-more-squares} combined with Lemmas~\ref{lem-max-cell-diam} and~\ref{lem-cell-sum}, respectively, that there exists a deterministic constant $C_0 >0$ (independent of $r$) such that a.s.\ for large enough $k\in\BB N$,
\eqb \label{eqn-recurrent-max-cell}
\op{diam}(H) \leq \frac{1}{100}  |S_k | ,\quad \forall H\in \mcl H(B_k^{(r)}) 
\eqe
and
\eqb \label{eqn-recurrent-cell-sum}
\frac{1}{  2^{2j} |S_k|^2} \sum_{H\in \mcl H(B_k^{(j)})} \op{diam}(H)^2 \pi(H)  \leq C_0 , \quad \forall j \in [1,r]_{\BB Z} .
\eqe
Henceforth assume that $k$ is chosen sufficiently large that~\eqref{eqn-recurrent-max-cell} and~\eqref{eqn-recurrent-cell-sum} hold. 
 
We first observe that, by direct calculation and since $\log(2^{r-1} )\asymp r$, the function of~\eqref{eqn-cont-log-function} satisfies
\eqb \label{eqn-log-gradient}
|\nabla \frk g_k^{(r)}(z) | \preceq \frac{1}{r |z|    }   
\eqe 
with a universal implicit constant. 
Now consider $j\in [1,r]_{\BB Z}$ and an edge 
$\{H,H'\} \in \mcl E\mcl H\left( B_k^{(j)} \setminus B_k^{(j-1)}   \right)$. 
By the mean value theorem, for appropriate points $z\in H$ and $z'\in H'$, 
\eqb \label{eqn-recurrent-mean-value}
\left| \frk f_k^{(r)}(H) - \frk f_k^{(r)}(H') \right| 
\leq \int_0^1  \left| \nabla \frk g_k^{(r)}\left( P_{z,z'}(t) \right) \cdot P_{z,z'}(t) \right|  \, dt    
\eqe  
where $P_{z,z'}(t)$ denotes a Euclidean path from $z$ to $z'$ in $B_k^{(r)} \setminus B_k^{(1)}$ of minimal Euclidean length. 

Since $P_{z,z'}(t)$ has minimal Euclidean length, it follows that its total length is 
\eqbn
\int_0^1 |P_{z,z'}(t)| \, dt \leq 2\left( \op{diam}(H) + \op{diam}(H') \right) .
\eqen
Furthermore, by~\eqref{eqn-recurrent-max-cell} we have $|P_{z,z'}(t)| \asymp 2^j |S_k|$ for $t\in [0,1]$, so by~\eqref{eqn-log-gradient}  $|\nabla \frk g_k^{(r)}(P_{z,z'}(t))| \preceq r^{-1} 2^{-j} |S_k|^{-1}$. Plugging these estimates into~\eqref{eqn-recurrent-mean-value} gives  
\eqb \label{eqn-one-edge-gradient}
\left| \frk f_k^{(r)}(H) - \frk f_k^{(r)}(H') \right|  
\preceq  \frac{  \op{diam}(H) + \op{diam}(H')  }{  r  2^{ j} |S_k|   }  ,\quad \forall \{H,H'\} \in \mcl E\mcl H\left(  B_k^{(j)}\setminus B_k^{(j-1)} \right) .
\eqe 
Squaring both sides of~\eqref{eqn-one-edge-gradient}, multiplying by $\frk c(H,H')$, and summing over all such edges $\{H,H'\}$, then applying~\eqref{eqn-recurrent-cell-sum} gives
\eqbn
\op{Energy}\left( \frk f_k^{(r)} ;  \mcl H\left( B_k^{(j)} \setminus B_k^{(j-1)}  \right)   \right)
\leq \frac{1}{r^2 2^{2j} |S_k|^2} \sum_{H\in \mcl H(B_k^{(j)})} \op{diam}(H)^2 \pi(H)
\preceq \frac{1}{r^2} 
\eqen
with a deterministic, $r$-independent implicit constant. Summing this over all $j \in [1, r]_{\BB Z}$ now yields~\eqref{eqn-log-energy} since $\frk f_k^{(r)}$ is constant on each of the two connected components of $\mcl H\setminus \mcl H(B_k^{(r)}\setminus B_k^{(1)})$. 
\end{proof}

\section{Convergence of random walk to Brownian motion}
\label{sec-general-clt}

Continue to assume that $\mcl H$ is a cell configuration satisfying the hypotheses of Theorem~\ref{thm-general-clt} and $\mcl D$ is an independent uniform dyadic system. In this section we will show that simple random walk on $\mcl H$ converges to Brownian motion, thereby completing the proof of Theorem~\ref{thm-general-clt}. 
To do this, we will show that the image of simple random walk under $\phi_\infty$ converges to Brownian motion (Proposition~\ref{prop-tutte-rw-conv}), then apply Theorem~\ref{thm-general-tutte}. The proof of this convergence will use the martingale central limit theorem (see Theorem~\ref{thm-martingale-clt} for a precise statement of the version we use). In order to prove the necessary quadratic variation convergence conditions, we will need analogs of some of the ergodic theory results in Section~\ref{sec-ergodic-avg} when we average along the path of the walk instead of over a large square. Such results are established in Section~\ref{sec-walk-ergodic}. We prove Theorem~\ref{thm-general-clt} in Section~\ref{sec-martingale-clt}. Sections~\ref{sec-clt-uniform} and~\ref{sec-variant-proof}, respectively, are devoted to the proof that the rate of convergence in Theorem~\ref{thm-general-clt} is uniform on compact subsets of $\BB C$ (Theorem~\ref{thm-general-clt-uniform}) and the proofs of the other variants of our main results.

\subsection{Ergodic averages for simple random walk}
\label{sec-walk-ergodic}

Let $Y$ be a two-sided discrete time simple random walk on the graph $\mcl H$, started at $Y_0   =H_0$. In this subsection we will prove an exact resampling property for $Y$ analogous to Lemma~\ref{lem-dyadic-resample} (Lemma~\ref{lem-walk-translate}) and a corresponding convergence statement for ergodic averages analogous to Lemma~\ref{lem-ergodic-avg} (Lemma~\ref{lem-walk-ergodic}). 

For this purpose we first need to choose a particular continuous-time parameterization $X$ of $Y$ with the property that $X$ spends $\op{Area}(H)/\pi(H)$ units of time at the cell $H$ whenever this cell is hit (recall that $\pi(H)$ is as in~\eqref{eqn-stationary-measure}). In order to avoid making 0 a special time we first sample a uniform random variable $\indshift_X$ from $[0,\op{Area}(H_0)/\pi(H_0)]$ and declare that $n_t=0$ and $X_t = H_0$ for $t\in [-\indshift_X , \op{Area}(H_0)/\pi(H_0) - \indshift_X]$. For $t\geq  \op{Area}(H_0)/\pi(H_0) - \indshift_X  $, we set
\eqb \label{eqn-walk-time-def}
X_t = Y_{n_t}  \quad \op{where} \quad n_t := \min\left\{ n\in\BB N  : \sum_{j=0}^n \frac{\op{Area}(Y_j)}{\pi(Y_j)} \geq t + \indshift_X  \right\}   
\eqe 
and for $t\leq -\indshift_X$, we define $n_t$ and $X_t$ analogously but with $Y$ run backward, rather than forward, and the sum going from $-n$ to $-1$ instead of $0$ to $n$. 

For $j\in\BB N $, let $\tau_j$ be the $j$th smallest time after 0 at which $X$ jumps to a new vertex and for $j\in\BB N_0$, let $\tau_{-j}$ be the $(j+1)$st largest time before 0 at which $X$ jumps to a new vertex (so $\tau_0 = -\indshift$). Note that $X_t = Y_{n_t} = X_{\tau_{n_t}}$.

It will be convenient to view $X_t$ as a process taking values in $\BB R^2$ rather than in $\mcl H$. Conditional on $(\mcl H , X)$, sample for each $j \in\BB Z \setminus\{0\}$ a point $\wh X_{\tau_j}$ uniformly from Lebesgue measure on the cell $X_{\tau_j}$ (independently for different cells and different values of $j$). Also set $\wh X_0 = 0$. Define
\eqb \label{eqn-walk-embed-def}
\wh X_t := \wh X_{\tau_j}  ,\quad \forall t\in (\tau_{j-1},\tau_j] \quad \text{so that} \quad H_{\wh X_t} = X_t  
\eqe
where in the second formula we recall that $H_z$ for $z\in\BB C$ denotes the cell containing $z$. 

We will now state and prove a translation invariance property for the walk $\wh X$ analogous to Lemma~\ref{lem-dyadic-resample}. To this end, we let $\wt{\mcl D}$ be a one-dimensional uniform dyadic system (Remark~\ref{remark-1d-dyadic-system}), independent from $(\mcl H,\mcl D , \wh X)$, and we let $\{I_k\}_{k\in\BB Z}$ be the sequence of intervals in $\wt{\mcl D}$ which contain the origin. For $a > 0$, we let $\wh I_a$ be the largest interval in $\wt{\mcl D}$ containing 0 with the property that
\eqb \label{eqn-I-def}
\sum_{j \in \BB Z} | [\tau_{j-1},\tau_j] \cap \wh I_a | \leq a .
\eqe
The interval $\wh I_a$ is the one-dimensional analog of the square $\wh S_m^0$ from~\eqref{eqn-mass-time-def}.

\begin{lem} \label{lem-walk-translate}
The law of $(\mcl H , \mcl D , \wh X , \wt{\mcl D} )$ satisfies the following translation invariance property. Let $a \geq 0$ and conditional on everything else, let $t$ be sampled uniformly from $\wh I_a$. Then the 4-tuple
\eqb \label{eqn-walk-translate} 
\left( \mcl H - \wh X_{t}  , \mcl D - \wh X_{t} , \wh X_{\cdot  +t} - \wh X_{t}  , \wt{\mcl D} - t  \right)
\eqe
agrees in law with $(\mcl H , \mcl D , \wh X , \wt{\mcl D} )$ modulo a global re-scaling, i.e., there is a random $C>0$ such that
\eqb \label{eqn-walk-translate'} 
\left( C(\mcl H - \wh X_{t} )  , C (\mcl D - \wh X_{t} ) ,  C(\wh X_{\cdot / C^2  +t} - \wh X_{t} )  , C^2(\wt{\mcl D} - t )  \right)
\eqD (\mcl H,\mcl D , \wh X , \wt{\mcl D} ).
\eqe  
\end{lem}
\begin{proof}
We will prove the lemma by considering random walk on $\mcl H(\wh S_m)$ with reflected boundary conditions and sending $m\rta\infty$. 
Let $X^m$, $\wh X^m$, and $\tau_j^m$ be defined in the same manner as $\wh X$ and $\tau_j$ above except that we start with a simple random walk $Y^m$ on $\mcl H(\wh S_m)$ (with reflected boundary conditions) started from $H_0$ instead of a simple random walk on $\mcl H$ started from $H_0$; we replace $\op{Area}(Y_j)/\pi(Y_j)$ by $\op{Area}(Y_j^m\cap \wh S_m)/\pi^{\mcl H(\wh S_m)}(Y_j^m)$ in~\eqref{eqn-walk-time-def}, where $\pi^{\mcl H(\wh S_m)}(H)$ is the sum of the conductances of the edges adjacent to $H$ in $\mcl H(\wh S_m)$; and we sample the points $\wh X_{\tau_j^m}^m$ from $X_{\tau_j^m}^m \cap \wh S_m$ instead of from $X_{\tau_j^m}^m$. 
It is easily seen that, conditional on $(\mcl H,\mcl D)$, the stationary distribution for $X^m$ is the one which assigns probability proportional to $\op{Area}(H \cap \wh S_m) $ to each cell $H\in \mcl H(\wh S_m)$. Equivalently, the stationary distribution of $\wh X^m$ is uniform on $\wh S_m$. 

As $t\rta \infty$ or $t\rta-\infty$, the conditional distribution of $\wh X^m_t$ given $(\mcl H , \mcl D)$ converges in the total variation sense to the stationary distribution.
Consequently, if we let $t_T$ be sampled uniformly from Lebesgue measure on $[-T,T]$ (independently from everything else), then as $T\rta\infty$ the conditional distribution of $\wh X_{t_T}^m$ given $(\mcl H,\mcl D)$ converges in the total variation sense to the uniform measure on $\wh S_m$. Hence, as $T\rta\infty$ the joint law of
\eqbn
\left( \mcl H , \mcl D , \wh X^m_{t_T} , \wh X^m_{\cdot + t_T}   , \right) 
\eqen
converges in the total variation sense to the law of the 4-tuple consisting of $\mcl H,\mcl D$, a point $z$ sampled uniformly from Lebesgue measure on $\wh S_m$, and the walk $\wh X^{m,z}$ which is defined in the same manner as $\wh X^m$ but with $\wh X^m_0=z$. 
 
By Lemma~\ref{lem-dyadic-resample}, there is a random $C>0$ such that with $z$ as above,
\eqb  \label{eqn-walk-resample}
\left( C( \mcl H - z)  , C(\mcl D - z) ,  C(\wh X^{m,z}_{\cdot/C^2  } - z) \right) 
\eqD (\mcl H,\mcl D , \wh X^m).
\eqe 
Note that to get the scaling of the time parametrization of $\wh X^{m,z}$, we use that replacing $\mcl H$ by $C\mcl H$ for $C>0$ scales areas of cells by $C^2$, which in turn scales the time parameterization of $\wh X^{m,z}$ by $1/C^2$. 

From~\eqref{eqn-walk-resample} and the above total variation convergence, we infer that there are random variables $C_T > 0$ such that
\eqbn
\left( C_T ( \mcl H -  \wh X^m_{t_T} )    , C_T ( \mcl D -  \wh X^m_{t_T} )  ,  C_T( \wh X^m_{\cdot / C_T^2 + t_T}  - \wh X^m_{t_T})  \right) 
\rta (\mcl H,\mcl D , \wh X^m) 
\eqen
in the total variation sense. By the one-dimensional analog of Lemma~\ref{lem-uniform-dyadic}, for each $T>0$ the dyadic system $C_T^2(\wt{\mcl D} - t_T)$ is a uniform dyadic system independent from $(\mcl H , \mcl D , \wh X^m , t_T)$. Therefore, we have the total variation convergence
\eqb \label{eqn-total-var-4tuple}
\left( C_T ( \mcl H -  \wh X^m_{t_T} )    , C_T ( \mcl D -  \wh X^m_{t_T} )  ,  C_T( \wh X^m_{\cdot / C_T^2 + t_T}  - \wh X^m_{t_T})  , C_T^2(\wt{\mcl D} - t_T) \right) 
\rta (\mcl H,\mcl D , \wh X^m , \wt{\mcl D}) .
\eqe 

The next several paragraphs of the proofs are similar to the proof that Condition~\ref{item-averaging} $\Rightarrow$ Condition~\ref{item-dyadic} in Appendix~\ref{sec-equivalence}.
Let $\wh I_a^{t_T}$ be the largest interval of $\wt{\mcl D}$ containing $t_T$ for which $\sum_{j \in \BB Z} | [\tau_{j-1},\tau_j] \cap \wh I_a^{t_T} | \leq a$ (i.e., the analog of $\wh I_a$ with $t_T$ in place of 0). We claim that 
\eqb \label{eqn-walk-translate-claim}
\lim_{T\rta\infty} \BB P\left[  \wh I_a^{t_T} \subset [-T,T] \right]= 1.  
\eqe 
Indeed, from~\eqref{eqn-total-var-4tuple} the law of the random interval $C_T^2 (\wh I_a^{t_T} - t_T)$ converges in the total variation sense to the law of $\wh I_a$, so $C_T^2 |\wh I_a^{t_T}|$ is a tight random variable. On the other hand, $C_T (\wh{\mcl S}_m - \wh X^m_{t_T}) \rta \wh S_m$ in the total variation sense, and $\wh X^m_{t_T} \subset \wh S_m$. From this, we get that $C_T$ and $C_T^{-1}$ are tight random variables. Therefore, $|\wh I_a^{t_T}|$ is a tight random variable. Since $t_T$ is sampled uniformly from $[-T,T]$, the probability that $|t_T| \geq T - r$ goes to zero as $T\rta \infty$ for any fixed $r >0$. From this and the tightness of $|\wh I_a^{t_T}|$, we obtain~\eqref{eqn-walk-translate-claim}.

On the event $\{ \wh I_a^{t_T} \subset [-T,T]\}$, the conditional law of $t_T$ given $(\mcl H , \mcl D , \wt{\mcl D} , \wh X^m ,  \wh I_a^{t_T})$ is uniform on $\wh I_a^{t_T}$. 
From~\eqref{eqn-walk-translate-claim}, we get that as $T\rta\infty$, the total variation distance between the conditional law of $t_T$ given $(\mcl H , \mcl D , \wt{\mcl D} , \wh X^m ,  \wh I_a^{t_T})$ and the uniform measure on $\wh I_a^{t_T}$ tends to zero as $T\rta \infty$.
By this and~\eqref{eqn-total-var-4tuple}, if we condition on $(\mcl H , \mcl D , \wt{\mcl D} , \wh X^m , t_T)$ and sample $t_T'$ uniformly from $\wh I_a^{t_T} $, then the total variation convergence~\eqref{eqn-total-var-4tuple} holds with $t_T'$ in place of $t_T$ and possibly different scaling factors $C_T' >0$ in place of $C_T$. 

On the other hand, the interval $C_T^2 (\wh I_a^{t_T  } - t_T )$ is determined by the left 4-tuple in~\eqref{eqn-total-var-4tuple} in the same manner that $\wh I_a$ is determined by the right 4-tuple in~\eqref{eqn-total-var-4tuple}. Consequently, we have the total variation convergence
\eqb \label{eqn-total-var-5tuple}
\left( C_T ( \mcl H -  \wh X^m_{t_T} )    , C_T ( \mcl D -  \wh X^m_{t_T} )  ,  C_T( \wh X^m_{\cdot / C_T^2 + t_T}  - \wh X^m_{t_T})  , C_T^2(\wt{\mcl D} - t_T) , C_T^2(t_T'-t_T) \right) 
\rta (\mcl H,\mcl D , \wh X^m , \wt{\mcl D} , t)  
\eqe 
where $t$ is sampled uniformly from $\wh I_a$. In particular, we have the total variation convergence
\allb \label{eqn-total-var-shifted}
&\left( C_T ( \mcl H -  \wh X^m_{t_T'} )    , C_T ( \mcl D -  \wh X^m_{t_T'} )  ,  C_T( \wh X^m_{\cdot / C_T^2 + t_T'}  - \wh X^m_{t_T'})  , C_T^2(\wt{\mcl D} - t_T')  \right) \notag\\
&\qquad \rta \left( \mcl H - \wh X_{t}^m , \mcl D - \wh X_{t}^m , \wh X_{\cdot  +t}^m - \wh X_{t}^m  , \wt{\mcl D} - t  \right).
\alle
From this and the analog of~\eqref{eqn-total-var-4tuple} with $t_T'$ in place of $t_T$, we obtain~\eqref{eqn-walk-translate'} with $\wh X^m$ instead of $\wh X$. We now conclude by sending $m\rta\infty$. 
\end{proof}

Next we will prove an analog of Lemma~\ref{lem-ergodic-avg} for averages along segments of the walk $\wh X$.

\begin{lem} \label{lem-walk-ergodic}
Let $F = F(\mcl H,\mcl D , \wh X , \wt{\mcl D}    )$ be a real-valued measurable function of $(\mcl H , \mcl D , \wh X ,\wt{\mcl D}   )$ which is invariant under global re-scaling in the sense that
\eqb \label{eqn-walk-scale}
F(C\mcl H, C\mcl D , C \wh X_{\cdot/C^2} , C^2 \wt{\mcl D}    ) = F(\mcl H, \mcl D , \wh X , \wt{\mcl D}  )
\eqe
for each $C>0$. For $t \in \BB R$, let $F_{t}$ denote $F$ applied to the shifted 4-tuple of~\eqref{eqn-walk-translate}. If either $\BB E[|F|] < \infty$ or $F\geq 0$ a.s., then a.s.\ 
\eqb \label{eqn-walk-ergodic}
\lim_{a \rta\infty} \frac{1}{|\wh I_a|} \int_{\wh I_a}  F_t \,dt = \BB E[F] .
\eqe 
\end{lem} 
\begin{proof}
We start out as in the proof of Lemma~\ref{lem-ergodic-avg}. 
Throughout the proof we assume that $\BB E[|F|] < \infty$ (the case $F\geq 0$ and $\BB E[F] = \infty$ is treated exactly as in Lemma~\ref{lem-ergodic-avg}). 
By analogy with Definition~\ref{def-dyadic-sigma-algebra}, for $a \geq 0$ let $\wt{\mcl F}_a$ be the $\sigma$-algebra generated by the measurable functions of $(\mcl H, \mcl D , \wh X , \wt{\mcl D}  )$ whose values are unchanged if we replace $(\mcl H, \mcl D , \wh X , \wt{\mcl D}   )$ by the 4-tuple in~\eqref{eqn-walk-translate} for any $t\in \wh I_a$ and then apply a global re-scaling to the function as in~\eqref{eqn-walk-scale}. By Lemma~\ref{lem-walk-translate},
\eqbn
 \frac{1}{|\wh I_a|} \int_{\wh I_a}  F_t \,dt = \BB E\left[ F \,|\, \wt{\mcl F}_a\right] .
\eqen
By the backward martingale convergence theorem, the limit in~\eqref{eqn-walk-ergodic} exists a.s.\ and in $L^1$. Call the limiting random variable $A = A (\mcl H,\mcl D , \wh X , \wt{\mcl D}    )$. We need to show that $A$ is equal to a deterministic constant a.s. 

The random variable $A$ is $\bigcap_{a>0} \wt{\mcl F}_a$-measurable, so by the definition of $\wt{\mcl F}_a$ and since $\bigcup_{a>0} \wh I_a = \BB R$, a.s.\  
\eqb \label{eqn-A-scale}
A(C(\mcl H - \wt X_t) , C( \mcl D - \wt X_t) , C (\wh X_{\cdot/C^2 + t} - \wh X_t) , C^2( \wt{\mcl D}  - t)   ) = A(\mcl H, \mcl D , \wh X , \wt{\mcl D}  )
\eqe 
for all $C>0$ and all $t\in\BB R$. 

Recall from Proposition~\ref{prop-cell-recurrent} that a.s.\ the random walk on $\mcl H$ is recurrent. 
If we condition on $(\mcl H,\mcl D,\wt{\mcl D})$, then the excursions of the walk $X$ away from $X_0$ are i.i.d.\ random variables. Under this conditional law, $A$ belongs to the tail $\sigma$-algebra for these random variables and the independent uniform random points of the cells of $\mcl H$ which we used to define $\wh X_t$. By the Kolmogorov zero-one law, $A$ is a.s.\ determined by $(\mcl H,\mcl D,\wt{\mcl D})$. By a slight abuse of notation, we now write $A = A(\mcl H,\mcl D,\wt{\mcl D})$. 
 
We will now argue that $A$ is a.s.\ determined by $\wt{\mcl D}$. To this end, let
\eqbn
\wt A(\mcl H ,\mcl D) :=  \int A(\mcl H , \mcl D , \wt{\mcl D}) d\wt{\frk m}(\wt{\mcl D})  
\eqen
where $\wt{\frk m}$ is the law of $\wt{\mcl D}$. 
If $C>0$ and $t\in\BB R$ are random and independent from $\wt{\mcl D}$, but allowed to depend on $(\mcl H,\mcl D)$, then since $\wt{\mcl D}$ is independent from $(\mcl H,\mcl D)$ and by the one-dimensional analog of Lemma~\ref{lem-uniform-dyadic}, $C^2 (\wt{\mcl D} - t)$ is a uniform dyadic system independent from $(\mcl H,\mcl D)$. 
By this and~\eqref{eqn-A-scale},  
\allb \label{eqn-A-cond}
\wt A(C(\mcl H - \wh X_t) , C(\mcl D- \wh X_t) )
&=  \int A(C(\mcl H- \wh X_t) , C(\mcl D- \wh X_t) ,  \wt{\mcl D} ) \, d\wt{\frk m}(\wt{\mcl D}) \notag\\
&=   \int  A(C(\mcl H - \wh X_t) , C(\mcl D-\wh X_t) ,  C^2(\wt{\mcl D} - t) ) \, d\wt{\frk m}(\wt{\mcl D})       \notag\\
&=   \wt A(\mcl H,\mcl D) .
\alle

Now let $R>0$ be random and independent from $\wt{\mcl D}$, but allowed to depend on $(\mcl H,\mcl D)$. 
Since the random walk on $\mcl H$ is recurrent and $\mcl H$ is connected, it is a.s.\ the case that for large enough $t > 0$ (depending on $(\mcl H,\mcl D , R)$), the Lebesgue measure on $B_R(0)$ is absolutely continuous with respect to the conditional law of $\wh X_t$ given $(\mcl H,\mcl D)$. 
From this and~\eqref{eqn-A-cond}, we get that if $C>0$ is random and independent from $\wt{\mcl D}$, but allowed to depend on $(\mcl H,\mcl D)$, then for Lebesgue-a.e.\ $z\in\BB C$, 
\eqb \label{eqn-A-translate}
\wt A(C(\mcl H - z) , C(\mcl D- z) ) = \wt A( \mcl H,\mcl D ) .
\eqe 

If $w$ is a point in the origin-containing cell $H_0$, then the definition of $\wh X$ implies that we can couple the process $\wh X_t$ with the analogous process $\wh X_t'$ defined with $( \mcl H - w  ,  \mcl D-w )$ in place of $(\mcl H,\mcl D)$ in such a way that the following is true. The associated discrete time walks $Y$ and $Y'$ satisfy $Y_j' = Y_j' - w$ for all $j \in \BB Z$. Moreover, we have $\wh X_{t+s} =  \wh X_{t+s'} - w $ for all $t\geq 0$, where $s$ and $s'$ are the first times that $\wh X$ and $\wh X'$, respectively jump away from the origin-containing cell; and an analogous relation holds for negative times.   
Since $A$ is the limit of $\frac{1}{|\wh I_a|} \int_{\wh I_a}  F_t \,dt$ as $t\rta \infty$, we infer that the value of $A$, and hence also the value of $\wt A$, is unaffected if we replace $(\mcl H,\mcl D)$ by $( \mcl H - w  ,  \mcl D-w )$. 
From this and the fact that the boundaries of the cells of $\mcl H$ have zero Lebesgue measure, we see that the fact that~\eqref{eqn-A-translate} holds a.s.\ for Lebesgue-a.e.\ $z$ implies that it in fact holds a.s.\ for every choice of $z$.  

By~\eqref{eqn-A-translate} and Lemma~\ref{lem-tail-trivial}, $\wt A$ is a.s.\ equal to a deterministic constant, so $A$ is a.s.\ determined by $\wt{\mcl D}$. By~\eqref{eqn-A-scale}, the value of $A$ is unaffected if we replace $\wt{\mcl D}$ by $C(\wt{\mcl D}-t)$ for any $C>0$ and $t\in \BB R$. The same argument as at the end of Lemma~\ref{lem-tail-trivial} now shows that $A$ is equal to a deterministic constant a.s. 
\end{proof}

To end this subsection, we record analogs of some of the lemmas from Section~\ref{sec-ergodic-avg} in the setting of Lemma~\ref{lem-walk-ergodic}. 
We start with a lemma which allows us to prove statements for intervals of $\wt{\mcl D}$ which do not contain the origin, analogous to Lemma~\ref{lem-more-squares}.

\begin{lem} \label{lem-more-intervals}
Let $E(S) = E(S,\mcl H)$ be an event depending on an interval $I\subset \BB R$ and the 4-tuple $(\mcl H,\mcl D,\wh X ,\wt{\mcl D} )$ of Lemma~\ref{lem-walk-ergodic}. 
Suppose that a.s.\ $E(I_k)$ occurs for each large enough $k\in\BB N$. Then for each $\el\in\BB N$, it is a.s.\ the case that for large enough $k\in\BB N$, $E(I)$ occurs for each dyadic ancestor $I$ of $I_k$ of length at most $2^\el |I_k|$ and each dyadic descendant $I$ of $I_k$ of length at least $2^{-\el} |I_k|$. 
\end{lem}
\begin{proof}
This follows from exactly the same argument use to prove Lemma~\ref{lem-more-squares}. 
\end{proof}

Similarly to Lemma~\ref{lem-max-cell-diam}, we also find that the maximum size of the intervals $[\tau_{j-1},\tau_j]$ (on which $X$ is constant) which intersect $I_k$ is a.s.\ $o_k(|I_k|)$. 

\begin{lem} \label{lem-max-interval-diam}
Almost surely, for each $\ep \in (0,1)$ it holds for large enough $k\in\BB N$ that
\eqbn
\tau_j- \tau_{j-1} \leq \ep |I_k| ,\quad \forall j \in \BB N \quad \text{with} \quad [\tau_{j-1} , \tau_j] \cap I_k\not=\emptyset .
\eqen
\end{lem}
\begin{proof}
We set
\eqbn
\wt{\mcl K}_\ep = \min\left\{k\in\BB N_0 : \tau_1 - \tau_0 \leq \ep |I_k| \right\} 
\eqen
then argue exactly as in the proof of Lemma~\ref{lem-max-cell-diam}, with Lemma~\ref{lem-walk-ergodic} used in place of Lemma~\ref{lem-ergodic-avg}. 
\end{proof}

\subsection{Proof of Theorem~\ref{thm-general-clt}}
\label{sec-martingale-clt}

Continue to assume we are in the setting of Section~\ref{sec-walk-ergodic}, so that $\{Y_j\}_{j\in\BB Z}$ is a two-sided simply random walk on $\mcl H$ started from $Y_0 = 0$ and $X_t = Y_{n_t}$ is obtained by re-parameterizing $Y$ so that it spends $\op{Area}(H)/\pi(H)$ units of time at each cell $H$ that it hits. In this subsection we will prove the following proposition, then use it in conjunction with~\eqref{eqn-cell-tutte-conv} (which was established in Section~\ref{sec-corrector-sublinearity}) to conclude the proof of Theorem~\ref{thm-general-clt}.
 
\begin{prop} \label{prop-tutte-rw-conv} 
There is a deterministic covariance matrix $\Sigma$ such that a.s.\ as $\ep\rta0$, the conditional law given $(\mcl H,\mcl D)$ of $t\mapsto \ep \phi_\infty(X_{t/\ep^2})$ converges weakly to the law of Brownian motion with covariance matrix $\Sigma$ started from 0, with respect to the local uniform topology.
\end{prop}

Since $\phi_\infty$ is discrete harmonic and by the optional stopping theorem, the process $t\mapsto \phi_\infty(X_t)$ is a local martingale. We can therefore prove Proposition~\ref{prop-tutte-rw-conv} using the martingale functional central limit theorem. In particular, we will use the following theorem, which is proven in~\cite{rebolledo-martingale-clt} (see~\cite[Theorem 5.1(a)]{helland-martingale-clt} for a statement in English).

\begin{thm}[\cite{rebolledo-martingale-clt}] \label{thm-martingale-clt}
Let $\{( M_n , \mcl F_n) \}_{n\in\BB N}$ be continuous-time real-valued local martingales, not necessarily continuous. For $t\geq 0$, let $\Delta M_n(t) := M_n(t) - M_n(t^-)$ be the jump of $M_n$ at time $t$ and for $\delta > 0$, let
\eqbn
\sigma^\delta[M_n](t) := \sum_{s\leq t}  \Delta M_n(s)^2 \BB 1_{(|\Delta M_n(s)|  \geq \delta )} .
\eqen 
Also let $\wt\sigma^\delta[M_n](t)$ be the compensator in the Doob-Meyer decomposition of $\sigma^\delta[M_n]$, so that $\sigma^\delta[M_n] - \wt\sigma^\delta[M_n]$ is a local martingale.
Suppose that we have the following convergence in probability for each $t\geq 0$:
\eqbn
\lim_{n\rta\infty} \la M_n \ra_t = t \quad\op{and} \quad \lim_{n\rta\infty} \wt\sigma^\delta[M_n](t) = 0\quad \forall \delta > 0 ,
\eqen
where $\la M_n \ra$ denotes the quadratic variation. Then $M_n$ converges in law to standard linear Brownian motion in the local uniform topology.
\end{thm}

We will apply Theorem~\ref{thm-martingale-clt} to the one-dimensional projections $\BB v\cdot (\ep \phi_\infty(X_{\cdot/\ep^2}) ) $ for unit vectors $\BB v\in\BB R^2$, then use the Cram\'er-Wold theorem. We first check the condition concerning the convergence of quadratic variations. 
For a unit vector $\BB v \in \BB R^2$, let
\eqb \label{eqn-quad-var-def}
V^{\BB v}_t 
:= \la \BB v\cdot \phi_\infty(X )  , \BB v \cdot \phi_\infty(X) \ra_{t} 
= \sum_{j=1}^{n_t} \left( \BB v\cdot (\phi_\infty(Y_j) - \phi_\infty(Y_{j-1}) ) \right)^2 .
\eqe
Also set
\eqb \label{eqn-quad-var-constant}
c_{\BB v} 
:= \BB E\left[ \frac{\pi(H_0)}{\op{Area}(H_0)}   \left( \BB v\cdot \phi_\infty(Y_1)  \right)^2  \right]  .
\eqe

\begin{lem} \label{lem-quad-var-finite}
For each unit vector $\BB v$, one has $0 < c_{\BB v } < \infty$.
\end{lem}
\begin{proof}
One obtains $c_{\BB v}  <\infty$ by first taking a conditional expectation given $(\mcl H,\mcl D)$ in~\eqref{eqn-quad-var-constant} then applying Lemma~\ref{lem-infty-energy}. 
If $c_{\BB v}$ were equal to zero for some $\BB v$, then one would have $\BB v\cdot \phi_\infty(Y_1) = 0$ a.s. This together with Lemma~\ref{lem-walk-translate} would then show that a.s.\ $\BB v\cdot (\phi_\infty(Y_j) - \phi_\infty(Y_{j-1}) ) = 0$ for every $j$. Since $\mcl H$ is connected and $Y$ is recurrent, this means that a.s.\ the range of $\phi_\infty$ is contained in a straight line orthogonal to $\BB v$, which is impossible by Proposition~\ref{prop-corrector}. 
\end{proof}

\begin{lem} \label{lem-quad-var-lim}
For each fixed unit vector $\BB v\in\BB R^2$, a.s.\
\eqb \label{eqn-quad-var-lim}
\lim_{T \rta\infty} \frac{1}{T}  V^{\BB v}_T = c_{\BB v} .
\eqe
\end{lem}
\begin{proof}
For $t \in \BB R$, let 
\eqb \label{eqn-bm-var-inc}
F_t := \frac{\pi(X_t)}{\op{Area}(X_t)}  \left( \BB v\cdot (\phi_\infty(Y_{n_t+1}) - \phi_\infty(Y_{n_t }) ) \right)^2    .
\eqe 
Then $F_0$ is scale invariant in the sense of Lemma~\ref{lem-walk-ergodic}, so we can apply that lemma to find that a.s.\ 
\eqb \label{eqn-dyadic-interval-var}
\lim_{k\rta\infty} \frac{1}{|I_k|} \int_{I_k}  F_t       \, dt = c_{\BB v} ,
\eqe 
where here we recall that the $I_k$'s are the intervals containing 0 in the one-dimensional dyadic system $\wt{\mcl D}$. 
Furthermore, by Lemma~\ref{lem-max-interval-diam} it is a.s.\ the case that the maximal length of the intervals $[\tau_{j-1}  ,\tau_j]$ on which $ X$ is constant which intersect $I_k$ is $o(|I_k|)$. 

To obtain~\eqref{eqn-quad-var-lim}, we need to transfer from the random interval $I_k$ to an interval of the form $[0,T]$ for fixed $T > 0$. 
We will do this by approximating $[0,T]$ by a union of small intervals $I$ for which~\eqref{eqn-dyadic-interval-var} holds with $I$ in place of $I_k$. 
To this end, fix $\el\in\BB N$ (which we will eventually send to $\infty$). 
By~\eqref{eqn-dyadic-interval-var} and two applications of Lemma~\ref{lem-more-intervals}, it is a.s.\ the case that for large enough $k\in\BB N$, 
\eqb \label{eqn-one-interval-var}
\left| \frac{1}{|I|} \int_I      F_t      \,dt  - c_{\BB v} \right| \leq 2^{-2\el} 
\quad \op{and} \quad 
\tau_{n_t} - \tau_{n_t-1} \leq 2^{-2\el} |I_k| ,\quad \forall t\in I 
\eqe 
for each interval $I$ which is a dyadic descendant of either of the two dyadic parents of $I_k$ and which has side length $2^{-\el} |I_k|$. Note that the union of these intervals contains $[-|I_k| , |I_k|]$. 

Give $T >0$, let $k_T \in \BB N$ be chosen so that $T\in [|I_{k_T}|/4 , |I_{k_T}|/2]$. 
Let $\ul{\mcl I}_T$ (resp.\ $\ol{\mcl I}_T $) be the union of those intervals $I$ of the form described just after~\eqref{eqn-one-interval-var} which are contained in $[0,T]$ (resp.\ which either intersect $[0,T]$ or share an endpoint with such an interval which intersects $[0,T]$). 
Note that the intervals $I$ contained in $\ol{\mcl I}_T$ intersect only along their endpoints, $\ul{\mcl I}_T\subset [0,T] \subset \ol{\mcl I}_T$, and $\ol{\mcl I}_T \setminus \ul{\mcl I}_T$ is a union of at most four such intervals $I$.  

The second inequality in~\eqref{eqn-one-interval-var} implies that a.s.\ for each large enough $T>0$, the intervals $[\tau_0,\tau_1]$ and $[\tau_{n_T } ,\tau_{n_T+1}]$ which contain the endpoints of $[0,T]$ are contained in $\ol{\mcl I}_T$. 
Recalling that $\tau_j -\tau_{j-1} = \op{Area}(Y_j)/\pi(Y_j)$ and the definition~\ref{eqn-quad-var-def} of $V_T^{\BB v}$, we can break integrals up into sums over intervals of the form $[\tau_{j-1},\tau_j]$ to find that 
\eqb \label{eqn-quad-var-int}
  \int_{\ul{\mcl I}_T} F_t \,dt 
  \leq 
  V_T^{\BB v}
  \leq   
  \int_{\ol{\mcl I}_T} F_t \,dt  .
\eqe  

By summing the first estimate in~\eqref{eqn-one-interval-var} over those intervals $I$ which are contained in each of $\ul{\mcl I}_T$ and $\ol{\mcl I}_T$ (there are at most $2^{\el + 1}$ such intervals), we find that a.s.\ for large enough $T > 0$, 
\eqb \label{eqn-quad-int-compare}
\frac{1}{|\ul{\mcl I}_T|} \int_I  F_t \,dt \geq c_{\BB v} - O_\el(2^{-\el}) 
\quad \op{and} \quad
\frac{1}{|\ol{\mcl I}_T|} \int_I F_t \,dt \leq c_{\BB v} + O_\el(2^{-\el})  ,
\eqe
with the rate of the $O_\el(2^{-\el})$ depending only on $c_{0,0}$. Since the lengths of $\ul{\mcl I}_T$, $[0,T]$, and $\ol{\mcl I}_T$ each differ by at most $2^{-\el + 2}|I_k| \leq 2^{-\el  +4} T$, we can plug~\eqref{eqn-quad-int-compare} into~\eqref{eqn-quad-var-int} and send $\el\rta\infty$ to obtain~\eqref{eqn-quad-var-lim}. 
\end{proof}

We next need to check the condition on $\wt\sigma^\delta[M_n]$ from Theorem~\ref{thm-martingale-clt}. 
This will be accomplished using the following lemma.

\begin{lem}  \label{lem-jump-sum}
Fix a unit vector $\BB v\in\BB R^2$. 
For $j\in\BB Z$ and $r > 0$, let 
\eqb \label{eqn-one-jump}
\Delta_j^r := \left( \BB v\cdot \left( \phi_\infty(Y_j) - \phi_\infty(Y_{j-1})   \right)  \right)^2 \BB 1_{\left( \left|\BB v\cdot \left( \phi_\infty(Y_j) - \phi_\infty(Y_{j-1})   \right)  \right| \geq  r \right)}
\eqe
and
\allb \label{eqn-mean-jump}
\wt\Delta_j^r 
&:= \BB E\left[ \Delta_j^r \,|\, \mcl H,\mcl D , Y |_{(-\infty, j-1]} \right] \notag\\
&= \frac{1}{  \pi(Y_{j-1})  }  \sum_{\substack{H\in\mcl H \\ H\sim Y_{j-1}}} \frk c(Y_{j-1},H)  \left(\BB v\cdot ( \phi_\infty(H) - \phi_\infty(Y_{j-1})  ) \right)^2 \BB 1_{\left( \left| \BB v\cdot (\phi_\infty(H) - \phi_\infty(Y_{j-1} ))\right| \geq  r \right)}  .
\alle
Almost surely, for each $\delta > 0$,
\eqb \label{eqn-jump-sum}
\lim_{T\rta\infty} \frac{1}{T}  \sum_{j=1}^{n_T} \Delta_j^{\delta T^{1/2} } = \lim_{T\rta\infty} \sum_{j=1}^{n_T} \wt \Delta_j^{\delta T^{1/2} }  = 0.
\eqe
\end{lem}
\begin{proof}
We will prove that the second limit in~\eqref{eqn-jump-sum} is a.s.\ equal to 0, then comment on the adaptations needed to treat the first limit. 
For $a \geq 0$, let $\wh I_a \in\wt{\mcl D}$ be as in~\eqref{eqn-I-def} and let
\allb \label{eqn-big-jump-rv}
\wt F^a 
 := \frac{\pi(H_0)}{\op{Area}(H_0)} \wt\Delta_1^{\delta |\wh I_a|^{1/2}}  
 =\frac{1}{\op{Area}(H_0)}  \sum_{\substack{H\in\mcl H \\ H\sim H_0} } \frk c(H_0,H) \left(\BB v\cdot   \phi_\infty(H) \right)^2 \BB 1_{\left( \left|\BB v\cdot  \phi_\infty(H)  \right| \geq \delta |\wh I_a|^{1/2}       \right)}  ,
\alle
where here we set $|\wh I_0|  =0$, so that the indicator is identically equal to 1 in this case. By Lemma~\ref{lem-infty-energy}, we have $\BB E[\wt F^0]  < \infty$. Since a.s.\ $|\wh I_a| \rta \infty$ as $a\rta\infty$, for each $\alpha >0$ there is a deterministic $a_* = a_*(\delta ,\alpha)  > 0$ such that $\BB E[\wt F^{a_*}] \leq \alpha$. 
The random variable $\wt F^{a_*}$ is a scale-invariant functional of $(\mcl H,\mcl D, \wh X,\wt{\mcl D})$ in the sense of Lemma~\ref{lem-walk-ergodic} (note that spatially scaling by $C$ scales the time parameterization of $X$ by $1/C^2$, and hence the length of $\wh I_a$ by $C^2$). 
Therefore, Lemma~\ref{lem-walk-ergodic} applied to this random variable implies that a.s.\
\eqbn
\limsup_{a \rta\infty} \frac{1}{|\wh I_a|} \int_{\wh I_a} 
 \sum_{\substack{H\in\mcl H \\ H\sim X_t} } \frac{\frk c(X_t,H)}{\op{Area}(X_t)}   \left(\BB v\cdot ( \phi_\infty(H) - \phi_\infty(X_t) ) \right)^2   
\BB 1_{\left(    \left|\BB v\cdot ( \phi_\infty(H) - \phi_\infty(X_t) ) \right| \geq \delta |\wh I_a|^{1/2}     \right)}    
  \,dt 
\leq \alpha .
\eqen
Breaking up the integral into a sum over intervals of the form $[\tau_{j-1},\tau_j] $ which are contained in $\wh I_a$ and noting that each such interval has length $\op{Area}(Y_{j-1})/\pi(Y_{j-1})$, we get that a.s.\ 
\eqbn
\limsup_{a \rta\infty} \frac{1}{|\wh I_a|} \sum_{j : [\tau_{j-1},\tau_j] \subset \wh I_a}
\wt\Delta_j^{\delta |\wh I_a|^{1/2}} 
\leq \alpha .
\eqen
By Lemma~\ref{lem-more-intervals}, the same is true with $\wh I_a$ replaced by the union of the dyadic parents of $\wh I_a$, which in turn contains $[0,|\wh I_a|]$. 
Combining this with Lemma~\ref{lem-max-interval-diam} (to deal with the intervals containing 0 and $T$) and sending $\alpha\rta 0$ shows that the second limit in~\eqref{eqn-jump-sum} is 0. 

To treat the first limit in~\eqref{eqn-jump-sum}, we apply the same argument as above with the random variable $\wt F^a$ from~\eqref{eqn-big-jump-rv} replaced by $F^a := \frac{\pi(H_0)}{\op{Area}(H_0)} \Delta_1^{\delta |\wh I_a|^{1/2}} $. 
We note that $\BB E[F^0]$ can be seen to be finite by first taking a conditional expectation given $(\mcl H,\mcl D)$, then applying Lemma~\ref{lem-infty-energy}.
\end{proof}

\begin{proof}[Proof of Proposition~\ref{prop-tutte-rw-conv}]
Fix a unit vector $\BB v\in \BB R^2$. We will check the hypotheses of Theorem~\ref{thm-martingale-clt} under the conditional law given $(\mcl H, \mcl D)$ for the local martingales
\eqbn
M^{\BB v}_\ep(t) := \ep\left(  \BB v\cdot \phi_\infty(X_{t/\ep^2} ) \right) .
\eqen
Note that $M_\ep^{\BB v}$ is a martingale under the conditional law given $(\mcl H,\mcl D)$ since $\phi_\infty$ is discrete harmonic. 

Lemma~\ref{lem-quad-var-lim} shows that for each $t\geq 0$, a.s.\ 
\eqbn
\lim_{\ep\rta 0} \la M_\ep^{\BB v} , M_\ep^{\BB v} \ra_t 
= \lim_{\ep\rta 0} \ep^2 V^{\BB v}_{t/\ep^2} 
= c_{\BB v} t. 
\eqen
Furthermore, in the notation of Theorem~\ref{thm-martingale-clt} and Lemma~\ref{lem-jump-sum} we have
\eqbn
\sigma^\delta[M_\ep^{\BB v}](t) 
= \ep^2 \sum_{j=1}^{n_{t/\ep^2} } \Delta_j^{\delta/\ep} ,
\eqen
so by the discrete Doob decomposition and the optional stopping theorem, the compensator is given by
\eqbn
\wt\sigma^\delta[M_\ep^{\BB v}](t) 
= \ep^2 \sum_{j=1}^{n_{t/\ep^2}}  \wt\Delta_j^{\delta/\ep}   .
\eqen
Therefore, Lemma~\ref{lem-jump-sum} shows that a.s.\ $\lim_{\ep\rta 0 } \wt\sigma^\delta[M_\ep^{\BB v}](t)  = 0$ for each $t\geq 0$. 

Hence, we can apply Theorem~\ref{thm-martingale-clt} to find that for each fixed unit vector $\BB v \in \BB R^2$, a.s.\ the conditional distribution given $(\mcl H,\mcl D)$ of $  M^{\BB v}_\ep$ converges in law as $\ep\rta 0$ to standard linear Brownian motion with quadratic variation $c_{\BB v} \,dt$. Almost surely, this holds for Lebesgue-a.e.\ choice of $\BB v$ simultaneously. By the Cram\'er-Wold theorem and~\eqref{eqn-quad-var-constant}, we obtain the proposition statement for an appropriate choice of covariance matrix $\Sigma$. 
This covariance matrix is positive definite since each $c_{\BB v}$ is positive (Lemma~\ref{lem-quad-var-finite}). 
\end{proof}

\begin{proof}[Proof of Theorem~\ref{thm-general-clt}]
Combine Proposition~\ref{prop-tutte-rw-conv} with Proposition~\ref{prop-cell-recurrent} and~\eqref{eqn-cell-tutte-conv} (which has already been proven).
\end{proof}

\subsection{Convergence is uniform in the starting point}
\label{sec-clt-uniform}

In this subsection we prove a straightforward extension of Theorem~\ref{thm-general-clt} which shows that in fact random walk on $\ep \mcl H$ converges to Brownian motion modulo time parameterization \emph{uniformly} over all choices of starting point in a compact subset of $\BB C$. This result is needed in~\cite{gms-tutte}.

\begin{thm} \label{thm-general-clt-uniform}
For $z\in\BB C$, let $Y^z$ denote the simple random walk on $\mcl H$ started from $H_z$. For $j\in\BB N_0$, let $\wh Y_j^z$ be an arbitrarily chosen point of the cell $Y_j^z$ and extend $\wh Y^z$ from $\BB N_0$ to $[0,\infty)$ by piecewise linear interpolation. 
Let $\Sigma$ be the covariance matrix from Theorem~\ref{thm-general-clt}. 
For each fixed compact set $A \subset \BB C$, it is a.s.\ the case that as $\ep \rta 0$, the maximum over all $z\in A$ of the Prokhorov distance between the conditional law of $\ep \wh Y^{z/\ep}$ given $\mcl H$ and the law of Brownian motion started from $z$ with covariance matrix $\Sigma$, with respect to the topology on curves modulo time parameterization (recall~\eqref{eqn-cmp-metric-loc}), tends to 0. 
\end{thm}

An important consequence of Theorem~\ref{thm-general-clt-uniform} is the convergence of discrete harmonic functions to their continuous counterparts. In the statement of the following corollary and in the rest of this subsection, when we refer to ``Euclidean harmonic" functions we mean functions which are harmonic with respect to the Brownian motion with covariance matrix $\Sigma$, rather than the standard two-dimensional Brownian motion. 

\begin{cor} \label{cor-harmonic-conv}
Let $D\subset \BB C$ be an open domain bounded by a Jordan curve, let $\frk f : \ol D\rta \BB R$ be Euclidean harmonic on $D$ and continuous on $\ol D$, and for $\ep > 0$ let $\frk f^\ep  :\mcl H(\ep^{-1}\ol D) \rta \BB R$ be discrete harmonic on $\mcl H(\ep^{-1}\ol D) \setminus \mcl H(\ep^{-1}\bdy D)$ with boundary data chosen so that for $H\in\mcl H(\ep^{-1}\bdy D)$, $\frk f^\ep(H)  = \frk f(z)$ for some (arbitrary) $z\in H\cap\ol D$. View $\frk f^\ep$ as a function on $\ep^{-1} \ol D$ which is constant on each cell. Then a.s.\ $\ep \frk f^\ep(\ep^{-1} \cdot) \rta \frk f$ uniformly as $\ep\rta 0$. 
\end{cor}

Theorem~\ref{thm-general-clt-uniform} will be extracted from Theorem~\ref{thm-general-clt} by using the re-sampling property of Lemma~\ref{lem-dyadic-resample} to get uniform convergence of random walk started from a large (but finite) number of points to Brownian motion, then using continuity considerations to transfer to uniform convergence at all points simultaneously. The following lemma is needed to compare the random walks started from two nearby points.

\begin{lem} \label{lem-disconnect-coupling}
Let $G$ be a connected graph and let $A\subset \mcl V(G)$ be a set such that the simple random walk started from any vertex of $G$ a.s.\ hits $A$ in finite time.
For $x\in \mcl V(G)$, let $X^x$ be the simple random walk started from $x$ and let $\tau^x$ be the first time $X^x$ hits $A$. 
For $x,y\in \mcl V(G) \setminus A$,
\eqb \label{eqn-disconnect-coupling}
\BB d_{\op{TV}}\left(X^x_{\tau^x} , X^y_{\tau^y} \right) \leq  1 - \BB P\left[ \text{$X^x$ disconnects $y$ from $A$ before time $\tau^x$} \right] ,
\eqe 
where $\BB d_{\op{TV}}$ denotes the total variation distance.
\end{lem}
\begin{proof} 
The lemma is a consequence of Wilson's algorithm. 
Let $T$ be the uniform spanning tree of $G$ with all of the vertices of $A$ wired to a single point. For $x\in \mcl V(G)$, let $L^x$ be the unique path in $T$ from $x$ to $A$. 
For a path $P$ in $G$, write $\op{LE}(P) $ for its chronological loop erasure. 
By Wilson's algorithm~\cite{wilson-algorithm}, we can generate the union $L^x\cup L^y$ in the following manner. 
\begin{enumerate}
\item Run $X^y$ until time $\tau^y$ and generate the loop erasure $ \op{LE}(X^y|_{[0,\tau^y]})$.
\item Conditional on $X^y|_{[0,\tau^y]}$, run $X^x$ until the first time $\wt\tau^x$ that it hits either $\op{LE}(X^y|_{[0,\tau^y]})$ or $A$. 
\item Set $L^x \cup L^y = \op{LE}(X^y|_{[0,\tau^y]}) \cup \op{LE}(X^x|_{[0,\wt\tau^x]})$. 
\end{enumerate}
Note that $L^y =  \op{LE}(X^y|_{[0,\tau^y]})$ in the above procedure. 
Applying the above procedure with the roles of $x$ and $y$ interchanged shows that $L^x \eqD \op{LE}(X^x|_{[0,\tau^x]})$.
When we construct $L^x\cup L^y$ as above, the points where $L^x$ and $L^y$ hit $A$ coincide provided $X^x$ hits $ \op{LE}(X^y|_{[0,\tau^y]})$ before $A$, which happens in particular if $X^x$ disconnects $y$ from $A$ before time $\tau^x$. We thus obtain a coupling of $\op{LE}(X^x|_{[0,\tau^x]})$ and $\op{LE}(X^y|_{[0,\tau^y]})$ such that the probability that  these two loop erasures hit $A$ at the same point is at least $ \BB P\left[ \text{$X^x$ disconnects $y$ from $A$ before time $\tau^x$} \right]$. We now obtain~\eqref{eqn-disconnect-coupling} by observing that $X^x_{\tau^x}$ is the same as the point where $\op{LE}(X^x|_{[0,\tau^x]})$ first hits $A$, and similarly with $y$ in place of $x$.
\end{proof}

Recall the dyadic system $\mcl D$ from Section~\ref{sec-dyadic-system} and the sequence $\{S_k\}_{k\in\BB Z}$ of squares containing the origin in this system. We also use the following notation. 

\begin{notation}
For a square $S$, write $S^{\op{P}}$ for the union of the four dyadic parents of $S$ (equivalently, the square with the same center as $S$ with side length $3|S|$). 
\end{notation}

\begin{lem} \label{lem-uniform-hm}
Almost surely, the supremum over all $z\in S_k$ of the Prokhorov distance between the conditional law given $(\mcl H,\mcl D)$ of the point where the linearly interpolated walk $\wh Y^z$ exits $S_k^{\op{P}}$ and the Euclidean harmonic measure on $\bdy S_k^{\op{P}}$ as viewed from $z$, with respect to $|S_k|^{-1}$ times the Euclidean metric on $\bdy S_k^{\op{P}}$, tends to 0 as $k\rta\infty$.
\end{lem}
\begin{proof}
Fix $N\in\BB N$ and for $m > 0$, let $z_1^m ,\dots,z_N^m$ be sampled uniformly and independently from Lebesgue measure on the square $\wh S_m$ of~\eqref{eqn-mass-time-def}. 
By Theorem~\ref{thm-general-clt} and Lemma~\ref{lem-dyadic-resample}, if we take the walks $\wh Y^{z_j^m}$ for $j\in [1,N]_{\BB Z}$ to be conditionally independent given $(\mcl H, \mcl D)$, then a.s.\ the joint conditional law given $(\mcl H, \mcl D)$ of $t\mapsto |\wh S_m|^{-1} ( \wh Y^{z_j^m} - z_j^m)$ for $j =1,\dots,N$ converges to the joint law of $N$ independent Brownian motions with covariance matrix $\Sigma$ started from 0 with respect to the local topology on curves modulo time parameterization. 
  
Since a Brownian motion started from 0 has probability tending to 1 as $\delta\rta 0$ to disconnect $B_\delta(0)$ from $\infty$ before leaving $B_1(0)$,  
if we are given $\delta > 0$, we can find $N = N(\delta) \in \BB N$ such that a.s., for each large enough $m > 0$ the following is true.
\begin{enumerate}
\item For each $j\in [1,N]_{\BB Z}$, the Prokhorov distance between the conditional law given $(\mcl H, \mcl D)$ of the point where $\wh Y^{z_j^m}$ exits $\wh S_m^{\op{P}}$ and the Euclidean harmonic measure on $\bdy \wh S_m^{\op{P}}$ as viewed from $z_j^m$, with respect to $|\wh S_m|^{-1}$ times the Euclidean distance, is at most $\delta  $. \label{item-uniform-hm}
\item With conditional probability at least $1-\delta$ given $(\mcl H , \mcl D)$, we have $\wh S_m \subset  \bigcup_{j=1}^N B_{\delta |\wh S_m|} (z_j^m)$. \label{item-uniform-cover}
\item For each $j \in [1,N]_{\BB Z}$, it holds with conditional probability at least $1-\delta$ given $(\mcl H,\mcl D)$ that $ Y^{z_j^m}$ disconnects $\mcl H( B_{\delta |\wh S_m|}(z_j^m)) $ from $\infty$ before hitting $\mcl H( \bdy \wh S_m^{\op{P}} )$.  \label{item-uniform-dc}
\end{enumerate}

By Lemma~\ref{lem-disconnect-coupling}, the condition~\ref{item-uniform-dc} implies that the total variation distance between the discrete harmonic measure on $\mcl H(\bdy \wh S_m^{\op{P}})$ as viewed from $H_{z_j^m}$ and the harmonic measure on $\mcl H(\bdy \wh S_m^{\op{P}})$ as viewed from any $z \in  B_{\delta |\wh S_m|}(z_j^m)$ is at most $\delta$. 
Combined with condition~\ref{item-uniform-hm} and the continuity of Euclidean harmonic measure with respect to the base point, this implies that for any $j\in [1,N]_{\BB Z}$ and any $z\in B_{\delta |\wh S_m|}(z_j^m)$, the Prokhorov distance between the conditional law given $(\mcl H, \mcl D)$ of the point where $\wh Y^z$ exits $\wh S_m^{\op{P}}$ and the Euclidean harmonic measure on $\bdy \wh S_m^{\op{P}}$ as viewed from $z$ is at most $o_\delta(1) $, with the $o_\delta(1)$ deterministic and independent of $k$. 
Finally, combining this with condition~\ref{item-uniform-cover} and sending $\delta\rta 0$ yields the statement of the lemma.
\end{proof}

\begin{proof}[Proof of Theorem~\ref{thm-general-clt-uniform}]
Fix $\el \in \BB N$ (which we will eventually send to $\infty$).
By Lemma~\ref{lem-uniform-hm} and two applications of Lemma~\ref{lem-more-squares}, the following holds a.s. 
Let $\mcl D_{k,\el}$ be the set of squares which are dyadic descendants of at least one of the four dyadic parents of $S_k$ and which have side length at least $ 2^{-\el} |S_k|$
Then for $S\in \mcl D_{k,\el}$, the maximum over all $z\in S$ of the Prokhorov distance between the conditional law given $(\mcl H,\mcl D)$ of the point where the linearly interpolated walk $\wh Y^z$ exits $S^{\op{P}}$ and the Euclidean harmonic measure on $\bdy S^{\op{P}}$ as viewed from $z$ with respect to $|S_k|^{-1}$ times the Euclidean distance tends to 0 as $k\rta\infty$.

For $z\in S_k^{\op{P}}$, let $\tau_0^z := z$ and inductively let $\tau_n^z$ for $n\in\BB N$ be the first time that $\wh Y^z$ exits $S^{\op{P}}$, for $S$ the square of $\mcl D$ with side length $2^{-\el} |S_k|$ which contains $\wh Y^z_{\tau_{n-1}^z}$; or $\tau_n^z = \tau_{n-1}^z$ if $\wh Y^z$ is not in $S_k^{\op{P}}$. Similarly define $\sigma^z_n$ for $n\in\BB N_0$ with Brownian motion $\mcl B^z$ started $z$ in place of $\wh Y^z$. 
The preceding paragraph and the strong Markov property shows that a.s.\ for each fixed $N\in\BB N$, the maximum over all $z\in S_k^{\op{P}}$ and all $n\in [0,N]_{\BB Z}$ of the Prokhorov distance between the conditional laws given $(\mcl H,\mcl D)$ of $  \wh Y^z_{\tau_n^z}$ and $\mcl B^z_{\sigma_n^z}$ (with respect to $|S_k|^{-1}$ times the Euclidean metric) tends to zero as $k \rta\infty$. We have $|\wh Y^z_t - \wh Y^z_{\tau_{n-1}^z}| \leq 2^{-\el +2}|S_k|$ for $t\in [\tau_{n-1}^z,\tau_n^z]$ and similarly for $\mcl B^z$. By Brownian scaling, the law of the largest $n$ for which $\sigma_{n-1}^z\not=\sigma_n^z$ can be bounded uniformly independently of $z$ and $k$.  

It follows that a.s.\ it holds uniformly over all $z\in S_k^{\op{P}}$ that the Prokhorov distance between the conditional law given $(\mcl H,\mcl D)$ of $  \wh Y^z$, stopped upon exiting $S_k^{\op{P}}$ and $\mcl B^z$, stopped upon exiting $S_k^{\op{P}}$, each scaled by $|S_k|^{-1}$, is $o_\el(1) + o_k(1)$, with the rate of the $o_\el(1)$ deterministic and independent of $k$.  

We now conclude by sending $\el\rta \infty$ and noting that $B_r(0) \subset S_k^{\op{P}}$ for $r\leq |S_k|$. 
\end{proof}

\subsection{Proof of Theorems~\ref{thm-general-clt0}, Theorem~\ref{thm-graph-clt}, and Theorem~\ref{thm-dual-clt}}
\label{sec-variant-proof}

We will now prove the other variants of our main theorem.

\begin{proof}[Proof of Theorems~\ref{thm-general-clt0}, Theorem~\ref{thm-general-tutte}, and Theorem~\ref{thm-recurrent}]
The combination of these theorems is a special case of Theorem~\ref{thm-general-clt} since the adjacency graph of faces of an embedded lattice is a cell configuration satisfying the connectivity along lines hypothesis~\ref{item-hyp-adjacency}. 
\end{proof}

\begin{proof}[Proof of Theorem~\ref{thm-graph-clt}]
Let $\mcl M$ be an embedded lattice satisfying the hypotheses of Theorem~\ref{thm-graph-clt}. 
We will construct a cell configuration $\rng{\mcl H}$ which satisfies the hypotheses of Theorem~\ref{thm-general-clt} and whose cells correspond to the vertices of $\mcl M$. 

We first divide each face $H\in \mcl F\mcl M$ into cells corresponding to the vertices in $\bdy H$. One way of doing this is as follows. For each such face $H$, let $w_H$ be sampled uniformly from Lebesgue measure on $H$ and let $x_1,\dots,x_n$ be the vertices lying on $\bdy H$, in counterclockwise cyclic order. Let $\psi : H \rta \BB D$ be a conformal map sending $w_H$ to 0. For $j\in [1,n]_{\BB Z}$, let $ P_j$ be a radial line segment in $\BB D$ from 0 to the midpoint of the counterclockwise arc of $\bdy\BB D$ from $\psi(x_j)$ to $\psi(x_{j-1})$ (where here we set $x_0= x_n$). We then let $\rng H_{x_j,H} \subset H$ be the closure of the image under $\psi^{-1}$ of the slice of $\BB D$ lying between $P_{j-1}$ and $P_j$ (where here we set $P_0 = P_n$). 

For $x\in \mcl V\mcl M$, we define the cell $\rng H_x := \bigcup_{H \in \mcl F\mcl M : x\in\bdy H} \rng H_{x,H}$. We then set $\rng{\mcl H} := \{\rng H_x : x\in \mcl V\mcl M\}$ and declare that $\rng H_x\sim  \rng H_y$ if and only if $x$ and $y$ are joined by an edge in $\mcl M$, which (by the above definition of the $\rng H_{x,H}$'s) is the case if any only if $\rng H_x$ and $\rng H_y$ intersect along a non-trivial boundary arc. We also set $\frk c(\rng H_x, \rng H_y) := \frk c(x,y)$ and $\pi(\rng H_x) := \pi(x)$.   

We now check the hypotheses of Theorem~\ref{thm-general-clt} for the cell configuration $\rng{\mcl H}$. The translation invariance modulo scaling of $\rng{\mcl H}$ follows from the translation invariance modulo scaling hypothesis for $\mcl M$ and the fact that we defined the cells in each face in a manner depending only on the face (to avoid continuity issues, we use one of the definitions which doesn't involve limits, e.g., the one given in Lemma~\ref{lem-dyadic-resample}). Ergodicity modulo scaling for $\rng{\mcl H}$ is then immediate from ergodicity modulo scaling for $\mcl M$.
To check hypothesis~\ref{item-hyp-moment} for $\rng{\mcl H}$, we observe that the translation invariance modulo scaling of $\mcl M$ shows that if we condition on $\mcl M$ and sample $z$ uniformly from the face $H_0$, then $\mcl M-z\eqD \mcl M$. Consequently, we can average over $H_0$ to find that if $\rng H_0$ is the cell of $\rng{\mcl H}$ containing $0$, then
\alb
\BB E\left[ \frac{\op{diam}(\rng H_0)^2 \pi(\rng H_0)}{\op{Area}(\rng H_0)} \right] 
&= \BB E\left[ \sum_{x\in \mcl V\mcl M\cap \bdy H_0} \frac{\op{diam}( \rng H_x )^2 \pi(x) \op{Area}(\rng H_x\cap h_0) }{ \op{Area}(\rng H_x) \op{Area}(H_0)} \right]\\
&\leq \BB E\left[ \sum_{x \in \mcl V\mcl M \cap \bdy H_0}  \frac{  \op{Outrad}(x)^2 \pi(x) }{ \op{Area}(H_0) }   \right] ,
\ale
which is finite by our finite expectation hypothesis for $\mcl M$. Similarly with $\pi^*$ in place of $\pi$. 
Hypothesis~\ref{item-hyp-adjacency} for $\mcl H$ follows since one can check from the above construction that if $H_x\cap H_y \not=\emptyset$ but $H_x\not\sim H_y$, then $H_x\cap H_y$ is a finite collection of points of the form $w_H$ defined above. The theorem statement now follows from Theorem~\ref{thm-general-clt}.
\end{proof}

The hardest part of the proof of Theorem~\ref{thm-dual-clt} is the following analog of Lemma~\ref{lem-max-cell-diam}. 

\begin{lem} \label{lem-dual-max-diam}
In the setting of Theorem~\ref{thm-dual-clt}, a.s.\ 
\eqb \label{eqn-dual-max-diam}
\lim_{r\rta\infty} \frac{1}{r} \max_{e \in  \mcl E\mcl M  : e\cap \cap B_r(0) \not=\emptyset } \op{diam}(e)
= \lim_{r\rta\infty} \frac{1}{r} \max_{e^* \in \mcl E\mcl M^* : e^*\cap B_r(0) \not=\emptyset } \op{diam}(e^*)
=  0. 
\eqe 
\end{lem}

We note that Lemma~\ref{lem-dual-max-diam} implies that also the maximum diameter of the primal or dual faces which intersect $B_r(0)$ tends to zero since any two points on the boundary of such a face are contained in a connected union of two primal edges and two dual edges.

Lemma~\ref{lem-dual-max-diam} will be proven in Appendix~\ref{sec-dual-max-diam}. The proof is much harder than that of Lemma~\ref{lem-max-cell-diam} since we need to rule out the possibility of large edges with small conductances. In the setting of Theorem~\ref{thm-general-clt}, we have a bound for $\BB E[\op{diam}(H_0)/\op{Area}(H_0)]$ since $\pi(H_0) + \pi^*(H_0) \leq 1$, but we have no such bound in the setting of Theorem~\ref{thm-dual-clt}. Once Lemma~\ref{lem-dual-max-diam} is established, Theorem~\ref{thm-dual-clt} follows from essentially the same argument used to prove Theorem~\ref{thm-general-clt}. We now give a brief account of the (entirely cosmetic) adaptations necessary.  
 
\begin{proof}[Proof of Theorem~\ref{thm-dual-clt}] 
All of the proofs in Sections~\ref{sec-dyadic-system} and~\ref{sec-ergodic-avg} carry over verbatim except that we replace the quantities appearing in~\eqref{eqn-hyp-moment} with the quantities appearing in~\eqref{eqn-dual-hyp-moment} and we apply Lemma~\ref{lem-dual-max-diam} instead of Lemma~\ref{lem-max-cell-diam}. The arguments of Section~\ref{sec-corrector} through~\ref{sec-recurrence} are exactly the same except that in the proof of Lemma~\ref{lem-path-sum} we sum over edges of $\mcl M$ whose corresponding dual edges cross the line segment $P_k(r)$ instead of over cells of $\mcl H$ which intersect this line segment. The set of all such edges is necessarily connected since if a dual edge crosses $P_k(r)$, then the faces on either side of the edge must intersect $P_k(r)$. The argument of Section~\ref{sec-general-clt} is identical, with the area of the face of $\mcl M^*$ containing a vertex used in place of the area of the cell in the parameterization of the walk $X$.
\end{proof}

\section{Applications to planar maps and Liouville quantum gravity}
\label{sec-lqg-application}

Our results are motivated by applications to the convergence of embedded random planar maps toward Liouville quantum gravity, which are explored in more depth in the papers~\cite{gms-tutte,gms-poisson-voronoi}. Here we provide a brief overview of such applications. 

There are various ways of embedding a planar map $M$ into $\BB C$ (i.e., drawing the vertices and edges) which can be thought of as canonical, in some sense. Particularly relevant for us is the \emph{Tutte embedding} (a.k.a.\ \emph{harmonic} or \emph{barycentric} embedding) where the position of each vertex is the average of the positions of its neighbors, equivalently the embedding function is discrete harmonic. In the setting of Theorem~\ref{thm-general-tutte}, the function $\phi_\infty$ is a version of the Tutte embedding of the dual map of $\mcl M$. Other examples of embeddings include circle packing, Riemann uniformization, and various types of square tilings. 

It is a long-standing conjecture that for a large class of random planar maps --- including uniform maps (or triangulations, quadrangulations, etc.) and maps sampled with probability proportional to the partition functions of various statistical mechanics models --- the image of the map under any one of the above embeddings should converge to \emph{$\gamma$-Liouville quantum gravity ($\gamma$-LQG)} with parameter $\gamma \in (0,2]$ depending on the universality class of the random planar map. 
More precisely, the counting measure on vertices of the map, appropriately re-scaled, should converge to a version of the \emph{$\gamma$-LQG measure}, a random fractal measure on the plane constructed e.g.\ in~\cite{shef-kpz} (see also~\cite{kahane,rhodes-vargas-review}). 
The idea that LQG should describe (in some sense) the scaling limit of random planar maps dates back at least to Polyakov in the 1980's~\cite{polyakov-qg1,polyakov-qg2}. Various precise conjectures along these lines are stated, e.g., in~\cite{shef-kpz,shef-zipper,wedges,hrv-disk,curien-glimpse}.
We will not review the definition of LQG here. See the above references for more detail.

For each $\gamma \in (0,2)$, there is a certain special LQG surface called the \emph{0-quantum cone} whose corresponding $\gamma$-LQG measure $\mu$ satisfies an exact continuum analog of translation invariance modulo scaling~\cite[Proposition 4.13(ii)]{wedges} (see also~\cite[Proposition 3.1]{gms-tutte}). Roughly speaking, a 0-quantum cone is an infinite volume $\gamma$-LQG surface centered at a ``Lebesgue typical'' point.
One can use the 0-quantum cone to construct embedded lattices or cell configurations which approximate LQG and are translation invariant modulo scaling in the sense of Definition~\ref{def-translation-invariance}. For example, we could sample a dyadic system $\mcl D$ independent from $\mu$ and look at the embedded lattice whose faces are the maximal dyadic squares of $\mcl D$ with $\mu$-mass smaller than 1. 
Other LQG surfaces are not exactly translation invariant modulo scaling, but locally look like the 0-quantum cone in the sense of local absolute continuity. 
 
It is easy to construct embeddings of random planar maps which converge to LQG if the embedding is not required to be defined in an instrinsic way with respect to the map. As a trivial example, if we wanted to embed a map with $n$ vertices we could sample $n$ i.i.d.\ points from the $\gamma$-LQG measure, send one vertex to each of these points, then draw in the edges arbitrarily. More useful examples arise from the theory developed in~\cite{wedges}, as we discuss just below. 

The results of the present paper tell us that if we can find \emph{some} a priori embedding of a random planar map or its dual into $\BB C$ under which (a) the embedded map converges to $\gamma$-LQG and (b) the hypotheses of one of our main theorems are satisfied, then the Tutte embedding will be close to the a priori embedding at large scales, and hence the map will also converge to $\gamma$-LQG under the Tutte embedding. Moreover, one also gets convergence of random walk on the map to Brownian motion under both the a priori embedding and the Tutte embedding.  

The above approach is used in~\cite{gms-tutte} to show that the Tutte embedding of the \emph{mated-CRT maps}, a family of random planar maps constructed by gluing together a pair of correlated continuum random trees (CRT's), converges to $\gamma$-LQG, with $\gamma \in (0,2)$ determined by the correlation parameter. This constitutes the first rigorous proof that embedded random planar maps converge to LQG. 
 
The needed a priori embedding comes from the main result of~\cite{wedges}, which gives a correspondence between pairs of correlated CRT's and LQG surfaces decorated by space-filling Schramm-Loewner evolution (SLE$_\kappa$) curves for $\kappa =16/\gamma^2$. In particular, if $\mu$ is an appropriate version of the $\gamma$-LQG measure and $\eta$ is a whole-plane space-filling SLE$_\kappa$ sampled independently from $\mu$ and then parameterized so that it traverses one unit of $\mu$-mass in one unit of time, then~\cite[Theorem 1.9]{wedges} implies that the mated-CRT map is isomorphic to the adjacency graph of unit $\mu$-mass ``cells" $\eta([x-1,x])$ for $x\in\BB Z$. Theorem~\ref{thm-general-clt} is used to show that the random walk on this adjacency graph of cells converges to Brownian motion modulo time parameterization, and hence that the a priori embedding is close to the Tutte embedding. (It is expected---but not yet proven---that the random walk on this adjacency graph of cells converges uniformly to \emph{Liouville Brownian motion}, the ``quantum" time parameterization of Brownian motion constructed in~\cite{grv-lbm,berestycki-lbm}.) We also note that the paper~\cite{gms-harmonic} proves more quantitative versions of many of the estimates in this paper (e.g., for the Dirichlet energy and modulus of continuity of discrete harmonic functions) in the special case of the cell configuration associated with a mated-CRT map. 

Another application of the above approach appears in \cite{gms-poisson-voronoi}, which considers a random planar map whose vertices correspond to the cells of a Poisson-Voronoi tessellation of the Brownian map (or some other Brownian surface, like the Brownian disk), with two vertices considered to be adjacent if the corresponding cells intersect. As in~\cite{gms-tutte}, it is proven that the Tutte embeddings of these maps converge to $\sqrt{8/3}$-Liouville quantum gravity in the appropriate sense when the intensity of the Poisson point process tends to $\infty$. In this case, the a priori embedding comes from the results of~\cite{lqg-tbm1,lqg-tbm2,lqg-tbm3}, which show that there is a way of embedding of a Brownian surface into the plane under which it is equivalent to $\sqrt{8/3}$-Liouville quantum gravity (but do not give a direct description of this embedding in terms of the Brownian surface). We also show that ``Brownian motion on a Brownian surface'' (an object that was previously constructed in a somewhat indirect way) can obtained as the scaling limit of simple random walk on these cells. In~\cite{afs-metric-ball}, these results are extended to the case when the Brownian map is replaced by a $\gamma$-LQG surface for general $\gamma \in (0,2)$, viewed as a metric measure space. 

As we discuss in more detail in~\cite{gms-tutte,gms-poisson-voronoi}, there are many additional potential applications of our results in the theory of random planar maps and LQG along these same lines.

\appendix

\section{Specific Dirichlet energy}
\label{sec-specific}

Let $\mcl M$ be an embedded lattice (Definition~\ref{def-embedded-map}) which is translation invariant modulo scaling (Definition~\ref{def-translation-invariance}). 
Recall that the \emph{specific Dirichlet energy} associated with $\mcl M$ is the ``expected amount of discrete Dirichlet energy per unit area" for the function $\mcl V(\mcl M) \rta\BB C$ which takes each vertex of $\mcl M$ to its position in $\BB C$. 

As mentioned in the introduction, there are various ways to make this definition precise, depending how one chooses the origin-containing region and how each edge's energy is ``localized''.  In this appendix, we formalize this family of definitions (which includes the definition presented as  \eqref{eqn-dual-hyp-moment} within the statement of Theorem~\ref{thm-dual-clt}). We then use mass transport to briefly explain why the different definitions are all equivalent.

To make the notion of specific Dirichlet energy precise, the Dirichlet energy $\frk c(e) |x-y|^2$ associated to each edge $e = \{x,y\} \in \mcl E(\mcl M)$ needs to be spread out in space in some way. That is, we assume we have a way to associate to each edge a measure $\nu_{\mcl M}^e$ on the plane of total mass $\frk c(e) |x-y|^2$, where the map from $(\mcl M, e)$ to $\nu_{\mcl M}^e$ commutes with dilation and translation in the obvious way: $\nu_{C(\mcl M-z)}^{C(e-z)}(\cdot) = C^2 \nu_{\mcl M}^e(C^{-1} \cdot)$. We set $\nu = \nu_{\mcl M} = \sum_{e\in\mcl E \mcl M} \nu_{\mcl M}^e$. 

One concrete way to do this is to let $\nu_{\mcl M}^e$ spread half its mass uniformly over each of the two faces incident to $e$. Then the density of $\nu$ on a face $H$ is given by $\frac12 \sum \frk c(e) \diam(e)^2/\area(H)$, where the sum is over all edges on the boundary of $H$. Another example is to let $\nu_{\mcl M}^e$ be a sum of two point masses (at the two endpoints of $e$), or a single point mass at some specified internal point on $e$. 

\begin{defn} \label{def-specific}
If the law of $\mcl M$ is translation invariant modulo scaling, then the {\em specific Dirichlet energy} of $\mcl M$ is defined as $\Bbb E[\nu (B)/\area(B)]$ where $\nu$ is as described above and $B$ is the origin-containing block in any block partition $\mcl B(\mcl M)$ of the sort defined in Condition~\ref{item-block} of Definition~\ref{def-translation-invariance} such that a.s.\ the $\nu$ measure of the boundary of any block is zero.\footnote{This definition is independent of the choice of block partition $\mcl B$ and the energy location rule $\nu$, as can easily be see, e.g., from the mass transport principle (condition~\ref{item-mass-transport} of Definition~\ref{def-translation-invariance}).} Equivalent definitions of the specific Dirichlet energy (obtained as special cases) are as follows.
\begin{enumerate}
\item The expected density of $\nu$ at $0$ (for any choice of $\nu$ that is a.s.\ absolutely continuous with respect to Lebesgue measure).
\item With $H_0$ the origin-containing face,
\eqbn
\Bbb E \left[  \sum_{e = \{x,y\} \in \mcl E\mcl M \cap \bdy H_0} \frk c(e) \frac{ \diam(e)^2}{\area(H_0)} \right]  . 
\eqen  
\end{enumerate}
\end{defn}

Now we will sketch an argument for the fact that Definition~\ref{def-specific} does not depend on the specific way the measures $\nu$ and the blocks $B$ are defined. We first observe that the definition is not changed if we replace each $\nu_{\mcl M}^e$ with a measure that spreads its mass {\em uniformly} over each block $B$ to which it assigns measure, but otherwise assigns the same measure to each block $B$. Such a change would not impact the value of $\nu(B)$ for the origin-containing block $B$ referenced in Definition~\ref{def-specific}. We next observe that if we had a different measure $\tilde \nu$ with this property (for a different block partition), we could define a mass transport function $G_e$ such that (for a given edge $e$) any amount of mass associated to $\nu_{\mcl M}^e$ is spread evenly according to the measure $\tilde \nu_{\mcl M}^e$.  In other words, $\int_A \int_B G_e(a,b) \, da \, db = \nu_{\mcl M}^e(A) \tilde \nu_{\mcl M}^e(B)$.  Summing these $G_e$ over all edges $e$, one obtains a mass transport function $F$ on the whole plane.  The difference in the density of $\tilde \nu$ and $\nu$ at zero is given by the net amount of flow in the transport function and since (by the assumption in Definition~\ref{def-translation-invariance}) this has expectation zero, the two definitions of specific free energy must agree.

\begin{remark}
In Definition~\ref{def-specific}, the Dirichlet energy associated with an edge $e = \{x,y\}$ is taken to be $\frk c(e) |x-y|^2$. 
However, in some of the proofs involving specific Dirichlet energy (and related quantities) we need bounds for $\frk c(e) \op{diam}(e)^2$ instead, which is larger than $\frk c(e) |x-y|^2$ unless the edges are straight lines.
Hence the finiteness conditions in those theorem statements involve bounds for diameters rather than just distance between vertices.
\end{remark}

\section{Equivalence of translation invariance modulo scaling definitions}
\label{sec-equivalence}

In this appendix we prove Lemma~\ref{lem-dyadic-resample}, which asserts the equivalence of the conditions of Definition~\ref{def-translation-invariance} plus the extra condition of Lemma~\ref{lem-dyadic-resample} in the setting of a cell configuration (analogous equivalence results for embedded lattices or pairs of primal/dual embedded lattices follow from identical arguments). We will prove that each of the conditions of Definition~\ref{def-translation-invariance} is equivalent to the dyadic system resampling condition (condition~\ref{item-dyadic} from Lemma~\ref{lem-dyadic-resample}).
\medskip

\noindent \textbf{Condition~\ref{item-inf-volume} $\Rightarrow$ Condition~\ref{item-dyadic}}. 
Fix $m > 0$. Let $\mcl H_n$, $z_n$, and $C_n$ be as in Condition~\ref{item-inf-volume} with cell configurations in place of lattices, so that each $\mcl H_n$ is a cell configuration with the property that the union of its cells is a disk or square (instead of all of $\BB C$) and $C_n(\mcl H_n -z_n) \rta \mcl H$ in law.  
For $n\in\BB N$, let $\mcl D_n$ be a uniform dyadic system independent from $(\mcl H_n, z_n, C_n)$ and let $\wh S_{m,n}^{z_n}$ be defined as in~\eqref{eqn-mass-time-def} with $\mcl H_n$ in place of $\mcl H$ and $z = z_n$, or $\wh S_{m,n}^{z_n} = \BB C$ if no square satisfying the condition of~\eqref{eqn-mass-time-def} exists (which can happen since $\mcl H_n$ is finite). 
By Lemma~\ref{lem-uniform-dyadic}, $C_n(\mcl D - z_n)$ is a uniform dyadic system independent from $C_n(\mcl H_n - z_n)$. 
The square $C_n(\wh S_{m,n}^{z_n} -z_n)$ is determined by $(C_n (\mcl H_n-z_n) , C_n(\mcl D_n-z_n))$ in the same continuous manner that $\wh S_m$ is determined by $(\mcl H,\mcl D)$. Thus
\eqb \label{eqn-dyadic-law-conv'}
\left( C_n (\mcl H_n-z_n) , C_n(\mcl D_n-z_n) ,  C_n(\wh S_{m,n}^{z_n} -z_n) \right) \rta \left( \mcl H,\mcl D , \wh S_m  \right)
\eqe 
in law, with respect to the topology on cell configurations on the first coordinate and any choice of topology on the other two coordinates. It follows from~\eqref{eqn-dyadic-law-conv'} that with probability tending to 1 as $n\rta\infty$, the square $\wh S_{m,n}^{z_n}$ is contained in the union of the cells of $\mcl H_n$. Since $z_n$ is a uniform sample from this union, it follows that the total variation distance between the conditional law of $z_n$ given $(\mcl H_n ,\mcl D_n , \wh S_{m,n}^{z_n})$ and the uniform measure on $\wh S_{m,n}^{z_n}$ a.s.\ converges to 0 as $n \rta \infty$. Combining this with~\eqref{eqn-dyadic-law-conv'} shows that Condition~\ref{item-dyadic} holds.  \qed
\medskip

\noindent\textbf{Condition~\ref{item-dyadic} $\Rightarrow$ Condition~\ref{item-inf-volume}}. 
For $ m > 0$, define a cell configuration on $\wh S_m$ by
\eqbn 
\mcl H_m := \left\{ H\cap \wh S_m : H\in \mcl H , H\subset \wh S_m \right\} \cup \left\{ L_m \right\} \quad \text{for} \quad L_m :=  \ol{\wh S_m\setminus \bigcup_{H \in \mcl H : H\subset \wh S_m} H }
\eqen
Then $\mcl H_m$ is a collection of cells whose union is the square $\wh S_m$. We endow $\mcl H_m \cap \mcl H$ with the adjacency relation and the conductances it inherits from $\mcl H$ and we declare that $H \sim L_m$ for each $H\in\mcl H_m\cap \mcl H$ with $H\cap L_m\not=\emptyset$. 

For $m\in\BB N$, let $z_m$ be sampled uniformly from Lebesgue measure on $\wh S_m$. 
By Condition~\ref{item-dyadic}, for each $m>0$ there exists $C_m > 0$ such that $C_m (\mcl H - z_m) \eqD \mcl H$. 
We claim that $C_m(\mcl H_m-z_m) \rta \mcl H$ in law. 
Indeed, if $\el > 0$ and $\wh S_\el^{z_m} \subset \wh S_m \setminus L_m$, then $\mcl H(\wh S_\el^{z_m})= \mcl H_m(\wh S_\el^{z_m})$, i.e., the analogs of $\mcl H(\wh S_\el)$ for the scaled/translated cell configurations $C_m (\mcl H - z_m)$ and $C_m(\mcl H_m - z_m)$ are identical. 
Since the union of $\wh S_\el$ over all $\el > 0$ is all of $\BB C$ and  $C_m(\mcl H-z_m) \eqD \mcl H$, it suffices to show that for each fixed $\el > 0$, one has $\lim_{m\rta\infty} \BB P\left[ \wh S_\el^{z_m} \subset \wh S_m \setminus L_m \right] = 1$. By the definition of $L_m$, this will follow from
\eqb \label{eqn-iv-lim-square}
\lim_{m\rta\infty} \BB P\left[ \bigcup_{H\in \mcl H(\wh S_\el^{z_m})} H \subset \wh S_m   \right]  = 1 .
\eqe
To prove~\eqref{eqn-iv-lim-square}, we observe that Condition~\ref{item-dyadic} implies that 
\eqbn
\frac{1}{|\wh S_m|} \op{diam}\left( \bigcup_{H\in \mcl H(\wh S_\el^{z_m})} H \right)   \eqD \frac{1}{|\wh S_m|} \op{diam}\left( \bigcup_{H\in \mcl H(\wh S_\el )} H \right) ,
\eqen 
which tends to zero in probability as $m\rta\infty$ ($\el$ fixed). Since $z_m$ is sampled uniformly from $\wh S_m$ (so is unlikely to be close to $\bdy \wh S_m$), this implies~\eqref{eqn-iv-lim-square}.  \qed

\medskip

\noindent\textbf{Condition~\ref{item-averaging} $\Rightarrow$ Condition~\ref{item-dyadic}}. 
The proof is similar to the analogous implication started from Condition~\ref{item-inf-volume}.  
Fix $m > 0$. Let $z_n$ be sampled uniformly from $U_n$, as in condition~\ref{item-averaging}, and let $C_n>0$ be chosen so that $C_n(\mcl H-z_n) \rta \mcl H$ in law.  
By Lemma~\ref{lem-uniform-dyadic}, $C_n(\mcl D - z_n)$ is a uniform dyadic system independent from $C_n(\mcl H - z_n)$. Using the notation~\eqref{eqn-mass-time-def}, the square $C_n(\wh S_m^{z_n}-z_n)$ is determined by $(C_n (\mcl H-z_n) , C_n(\mcl D-z_n))$ in the same continuous manner that $\wh S_m$ is determined by $(\mcl H,\mcl D)$. Thus
\eqb \label{eqn-dyadic-law-conv}
\left( C_n (\mcl H-z_n) , C_n(\mcl D-z_n) , C_n(\wh S_m^{z_n} - z_n) \right) \rta \left( \mcl H,\mcl D , \wh S_m  \right)
\eqe 
in law, with respect to the topology on cell configurations~\eqref{eqn-cell-metric} on the first coordinate, any choice of topology on the second coordinate, and the Hausdorff distance on the last coordinate.  

We will now argue that 
\eqb \label{eqn-square-prob-lim}
\lim_{n\rta\infty} \BB P[ \wh S_m^{z_n} \subset U_n] = 1 .
\eqe
Note that this is not obvious since we do not assume that the maximum diameter of the cells of $\mcl H$ which intersect $U_n$ is $o(\op{diam}(U_n))$, so it is a priori possible that $\wh S_m^{z_n}$ could be comparable in size to $U_n$ (although we will show that this is not the case). 
Since $\mcl H$ is locally finite and the number of cells which intersect $U_n$ tends to $\infty$ as $n\rta\infty$, the convergence $C_n(\mcl H-z_n) \rta \mcl H$ implies that $C_n \op{diam}(U_n) \rta \infty$ in probability (for otherwise there would be a subsequence of $n$'s along which $C_n(\mcl H-z_n)$ has uniformly positive probability to contain an arbitrarily large number of cells which intersect some fixed compact set). Since $\wh S_m$ has finite side length, it follows from~\eqref{eqn-dyadic-law-conv} that the laws of the random variables $C_n |\wh S_m^{z_n}|$ are tight. Combining the preceding two sentences shows that $|\wh S_m^{z_n}| / \op{diam}(U_n) \rta 0$ in probability, so since $z_n$ is sampled uniformly from $U_n$ (so is unlikely to be close to $\bdy U_n$) we get~\eqref{eqn-square-prob-lim}. 

By~\eqref{eqn-square-prob-lim} and since $z_n$ is sampled uniformly from $U_n$, a.s.\ the total variation distance between the conditional law of $z_n$ given $(\mcl H,\mcl D , \wh S_m^{z_n})$ and the uniform measure on $\wh S_m^{z_n}$ tends to zero as $n\rta\infty$. Combining this with~\eqref{eqn-dyadic-law-conv} shows that Condition~\ref{item-dyadic} holds.  \qed
\medskip

\noindent\textbf{Condition~\ref{item-dyadic} $\Rightarrow$ Condition~\ref{item-averaging}}.
Sample a uniform dyadic system $\mcl D$ independent from $\mcl H$ and take $U_n = \wh S_n$ (as in~\eqref{eqn-mass-time-def}) for $n\in\BB N$.  \qed
\medskip

\noindent\textbf{Condition~\ref{item-block} $\Rightarrow$ Condition~\ref{item-dyadic}}.
Let $\mcl D$ be a dyadic system independent from $\mcl H$ and apply Condition~\ref{item-block} to the block decomposition $\{\wh S_m^z : z\in\BB C\}$, with $\wh S_m^z$ as in~\eqref{eqn-mass-time-def}.  \qed
\medskip

\noindent\textbf{Condition~\ref{item-dyadic} $\Rightarrow$ Condition~\ref{item-block}}.
Given a block decomposition $\mcl B$ as in Condition~\ref{item-block}, let $B_z$ for $z\in\BB C$ be the (a.s.\ unique) block of $\mcl B$ containing the $z$. 
Also let $\mcl D$ be a uniform dyadic system independent from $(\mcl H,\mcl B)$. 
By Condition~\ref{item-dyadic} (and the scaling condition on block decompositions), if $m>0$ and we sample $z_m$ uniformly from $\wh S_m$, then $(\mcl H -z_m , \mcl B-z_m , \mcl D-z_m)$ agrees in law with $(\mcl H, \mcl B , \mcl D)$ modulo spatial scaling. Since the block $B_0$ is a compact set and the union of the $\wh S_m$'s is all of $\BB C$, this shows that
\eqb \label{eqn-block-bdy-intersect}
\BB P\left[B_{z_m} \subset  \wh S_m   \right] = \BB P\left[ B_0 \subset \wh S_m  \right]   \rta 1  \quad \text{as} \quad m\rta \infty .
\eqe 
If we condition on $(\mcl H,\mcl D,\mcl B)$ and sample $w_m$ uniformly from Lebesgue measure on $B_{z_m} \cap \wh S_m$, then the conditional law of $w_m$ given $(\mcl H , \mcl B,\mcl D)$ is uniform on $\wh S_m $, so by Condition~\ref{item-dyadic} $\mcl H - w_m$ agrees in law with $\mcl H$ modulo scaling. 
On the other hand,~\eqref{eqn-block-bdy-intersect} implies that the total variation distance between the law of $(\mcl H , w_m)$ and the law of $(\mcl H , w)$ tends to zero as $m\rta\infty$. Therefore, $\mcl H-w$ agrees in law with $\mcl H$ modulo scaling.  \qed
\medskip

\noindent\textbf{Condition~\ref{item-mass-transport} $\Rightarrow$ Condition~\ref{item-dyadic}}. To prove Condition~\ref{item-dyadic}, it suffices to show that for each non-negative measurable function $G(\mcl H , w)$ of a cell configuration and a point in $\BB C$ which is invariant under translation and scaling in the sense that $G(C(\mcl H-z)) , C(w-z)) = G(\mcl H , w)$ for each $C>0$ and $z\in\BB C$, we have
\eqb \label{eqn-mass-transport-show}
\BB E\left[ \frac{1}{|\wh S_m|^2}\int_{\wh S_m} G(\mcl H , z ) \,dz \right] = \BB E\left[ G(\mcl H ,0) \right] ,\quad\forall m  >  0 . 
\eqe
To see why this is sufficient, suppose that~\eqref{eqn-mass-transport-show} holds and we are given an arbitrary function $\wh G$ from the space of cell configurations to $[0,\infty)$. Define $G(\mcl H , w) := \wh G\left( (\mcl H-w) / \op{diam}(H_w) \right)$, where $H_w$ is a cell of $\mcl H$ which contains $w$ (chosen by some deterministic convention if there is more than one). If $z_m$ is sampled uniformly from $\wh S_m$, then~\eqref{eqn-mass-transport-show} applied with this choice of $G$ shows that $\BB E[\wh G(\mcl H/ \op{diam}(H_0) )] = \BB E[\wh G( (\mcl H - z_m) / \op{diam}(H_{z_m} ) )]$. Thus $\mcl H \eqD  \mcl H-z_m $ modulo scaling. 

To prove~\eqref{eqn-mass-transport-show}, we fix $m\geq 0$ and define a function on cell configurations with two marked points by 
\eqbn
F(\mcl H, w_0,w_1) := \BB E\left[  |\wh S_m^{w_0}|^{-2} G(\mcl H , w_1 ) \BB 1_{(w_1 \in \wh S_m^{w_0})}  \,|\, \mcl H\right] .
\eqen
Note that $\{w_1 \in \wh S_m^{w_0}\} = \{\wh S_m^{w_1} =  \wh S_m^{w_0}\}$. 
Then $F$ is covariant with respect to translation and scaling in the sense of~\eqref{eqn-mass-transport-commutation}, so we can apply Condition~\ref{item-mass-transport} to $F$ to find that~\eqref{eqn-mass-transport-show} holds. \qed
\medskip

\noindent\textbf{Condition~\ref{item-dyadic} $\Rightarrow$ Condition~\ref{item-mass-transport}}. 
For $m\in\BB N$, let $z_m$ be sampled uniformly from Lebesgue measure on $\wh S_m$, as in Condition~\ref{item-dyadic}. 
This condition tells us that there is a random $C_m > 0$ such that $(C_m(\mcl H - z_m) , C_m( \mcl D -z_m))  \eqD (\mcl H , \mcl D)$.
If we condition on $(C_m(\mcl H - z_m) , C_m( \mcl D -z_m)) $, then the conditional law of the point $- C_m z_m$ (which corresponds to the origin) is uniform on $C_m(\wh S_m -z_m)$. Thus $ ( C_m(\mcl H-z_m) ,   C_m( \wh S_m - z_m)  ,  - C_m z_m ) \eqD (\mcl H ,   \wh S_m , z_m)$.
Hence, if $F(\mcl H ,w_1,w_2)$ is any non-negative measurable function of a cell configuration and a pair of points, then
\allb \label{eqn-mass-transport-square}
\BB E\left[   \int_{\wh S_m} F\left( \mcl H , 0 , z \right) \,dz \right] 
&= \BB E\left[   |\wh S_m|^2 F\left( \mcl H , 0 , z_m \right)  \right] \notag\\
&= \BB E\left[  C_m^2 |\wh S_m|^2  F\left(  C_m(\mcl H-z_m) ,  0 ,   - C_m z_m    \right)  \right]  \notag \\ 
&= \BB E\left[  \int_{C_m(\wh S_m - z_m)}   F\left(  C_m(\mcl H-z_m) ,  0 , z  \right)  \,dz \right] \notag\\
&= \BB E\left[  C_m^2 \int_{\wh S_m }   F\left(  C_m(\mcl H-z_m) , 0 ,   C_m(u-z_m)   \right)  \,du \right] ,
\alle
where in the last line we have made the change of variables $z = C_m(u-z_m)$, $dz = C_m^{  2} \,du$. 
If $F$ satisfies~\eqref{eqn-mass-transport-commutation}, then 
$C_m^2 F\left(  C_m(\mcl H-z_m) , 0 ,   C_m(u-z_m)   \right) =   F\left(  \mcl H  ,   z_m , u \right)  $
so~\eqref{eqn-mass-transport-square} gives
\allb  \label{eqn-mass-transport-int}
\BB E\left[   \int_{\wh S_m} F\left( \mcl H , 0 , z \right) \right] 
=  \BB E\left[  \int_{\wh S_m }   F\left( \mcl H , z_m , u   \right)  \,du \right] 
= \BB E\left[  \frac{1}{|\wh S_m|^2} \int_{\wh S_m  \times \wh S_m}  F\left(  \mcl H , v , u   \right)  \,du \, dv \right]   ,  
\alle
here using that $z_m$ is uniform  on $\wh S_m$ conditional on $(\mcl H,\mcl D)$. 
The same argument shows that the right side of~\eqref{eqn-mass-transport-int} is also equal to $\BB E\left[  \int_{\wh S_m} F\left( \mcl H , z , 0 \right) \,dz \right]$, so we get
\eqbn
\BB E\left[ \int_{\wh S_m} F\left( \mcl H , 0 , z \right) \,dz \right] = \BB E\left[   \int_{\wh S_m} F\left( \mcl H , z,0 \right) \,dz \right] ,\quad \forall m >0.
\eqen
Sending $m\rta\infty$ now shows that~\eqref{eqn-mass-transport} holds. \qed
\medskip

\section{No macroscopic edges in Theorem~\ref{thm-dual-clt}}
\label{sec-dual-max-diam}

In this appendix we prove Lemma~\ref{lem-dual-max-diam}, which says that the hypotheses of Theorem~\ref{thm-dual-clt} imply that the maximal diameter of the edges of the embedded lattice $\mcl M$ and of its dual $\mcl M^*$ which intersect $B_r(0)$ is a.s.\ of order $o_r(r)$. The proof of this assertion is much more difficult than the proof of the analogous assertion in the settings of the other theorems (see Lemma~\ref{lem-max-cell-diam}) since we only have bounds for edge diameters weighted by conductances, so one needs to rule out the possibility of macroscopic edges with small conductances. 

We emphasize that nothing in this appendix is needed for the proofs of our other main results (Theorems~\ref{thm-general-clt0}, \ref{thm-graph-clt}, and~\ref{thm-general-clt}) or in the companion paper~\cite{gms-tutte}. 

Throughout this appendix we identify the real axis with $\BB R$. We also fix a uniform dyadic system $\mcl D$ independent from $(\mcl M ,\mcl M^*)$ (Section~\ref{sec-dyadic-system}) and let $\{S_k\}_{k\in\BB Z}$ be its sequence of origin-containing squares.

The basic idea of the proof is to study the edges of $\mcl M^*$ whose corresponding primal edges intersect $\mcl M$. As we will explain just below, simple geometric considerations imply that this set of dual edges contains a bi-infinite simple path $\{\frk e^*(j)\}_{j\in\BB Z}$. The main step of the proof is to show that this path does not deviate too far from $\BB R$, in the sense that for each $\ep > 0$, it holds for large enough $k\in\BB N$ that $\{\frk e^*(j)\}_{j\in\BB Z}$ admits a sub-path which joins the left and right boundaries of $S_k$ and stays in the $\ep |S_k|$-neighborhood of $S_k\cap\BB R$ (Lemma~\ref{lem-dual-path}). This will be accomplished in two main steps: we first show that a.s.\ for large enough $k$, each point of $S_k\cap \BB R$ is contained in a small (side length of order $o_k(|S_k|)$) square of $\mcl D$ which intersects one of the $\frk e^*(j)$'s and satisfies an additional technical property (Lemma~\ref{lem-max-good-square}). We will then argue that the excursions of the path $\{\frk e^*(j)\}_{j\in\BB Z}$ between the times when it hits these squares cannot be macroscopic (Lemma~\ref{lem-max-line-diam}). Both of these statements are proven using variants of the argument used to prove Lemma~\ref{lem-max-cell-diam}. 

Once we know that the maximum deviation of the path $\{\frk e^*(j)\}_{j\in\BB Z}$ from $\BB R$ is macroscopic, we will use translation invariance modulo scaling to construct a ``grid" of paths in $\mcl M^*$ which approximate horizontal and vertical lines in $S_k$ such that each of the complementary connected components of the union of these paths has diameter $o_k(|S_k|)$. Since no dual edge can cross any of the paths, this will conclude the proof of Lemma~\ref{lem-dual-max-diam}.

Since we will mostly be working with $\BB R$ rather than $\BB C$, we will first set up an analog of Lemma~\ref{lem-ergodic-avg} for averages along $\BB R$. 
 
\begin{defn} \label{def-line-sigma-algebra}
Recall the squares $\wh S_m = \wh S_m^0$ for $m  >0$ from~\eqref{eqn-mass-time-def} (defined with faces of $\mcl M$ in place of cells of $\mcl H$).  
For $m\in\BB N$, let $\mcl F_m^{\BB R}$ be the $\sigma$-algebra generated by the measurable functions of $F = F(\mcl H,\mcl D)$ which satisfy 
\eqb \label{eqn-line-sigma-algebra}
F\left(C(\mcl H-z) , C(\mcl D-z) \right) = F(  \mcl H,   \mcl D) ,\quad \forall z\in \wh S_m \cap \BB R  ,\quad \forall C >0 .
\eqe
Also set
\eqb \label{eqn-line-sigma-algebra-infty}
\mcl F_\infty^{\BB R} := \bigcap_{m > 0} \mcl F_m^{\BB R} .
\eqe 
\end{defn}

Unlike in the setting of Section~\ref{sec-ergodic-avg}, we do \emph{not} know that the tail $\sigma$-algebra $\mcl F_\infty^{\BB R}$ is trivial since functions in $\mcl F^{\BB R}_\infty$ are only invariant under translation along the line $\BB R$.

\begin{lem} \label{lem-ergodic-avg-line}
Let $F = F(\mcl M,\mcl M^* ,\mcl D)$ be a measurable function on the space of embedded lattice/dual lattice/dyadic system pairs which is scale invariant (i.e., $F(C\mcl M, C\mcl M^* , C\mcl D) = F(\mcl M,\mcl M^* ,\mcl D)$ for each $C>0$). 
If either $\BB E[|F|] < \infty$ or $F\geq 0$ a.s., then a.s.\ 
\eqb \label{eqn-ergodic-avg-line}
\lim_{k \rta\infty} \frac{1}{|S_{k}| } \int_{S_{k} \cap \BB R} F(\mcl M-x ,  \mcl M^* - x ,  \mcl D-x) \, dx  = \BB E\left[ F \,|\, \mcl F_\infty^{\BB R} \right]  .
\eqe  
\end{lem}
\begin{proof} 
The analog of~\ref{lem-dyadic-resample} in our setting implies that if $x$ is sampled uniformly from $\wh S_m \cap \BB R$, then $(\mcl M-x,\mcl M^*-x)$ agrees in law with $ (\mcl M,\mcl M^*)$ 
modulo scaling. The lemma therefore follows from the backward martingale convergence theorem via exactly the same proof as Lemma~\ref{lem-ergodic-avg}.
\end{proof}

\begin{figure}[ht!]
 \begin{center} 
\includegraphics[scale=.75]{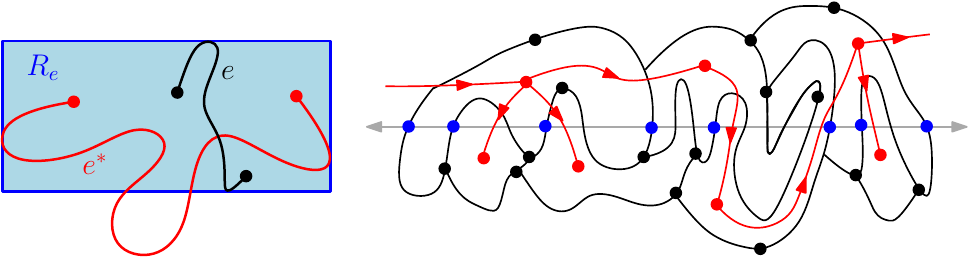}  
\vspace{-0.01\textheight}
\caption{\textbf{Left:} 
A primal edge/dual edge pair $(e,e^*)$ and the corresponding rectangle $R_e$, defined as in~\eqref{eqn-rectangle-def}.
\textbf{Right:} 
Illustration of the construction of the path $\{\frk e^*(j)\}_{j\in\BB Z}$. The faces $H$ of $\mcl M$ which intersect the real axis are shown in black and the corresponding first intersection points $\wt{\frk x}(H)$ are shown in blue. The edges $\wt{\frk e}^*(i)$ for $i\in\BB Z$ are shown in red. These edges form a tree. We take $\{\frk e^*(j)\}_{j\in\BB Z}$ to be the lowest bi-infinite path in this tree.
}\label{fig-dual-diam-path}
\end{center}
\vspace{-1em}
\end{figure} 

We will now prove a re-formulation of our finite expectation hypothesis~\eqref{eqn-dual-hyp-moment} which will be more convenient for our purposes.
For $e\in \mcl E\mcl M$, define the rectangle
\eqb \label{eqn-rectangle-def}
R_e :=      \left[ \min_{z\in e^*} \re z , \max_{z\in e^*} \re z \right]  \times    \left[ \min_{z\in e} \im z , \max_{z\in e} \im z \right] .
\eqe 
That is, the width (resp.\ height) of $R_e$ is the same as the width (resp.\ height) of $e^*$ (resp.\ $e$). See Figure~\ref{fig-dual-diam-path}, left. Note that $e$ and $e^*$ are not necessarily contained in $R_e$, but that $R_e$ intersects $\BB R$ if and only if $e$ intersects $\BB R$. 
For $z\in\BB C$, let 
\eqb \label{eqn-edge-set-def}
\mcl E_z :=  \left\{ e \in \mcl E\mcl M : R_e \ni z \right\} .
\eqe

\begin{lem} \label{lem-primal-dual-prod}
In the notation introduced just above, we have
\eqbn
\BB E\left[ \sum_{e \in \mcl E_0} \frac{\op{diam}(e) \op{diam}(e^*)}{ \op{Area}(R_e)} \right] < \infty .
\eqen
In particular, $\BB E[\# \mcl E_0] < \infty$.
\end{lem}
\begin{proof}
We will use the mass transport principle (Condition~\ref{item-mass-transport} of Definition~\ref{def-translation-invariance} with $(\mcl M,\mcl M^*)$ in place of $\mcl M$) and the Cauchy-Schwarz inequality to deduce the lemma from the finite expectation hypothesis~\eqref{eqn-dual-hyp-moment}. 
For $w_0,w_1\in\BB C$, define 
\eqbn
F(\mcl M,\mcl M^*, w_0,w_1)  := 
\sum_{e\in \mcl E\mcl M\cap \bdy H_{w_0}} \frac{ \op{diam}(e^*)^2 \frk c^*(e^*)   }{\op{Area}(H_{w_0}) \op{Area}(R_e) } \BB 1_{\left( w_1 \in R_e    \right)}  .  
\eqen
Then $ F$ is covariant with respect to scaling and translation in the sense of~\eqref{eqn-mass-transport-commutation}. 
We compute
\eqbn
\int_{\BB C} F(\mcl M,\mcl M^*, 0 , z)   \,dz
 = \sum_{e\in \mcl E\mcl M\cap \bdy H_{0}} \frac{ \op{diam}(e^*)^2 \frk c^*(e^*)   }{\op{Area}(H_{0}) }     
\eqen
and 
\alb
\int_{\BB C} F(\mcl M,\mcl M^*, z , 0)   \,dz
=  2 \sum_{e\in \mcl E_0}    \frac{ \op{diam}(e^*)^2 \frk c^*(e^*)   }{  \op{Area}(R_e) }  ,
\ale
where here we use that each edge $e\in\mcl E_0$ is contained in exactly two faces $H\in\mcl F\mcl M$. 
By the mass transport principle and the hypothesis~\eqref{eqn-dual-hyp-moment}, it therefore follows that
\eqb \label{eqn-primal-dual-prod1}
2 \BB E\left[   \sum_{e\in \mcl E_0}    \frac{ \op{diam}(e^*)^2 \frk c^*(e^*)   }{  \op{Area}(R_e) } \right]  = \BB E\left[  \sum_{e\in \mcl E\mcl M\cap \bdy H_0}   \frac{\op{diam}(e^*)^2 \frk c^*(e^*)}{\op{Area}(H_{0})  }    \right] < \infty .
\eqe 
Similarly,  
\eqb \label{eqn-primal-dual-prod2}
2\BB E\left[   \sum_{e\in \mcl E_0}    \frac{ \op{diam}(e )^2 \frk c (e )   }{  \op{Area}(R_e) } \right]  = \BB E\left[  \sum_{e\in \mcl E\mcl M^*\cap \bdy H_0^*}   \frac{\op{diam}(e )^2 \frk c(e )}{\op{Area}(H_{0}^*)  }    \right] < \infty .
\eqe  
Since $\frk c(e) = \frk c^*(e^*)^{-1}$, we can now apply the Cauchy-Schwarz inequality to get
\alb
\BB E\left[ \sum_{e \in \mcl E_0} \frac{\op{diam}(e) \op{diam}(e^*)}{ \op{Area}(R_e)} \right]
 \leq   \BB E\left[   \sum_{e\in \mcl E_0}    \frac{ \op{diam}(e^*)^2 \frk c^*(e^*)   }{  \op{Area}(R_e) } \right]^{1/2} \BB E\left[   \sum_{e\in \mcl E_0}    \frac{ \op{diam}(e )^2 \frk c (e )   }{  \op{Area}(R_e) } \right]^{1/2} < \infty .
\ale 
This implies that $\BB E[\# \mcl E_0] < \infty$ since $\op{Area}(R_e) \leq \op{diam}(e) \op{diam}(e^*)$.
\end{proof}

Let us now define the path we will consider. See Figure~\ref{fig-dual-diam-path}, right, for an illustration. 
 For each face $H$ of $\mcl M$ with $H\cap \BB R\not=\emptyset$, let $\wt{\frk x}(H)$ be the smallest $x\in \BB R$ with $x\in H$. We can enumerate $\{H\in \mcl F\mcl M   :H\cap \BB R\not=\emptyset\} = \{H(i) \}_{i\in\BB Z}$ in such a way that $\wt{\frk x}(H(i-1)) \leq \wt{\frk x}(H(i ))$ for each $i \in\BB Z$ and $H(0) = H_0$ is the origin-containing face.
Let $\wt{\frk e}(i)$ be the (a.s.\ unique, by translation invariance modulo scaling and since $\mcl V\mcl M$ is countable) edge containing $\wt{\frk x}(i)$ and let $\wt{\frk e}^*(i)$ be the corresponding dual edge.

Since $\wt{\frk e}(i)$ lies on the boundary of $H(i)$ and some face $H(i')$ with $i' \leq i-1$, the dual edge $\wt{\frk e}^*(i)$ joins the dual vertex corresponding to $H(i)$ and the dual vertex corresponding to $H(i')$. If we orient $\wt{\frk e}^*(i)$ toward the dual vertex corresponding to $H(i)$, then each of the edges $\wt{\frk e}^*(i)$ for $i\in\BB Z$ is oriented from a dual vertex with a lower index to a dual vertex with a higher index. Consequently, the collection of edges $\{\wt{\frk e}^*(i) : i\in\BB Z\}$ has no cycles, so is a tree. By planarity, this tree admits a unique bi-infinite simple oriented path which lies below every other such bi-infinite simple oriented path. (This path can be found via a depth-first search where we always start from the rightmost unexplored edge emanating from the current vertex.) 

Let $\{\frk e^*(j)\}_{j\in\BB Z}$ be this bi-infinite path. For $j\in\BB Z$ let $\frk e(j)$ be the primal edge which crosses $\frk e^*(j)$ and let $\frk x(j)$ be the leftmost $x\in\BB R$ with $x \in \frk e(j)$. We assume that the enumeration is chosen so that $0\in [\frk x(0), \frk x(1)]$. 

We emphasize that the above definition of the path $\{\frk e^*(j)\}_{j\in\BB Z}$ is determined by $(\mcl M,\mcl M^*)$ in a translation and scale invariant manner, modulo index shifts, i.e., replacing $(\mcl M,\mcl M^*)$ by $(C(\mcl M-x), C(\mcl M^*-x))$ for $C>0$ and  $x\in\BB R$ gives us the path $\{C(\frk e(j)-x) \}_{j\in\BB Z}$, modulo an index shift.

\begin{figure}[ht!]
 \begin{center} 
\includegraphics[scale=.75]{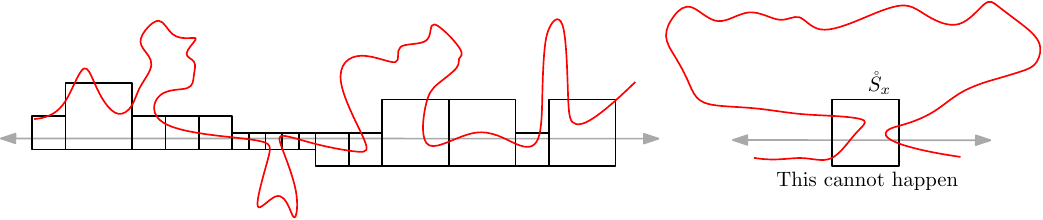} \hspace{15pt} 
\vspace{-0.01\textheight}
\caption{\textbf{Left:} The path $\{\frk e^*(j)\}_{j\in\BB Z}$ (red) and the squares $\rng S_x \in \mcl D$ for $x\in\BB R$ which it intersects (black). The set $\mcl A$ consists of excursions of the path away from the union of the squares $\rng S_x$. \textbf{Right:} The squares $\rng S_x$ for $x\in\BB R$ are chosen in such a way that $\{\frk e^*(j)\}_{j\in\BB Z}$ cannot cross between the vertical lines containing the left and right boundaries of $\rng S_x$ without hitting $\rng S_x$. This condition is needed to rule out pathological behavior of the sort shown in the figure.
}\label{fig-dual-diam-setup}
\end{center}
\vspace{-1em}
\end{figure}

We now introduce certain special squares of $\mcl D$ which we will use to ``pin down" the path $\{\frk e^*(j)\}_{j\in\BB Z}$ to $\BB R$. 

\begin{lem} \label{lem-rectangle-sep}
Almost surely, for each $x \in \BB R$ there exists a square $S$ of $\mcl D$ containing $x$ with the following property: Every path of edges $e^* \in \mcl E\mcl M^*$ whose corresponding dual edges $e$ intersect $\BB R$ and which crosses the vertical line through the left boundary of $S$ and the vertical line through the right boundary of $S$ must contain a dual edge which intersects $S$. 
\end{lem}
\begin{proof}
Lemma~\ref{lem-primal-dual-prod} shows that a.s.\ $\mcl E_0$ is a finite set, and this together with translation invariance modulo scaling (e.g., using the condition given in Lemma~\ref{lem-dyadic-resample}) shows that a.s.\ $\#\mcl E_x < \infty$ for each $x \in\BB R$. 
Therefore, a.s.\ for each $x \in\BB R$ there exists a square $S$ of $\mcl D$ containing $x$ such that $\bigcup_{e\in \mcl E_x} R_e \subset S$. 
We claim that the condition in the statement of the lemma holds for this choice of $S$. 
Indeed, the width of each rectangle $R_e$ is the same as the width of $e^*$ and the height of $R_e$ is the same as the height of $e$, so if $e\in\mcl E(\BB R)$ such that $e^*$ intersects the vertical line through $x$, then $x\in R_e$ and hence $R_e\subset S$. Therefore every path of edges $e^*$ for $e\in\mcl E(\BB R)$ which intersects this vertical line must pass through $S$.  
\end{proof}

For $x\in\BB R$, let $\rng S_x$ be the smallest square containing $x$ which satisfies the condition of Lemma~\ref{lem-rectangle-sep} (chosen by some arbitrary deterministic convention in the case when the square is not unique). Note that $\rng S_x$ necessarily intersects one of the edges $\frk e^*(j)$ for $j\in\BB Z$. In fact, every sub-path of $\{\frk e^*(j)\}_{j\in\BB Z}$ which crosses between the vertical lines containing the left and right boundaries of $\rng S_x$ must include an edge which intersects $\rng S_x$. 
We also note that $\rng S_x = \rng S_y$ whenever $y\in \rng S_x\cap \BB R$ and $y$ is not one of the two endpoints of $\rng S_x \cap \BB R$. We write
\eqb \label{eqn-rngS-def}
\rng{\mcl S} := \left\{ \rng S_x : x\in \BB R\right\} .
\eqe
We will now show that none of the squares $\rng S_x$ are macroscopic.

\begin{lem} \label{lem-max-good-square}
Almost surely,
\eqb \label{eqn-max-good-square}
\lim_{k\rta\infty} |S_k|^{-1} \max_{x\in S_k\cap \BB R} |\rng S_x| = 0 .
\eqe 
\end{lem}
\begin{proof}
The proof is similar to that of Lemma~\ref{lem-max-cell-diam} but since we are working with squares (so our sets cannot be ``long and skinny") no moment bound is required.
The square $\rng S_0$ a.s.\ has finite side length, so a.s.\ we can find a random $\mcl F_\infty^{\BB R}$-measurable $m_\ep > 0$ such that $\BB P\left[ |\rng S_0| > \ep |\wh S_{m_\ep} | \,|\, \mcl F_\infty^{\BB R} \right] \leq \ep^2$. 
Applying Lemma~\ref{lem-ergodic-avg-line} to the function $F = \BB 1_{(|\rng S_0| > \ep |\wh S_{m_\ep}|)}$ shows that a.s.\ 
\eqb \label{eqn-max-good-square-area}
\limsup_{m \rta\infty} \frac{1}{|\wh S_m| } \int_{\wh S_m \cap \BB R} \BB 1_{( |\rng S_x| > \ep |\wh S_{m } | )} \,dx \leq \ep^2 , 
\eqe 
where here we use that if $m\geq m_\ep$, then $\wh S_{m_\ep}^x \subset \wh S_m$ for each $x\in \wh S_m$ and so $\{|\rng S_x| > \ep |\wh S_{m } | \} \Rightarrow \{|\rng S_x| > \ep |\wh S_{m_\ep}^x| \} $. 
For Lebesgue-a.e.\ $x,y \in \wh S_m\cap \BB R$, we have $y \in \rng S_x \Rightarrow \rng S_{x} = \rng S_y$.
Hence if there is an $x\in \wh S_m\cap \BB R$ with $\rng S_x \cap \BB R \subset \wh S_m$ and $|\rng S_x| > \ep |\wh S_m|$, then the integral in~\eqref{eqn-max-good-square-area} is at least $\ep$. From this and the fact that a.s.\ each $S_k$ is one of the $\wh S_m$'s, we get that a.s.\ for large enough $k\in\BB N$,
\eqb \label{eqn-max-good-square-interior}
\max\left\{ |\rng S_x| : x\in  S_k \cap \BB R , \, \rng S_x \cap\BB R \subset  S_k \right\} \leq \ep |S_k| .
\eqe

Using~\eqref{eqn-max-good-square-interior} and Lemma~\ref{lem-more-squares}, one obtains~\eqref{eqn-max-good-square} via the same argument used at the end of the proof of Lemma~\ref{lem-max-cell-diam} to deal with the squares which intersect the endpoints of the interval $S_k\cap\BB R$. 
\end{proof}

Lemma~\ref{lem-max-good-square} ensures that our path $\{\frk e^*(j)\}_{j\in\BB Z}$ gets close to each point of $S_k\cap \BB R$, 
but we still need to rule out the possibility of large excursions between the times when it hits the $\rng S_x$'s. 
 
Recalling the set of squares $\rng{\mcl S}$ from~\eqref{eqn-rngS-def}, let $\{J_n\}_{n\in\BB Z}$ be the ordered bi-infinite sequence of times $j\in\BB Z$ for which $\frk e^*(j) \in \rng S$ for some $\rng S\in \rng{\mcl S}$, enumerated so that $0 \in [ J_0  ,  J_1 ]$. Let
\eqb \label{eqn-path-excursion-def}
\mcl A := \left\{ \bigcup_{j=J_{n-1}+1}^{J_n} \frk e^*(j)  : n\in\BB Z\right\} 
\eqe 
be the set of excursions of $j\mapsto \frk e^*(j)$ away from the union of the squares in $\rng{\mcl S}$. For $e \in \mcl E\mcl M$ with $e^* \in \{\frk e^*(j)\}_{j\in\BB Z}$, let $A(e) \in \mcl A$ be the unique such excursion with $e^* \in A(e)$. 

\begin{lem} \label{lem-max-line-diam}
Almost surely,
\eqb \label{eqn-max-line-diam}
\lim_{k\rta\infty} |S_k|^{-1} \max\left\{ \op{diam}\left( A  \right) : A \in \mcl A ,\, A \cap S_k\not=\emptyset   \right\}  =0 .
\eqe 
\end{lem}
\begin{proof}
The proof is similar to that of Lemma~\ref{lem-max-cell-diam}. 
By Lemma~\ref{lem-primal-dual-prod}, a.s.\ 
\eqb
\BB E\left[   \sum_{e \in \mcl E_0}  \frac{\op{diam}(e ) \op{diam}(e^*)}{ \op{Area}(R_{e } )}   \,|\, \mcl F_\infty^{\BB R} \right] < \infty  
\eqe
and a.s.\ $\max_{e\in \mcl E_0} \op{diam}(A(e)) < \infty$. 
So, for each $\ep > 0$, we can a.s.\ find a random $m_\ep > 0$ such that $\BB E[ D_{m_\ep}(0) \,|\, \mcl F_\infty^{\BB R}]  \leq \ep^2$, where
\eqbn
D_m(x) = D_m(x,\ep) := \sum_{e \in \mcl E_x}  \frac{\op{diam}(e ) \op{diam}(e^*)}{ \op{Area}(R_{e } )} \BB 1_{\left( \op{diam}(A(e)) \geq \ep |\wh S_{m }|  \right)} .
\eqen
Applying Lemma~\ref{lem-ergodic-avg-line} with $F = D_{m_\ep}(0)$ shows that a.s.\ 
\eqb \label{eqn-line-diam-int}
\limsup_{m \rta\infty} \frac{1}{|\wh S_m| } \int_{\wh S_m \cap \BB R}  D_m(x)     \, dx \leq \ep^2 , 
\eqe 
where here we use that if $m\geq m_\ep$, then $\wh S_{m_\ep}^x \subset \wh S_m$ for each $x \in \wh S_m$ and so for $e\in \mcl E_x$, $\{ \op{diam}(A(e)) > \ep |\wh S_{m } | \} \Rightarrow \{ \op{diam}(A(e)) > \ep |\wh S_{m_\ep}^x| \} $. 

By breaking up the integral~\eqref{eqn-line-diam-int} as a sum of integrals over $R_e \cap \BB R$ for $e\in \mcl E\mcl M$ with $R_e\cap \wh S_m\cap  \BB R \not=\emptyset$ and recalling that each $S_k$ is one of the $\wh S_m$'s, we get that a.s.\  
\eqb \label{eqn-line-diam-sum}
\limsup_{k\rta\infty} \frac{1}{|S_k| } \sum_{ \substack{ e\in \mcl E\mcl M \\ R_e\cap S_k \cap  \BB R \not=\emptyset}     } \frac{\op{diam}(e) \op{Len}(R_e\cap S_k\cap \BB R)  \op{diam}(e^*)}{\op{Area}(R_e)} \BB 1_{\left(  \op{diam}(A(e))  > \ep |S_k| \right)}  \leq \ep^2 , 
\eqe  
where here $\op{Len}$ denotes one-dimensional Lebesgue measure.
If $R_e\cap  \BB R \subset S_k$, then $\op{Len}( R_e\cap S_k\cap \BB R)$ is the width of $R_e$. Since the height of $R_e$ is at most $\op{diam}(e)$, in this case the corresponding summand in~\eqref{eqn-line-diam-sum} is at least $\op{diam}(e^*) \BB 1_{\left(  \op{diam}(A(e))  > \ep |S_k| \right)}$.
If $A\in \mcl A$ with $A\subset S_k$, then $R_e\cap \BB R\subset S_k$ for each $e\in\mcl E\mcl M$ with $A(e) = A$ (this is because the width of $R_e$ is the same as the width of $e^*$). If also $\op{diam}(A) > \ep |S_k|$, then the preceding discussion shows that the sum~\eqref{eqn-line-diam-sum} is at least 
\eqbn
\sum_{  e\in\mcl E\mcl M : e\subset A    } \op{diam}(e^*) \geq \op{diam}(A) \geq \ep |S_k| .
\eqen
Consequently,~\eqref{eqn-line-diam-sum} implies that a.s.\ for large enough $k\in\BB N$,  
\eqb \label{eqn-line-diam-sum-interior}
\op{diam}(A) \leq \ep |S_k| ,\quad \forall A\in\mcl A \quad \text{with} \quad A\subset S_k .
\eqe 
We now conclude the proof by applying Lemma~\ref{lem-more-squares} as at the end of the proof of Lemma~\ref{lem-max-cell-diam}. 
\end{proof}

\begin{lem} \label{lem-dual-path} 
Almost surely, for each fixed $\ep \in (0,1)$ and each large enough $k\in\BB N$ there is a path of edges $e^* \in \mcl E\mcl M^*$ with $\op{diam}(e^*) \leq \ep |S_k|$ from the left boundary of $S_k $ to the right boundary of $S_k $ which is contained in the $\ep|S_k|$-neighborhood of $S_k  \cap \BB R$.
\end{lem}

For the proof of Lemma~\ref{lem-dual-path}, we will need a slight extension of Lemma~\ref{lem-max-line-diam}. For an excursion set $A = \bigcup_{j=J_{n-1}+1}^{J_n} \frk e^*(j) \in \mcl A$, we set $A' := \bigcup_{j=J_{n-1}  }^{J_n  -1} \frk e^*(j) $, so that the \emph{first} edge of $A'$, rather than the \emph{last} edge, intersects a square in $\rng{\mcl S}$. The same argument as in the proof of Lemma~\ref{lem-max-line-diam} shows that a.s.\ $\lim_{k\rta\infty} \max\left\{ \op{diam}\left( A ' \right) : A \in \mcl A ,\, A' \cap S_k\not=\emptyset   \right\}  =0$. 

\begin{proof}[Proof of Lemma~\ref{lem-dual-path}]
For $k\in\BB N$, let $\mcl J_k^R$ be the smallest $j\in\BB Z$ for which $\frk e^*(j)$ intersects the vertical line through the right boundary of $S_k$ and let $\mcl J_k^L$ be the last time $j$ before $\mcl J_k^R$ for which $\frk e^*(j)$ intersects the vertical line through the left boundary of $S_k$. Then $\frk e^*|_{[\mcl J_k^L , \mcl J_k^R]_{\BB Z}}$ is a path of edges of $\mcl E\mcl M^*$ between these two horizontal lines and each edge of this path except for $\frk e^*(\mcl J_k^L)$ and $\frk e^*(\mcl J_k^R)$ lies strictly between these two horizontal lines.

By Lemma~\ref{lem-max-good-square}, a.s.\ for each large enough $k\in\BB N$, each of the squares $\rng S_x$ for $x\in S_k\cap\BB R$ is properly contained in $S_k$. 
By the definition of $\rng S_x$ (i.e., the property described in Lemma~\ref{lem-rectangle-sep}) and the preceding paragraph, the path $\frk e^*|_{[\mcl J_k^L , \mcl J_k^R]_{\BB Z}}$ must include an edge which intersects each such square $\rng S_x$, and no edge of $\frk e^*|_{[\mcl J_k^L +1 , \mcl J_k^R-1]_{\BB Z}}$ can intersect a square of the form $\rng S_y$ for $y \in \BB R\setminus S_k$. 
From this and the definition~\eqref{eqn-path-excursion-def} of $\mcl A$, it follows that a.s.\ for large enough $k\in\BB N$, 
\eqb \label{eqn-dual-path-include}
\bigcup_{j=\mcl J_k^L}^{\mcl J_k^R} \frk e^*(j) \subset \BB A_k ,
\eqe
where $\BB A_k$ is the union of all of the of sets of the form $A$ or $A'$ for $A\in\mcl A$ which intersect $\rng S_x$ for some $x\in S_k\cap\BB R$. 
Note that we need to allow for sets of the form $A'$ since the element of $\mcl A$ containing $\frk e^*(\mcl J_k^R)$ need not intersect $\rng S_x$ for any $x\in S_k\cap\BB R$.

By Lemma~\ref{lem-max-line-diam} and the extension of Lemma~\ref{lem-max-line-diam} mentioned just above, a.s.\ the maximum diameter of the sets $A$ or $A'$ in the union defining $\BB A_k$ is of order $o_k(|S_k|)$ as $k\rta\infty$. 
In particular, by combining this with~\eqref{eqn-dual-path-include} we get that a.s.\ the maximal diameter of the edges of $\frk e^*|_{[\mcl J_k^L,\mcl J_k^R]_{\BB Z}}$ is $o_k(|S_k|)$. Since each set in the union defining $\BB A_k$ intersects $\rng S_x$ for some $x \in S_k\cap\BB R$ by definition, Lemma~\ref{lem-max-good-square} implies that a.s.\ $|S_k|^{-1} \max_{z\in \BB A_k} \op{dist}(z,S_k\cap\BB R) \rta 0$ as $k\rta\infty$. It therefore follows from~\eqref{eqn-dual-path-include} that a.s.\ for large enough $k\in\BB N$, each edge of $\frk e^*|_{[\mcl J_k^L,\mcl J_k^R]_{\BB Z}}$ is contained in the $\ep|S_k|$-neighborhood of $S_k\cap\BB R$, as required. 
\end{proof}

\begin{proof}[Proof of Lemma~\ref{lem-dual-max-diam}]
By Lemma~\ref{lem-dyadic-resample} and Lemma~\ref{lem-dual-path}, if $z_k$ is sampled uniformly from $S_k$, then a.s.\ the conditional probability given $(\mcl M,\mcl M^*,\mcl D)$ that there is a path of edges $e^* \in \mcl E\mcl M^*$ with $\op{diam}(e^*) \leq \ep |S_k|$ from the left boundary of $S_k $ to the right boundary of $S_k $ which is contained in the $(\ep/100) |S_k|$-neighborhood of $S_k \cap (\BB R  +z_k)$ tends to 1 as $k\rta\infty$. A symmetric argument shows that the same is true with ``top" and ``bottom" in place of ``left" and ``right" and $\BB R + z_k$ replaced by the vertical line through $z_k$.

Applying this to a large, $\ep$-dependent number $N$ of uniformly random points sampled from $S_k$, we find that it is a.s.\ the case that for large enough $k\in\BB N$, we can find a finite collection $E_k^*$ of edges of $\mcl E\mcl M^*$ such that each $e^* \in E_k^*$ has diameter at most $\ep$ and each connected component of $S_k \setminus \bigcup_{e^* \in E_k^*} e^*$ has diameter at most $\ep$ (see Lemma~\ref{lem-var-lim-infty} for a similar argument). Since no two dual edges can cross, a.s.\ for large enough $k\in\BB N$ the maximal diameter of the edges of $\mcl E\mcl M^*$ which are contained in $S_k$ and lie at Euclidean distance at least $\ep$ from $\bdy S_k$ is at most $\ep$. Applying Lemma~\ref{lem-more-squares} to deal with the edges near the boundary (via exactly the same argument as at the end of the proof of Lemma~\ref{lem-max-cell-diam}) shows that a.s.\ the maximal diameter of the edges of $\mcl E\mcl M^*$ which intersect $S_k$ is of order $o_k(|S_k|)$ $k\rta\infty$. A symmetric argument shows that the same is true of the edges which of $\mcl E\mcl M$ which intersect $S_k$. Using Lemma~\ref{lem-more-squares} again to transfer from the squares $S_k$ to the disks $B_r(0)$ concludes the proof.  
\end{proof}

\bibliography{cibiblong,cibib}

\bibliographystyle{hmralphaabbrv}

\end{document}